\documentclass[12pt,a4paper]{amsart}
\usepackage[utf8]{inputenc}
\usepackage[T1]{fontenc}
\usepackage{amsmath}
\usepackage{amsthm}
\usepackage{amsfonts}
\usepackage{amssymb}
\usepackage{mathrsfs}
\usepackage{url}
\usepackage{mathdots}
\usepackage[english]{babel}
\usepackage{geometry}
\usepackage{verbatim}
\usepackage{hyperref}
\usepackage{float}
\usepackage{enumerate}
\usepackage{tikz-cd}
\usepackage[nolabel]{showlabels}
\usepackage{extarrows}
\usepackage{comment}
\usepackage{bbm}

\hypersetup{
    bookmarks=true,         
    unicode=false,          
    pdftoolbar=true,        
    pdfmenubar=true,        
    pdffitwindow=false,     
    pdfstartview={FitH},    
		pdftitle={Motivic weight {$23$}},    
    pdfauthor={Ga\"etan Chenevier and Olivier Ta\"ibi},     
    colorlinks=true,       
    linkcolor=blue,          
    citecolor=green,        
    filecolor=green,      
    urlcolor=cyan}           

\newcommand{\Q}{\mathbb{Q}}
\newcommand{\A}{\mathbb{A}}
\newcommand{\R}{\mathbb{R}}
\providecommand{\C}{\mathbb{C}}
\renewcommand{\C}{\mathbb{C}}

\newcommand{\Qp}{\mathbb{Q}_p}

\newcommand{\Qbar}{\overline{\mathbb{Q}}}

\newcommand{\Zp}{\mathbb{Z}_p}

\newcommand{\Z}{\mathbb{Z}}
\newcommand{\rmL}{\mathrm{L}}

\newcommand{\nG}{\mathrm{n}_G}

\newcommand{\GSpin}{\mathrm{GSpin}}

\newcommand{\vol}{\operatorname{vol}}

\newcommand{\Spin}{\mathrm{Spin}}

\newcommand{\tr}{\operatorname{tr}}
\newcommand{\sn}{\operatorname{sn}}
\newcommand{\Sym}{\operatorname{Sym}}

\newcommand{\disc}{\mathrm{disc}}

\newcommand{\GL}{\mathrm{GL}}
\newcommand{\PGL}{\mathrm{PGL}}
\newcommand{\SL}{\mathrm{SL}}
\newcommand{\Sp}{\mathrm{Sp}}
\newcommand{\SO}{\mathrm{SO}}
\renewcommand{\O}{\mathrm{O}}

\newcommand{\Ind}{\mathrm{Ind}}

\newcommand{\Ccal}{\mathcal{C}}

\newcommand{\Pcal}{\mathcal{P}}

\newcommand{\Zhat}{\widehat{\mathbb{Z}}}

\newcommand{\ps}{\par \smallskip}

\newcommand{\Pialg}{\Pi_{\mathrm{alg}}}

\newtheorem{theo}{Theorem}[section]

\newtheorem{lemm}[theo]{Lemma}

\newtheorem{coro}[theo]{Corollary}
\newtheorem{corointro}{Corollary}
\newtheorem{defi}[theo]{Definition}
\newtheorem{prop}[theo]{Proposition}
\newtheorem{theointro}{Theorem}
\newtheorem{problem}{Problem}
\newtheorem{question}[theo]{Question}
\theoremstyle{definition}
\newtheorem{rema}[theo]{Remark}
\newtheorem{exam}[theo]{Example}

\numberwithin{equation}{subsection}

\newenvironment{pf}
{\medskip\noindent {\it Proof.  }}
{\hfill\nobreak $\Box$ \par\bigbreak}

\begin{document}

\author{Ga\"etan Chenevier}
\address[Ga\"etan Chenevier]{CNRS, Universit\'e Paris-Sud}
\author{Olivier Ta\"ibi}
\address[Olivier Ta\"ibi]{CNRS, \'Ecole Normale Sup\'erieure de Lyon}

\title[Level $1$ algebraic cusp forms]
{Discrete series multiplicities for classical groups over $\Z$ and level $1$ algebraic cusp forms}

\thanks{Ga\"etan Chenevier and Olivier Ta\"ibi are supported by the C.N.R.S. and by the project ANR-14-CE25.}

\maketitle

\begin{abstract} The aim of this paper is twofold.
First, we introduce a new method for evaluating the multiplicity of a given discrete series
in the space of level $1$ automorphic forms of a split classical group $G$ over $\Z$,
and provide numerical applications in absolute rank $\leq 8$.
Second, we prove a classification result for
the level one cuspidal algebraic automorphic representations of ${\rm GL}_n$ over $\Q$ ($n$ arbitrary)
whose motivic weight is $\leq 24$.\ps
In both cases,
a key ingredient is a classical method based on the Weil explicit formula,
which allows to disprove the existence of certain level one algebraic cusp forms on ${\rm GL}_n$,
and that we push further on in this paper.
We use these vanishing results to obtain an arguably ``effortless'' computation
of the elliptic part of the geometric side of the trace formula of $G$,
for an appropriate test function. \ps
Thoses results have consequences for the computation
of the dimension of the spaces of (possibly vector-valued)
Siegel modular cuspforms for ${\rm Sp}_{2g}(\Z)$:
we recover all the previously known cases
without relying on any, and go further, by a unified and  ``effortless'' method.
\end{abstract}

\newcommand{\uk}{\underline{k}}

\newpage 

{\small
\tableofcontents
}
\section{Introduction}

\subsection{Siegel modular forms for ${\rm Sp}_{2g}(\Z)$} \label{introsmf}

We denote by ${\rm S}_k(\Gamma_g)$ and  ${\rm S}_{\uk}(\Gamma_g)$ respectively
the space of cuspidal Siegel modular forms for the full Siegel modular group
$\Gamma_g={\rm Sp}_{2g}(\Z)$, which are either scalar-valued of weight $k \in
\Z$, or more generally vector-valued of weight $\underline{k} =
(k_1,k_2,\dots,k_g)$ in $\Z^g$ with $k_1\geq k_2 \geq \cdots \geq
k_g$ (we refer to \cite{vanderGeer_Siegel} for a general introduction to Siegel
modular forms).
Recall that ${\rm S}_{\uk}(\Gamma_g)$ trivially vanishes when $\sum_i k_i$ is
odd, and also for $k_g<g/2$ (Freitag, Reznikoff, Weissauer). \ps
The question of determining the dimension of ${\rm S}_{\uk}(\Gamma_g)$, very
classical for $g=1$, has a long and rich history for $g>1$. It has first been
attacked for $g=2$ using geometric methods, in which case concrete formulas were
obtained by Igusa (1962) in the scalar-valued case, and by Tsushima (1983) for
the weights\footnote{More precisely, Tsushima could only prove that his formula
works for $k_2 \geq 5$, and later Petersen (2013) and Ta\"ibi (2016)
independently showed that it holds as well for $k_2\geq 3$ and $(k_1,k_2) \neq
(3,3)$, as conjectured by Ibukiyama.} $k_1 \geq k_2 \geq 3$. There is still no
known formula for $k_2=2$, although we have ${\rm S}_{\uk}(\Gamma_2)=0$ for
$k_2=1$ (Ibukiyama, Skoruppa): see \cite{cleryvandergeer} for a discussion of these singular
cases. An analogue of Igusa's result for $g=3$ was proved by Tsuyumine in 1986,
but only quite recently a conjectural formula was proposed by  Bergstr\"om,
Faber and van der Geer, in the vector valued case $k_1\geq k_2 \geq k_3 \geq 4$,
based on counting genus three curves over finite fields (2011). Their formula,
and more generally a formula\footnote{These are pretty huge formulas, which
can't be printed here already for $g>2$, but see Theorem A {\it loc. cit.} for
their general shape. } for $\dim {\rm S}_{\uk}(\Gamma_g)$ for arbitrary $g\leq
7$ and $k_g>g$ was proved by the second author in \cite{Taibi_dimtrace}, using a
method that we will recall in \S\,\ref{method}.
Actually, the general formulas given in \cite{Taibi_dimtrace} apply to any genus
$g$ and any weights with $k_g>g$.
However, they involve certain rational numbers, that we shall refer to later as
{\it masses}, that are rather difficult to compute; Ta\"ibi provided loc. cit. a
number of algorithms to determine them (more precisely, certain local orbital integrals, see
\S \,\ref{method}) that allowed him to numerically compute those masses for
$g\leq 7$.

Our first main result in this paper is a completely different and comparatively
much easier method to compute the aforementioned masses. This method allows us
to recover, in a uniform and rather ``effortless'' way, all the computations of
masses done in \cite{Taibi_dimtrace} for $g\leq 7$, and even to go further:

\begin{theointro}
\label{thmintro1}
There is an explicit and implemented formula
computing $\dim \, {\rm S}_{\underline{k}}(\Gamma_g)$
for any $g \leq 8$, and any $\uk$ with $k_g > g$.
\end{theointro}

\noindent See Theorems \ref{theomasseseffortless} \& \ref{theomassesmixed}
for equivalent, better formulated, statements.
Theorem \ref{thmintro1} is about Siegel modular forms of arbitrary weights $\underline{k}$
such that $k_g > g$, but with genus $g \leq 8$.
A second result concerns the Siegel modular forms of arbitrary genus,
but of weights $\leq 13$ (there are really finitely many relevant pairs $(\uk,g)$ here).
It is very much in the spirit of the determination of $\dim {\rm S}_k(\Gamma_g)$ by Chenevier-Lannes in \cite{CheLan}
in the cases $g \leq k \leq 12$. \ps

\begin{theointro} \label{thmintro2}
The dimension of ${\rm S}_{\underline{k}}(\Gamma_g)$
for $13 \geq k_1 \geq \dots \geq k_g >g$,
and $\underline{k}$ non scalar,
is given by Table \ref{tab:tableleq13nonscal}.
The dimension of ${\rm S}_k(\Gamma_g)$
for any $k \leq 13$ and any $g\geq 1$
is given by Table \ref{tab:tableleq13scal}.
\end{theointro}

The notations for these tables are explained in \S \ref{pfthm2}.
Table \ref{tab:tableleq13nonscal} shows in particular that
${\rm S}_{\underline{k}}(\Gamma_g)$ has dimension $\leq 1$ for $\uk$ non scalar and $13 \geq k_1 \geq \dots \geq k_g >g$,
and is nonzero for exactly $29$ values of $\uk$.
Table \ref{tab:tableleq13scal} includes the fact that ${\rm S}_k(\Gamma_g)$
vanishes whenever $k \leq 13$ and $g \geq k$, except in the three following situations:
$$\dim {\rm S}_{12}(\Gamma_{12})\,=\,\dim {\rm S}_{13}(\Gamma_{16})\, =\, \dim {\rm S}_{13}(\Gamma_{24})=1$$
(a nonzero element in the first and last spaces has been constructed in \cite{BoFrWe} and
\cite{Freitag_harm_theta}). We obtain for instance the following result.

\begin{corointro}
\label{corweight13}
${\rm S}_{13}(\Gamma_g)$ has dimension $1$ for $g=8,12,16,24$,
and $0$ otherwise.
\end{corointro}

 An inspection of standard ${\rm L}$-functions,
 and general results of B\"ocherer, Kudla-Rallis and Weissauer,
show that these four spaces are spanned by certain Siegel theta series
build on Niemeier lattices (see \S \ref{complconstruction}).
In a companion paper \cite{CheTaiLeech}
we come back to these constructions and study them in a much more elementary way.
Combined with \cite[Sect. 9.5]{CheLan},
{\it this provides an explicit construction of all the cuspidal Siegel modular eigenforms
of weight $k \leq 13$ and level $\Gamma_g$ for an arbitrary genus $g$}.
In \S \ref{complconstruction}, we also prove that
{\it Eichler's basis problem holds in weights $k=8$ and $12$ for arbitrary genus $g$},
completing the results of \cite{CheLan} for $g \leq k$. \ps

Last but not least, let us mention that in the past,
several other authors have computed $\dim {\rm S}_k(\Gamma_g)$
for a number of isolated and small pairs $(g,k)$,
sometimes with much efforts, e.g.
Igusa, Witt, Erokhin, Duke-Immamoglu,
\footnote{We warn the reader that the proofs of Duke and Immamoglu in \cite{DukeImamoglu}
are not valid in the case $g>k$ since they rely on the incorrect Corollary 3 p. 601 in \cite{mizumoto}:
see \cite{mizumoto_erratum} for a recent erratum.
Note also that a preliminary version of \cite{CheLan} did include a proof
of the vanishing of ${\rm S}_k(\Gamma_g)$ for $k\leq 12$, $g>k$ and $g\neq 24$,
but this statement was deleted in the published version for the same reason.
Our proofs here show that all these incriminated results for $g>k$ were nevertheless correct,
and actually do not rely anymore on the results in \cite{mizumoto}.}
Nebe-Venkov, and Poor-Yuen among others.
We would like to stress that
{\it none of the results of this paper depend on
a previous computation of the dimension of a space of Siegel modular forms},
not even of ${\rm S}_k({\rm SL}_2(\Z))$ ! see \S \ref{sub:intro_effortless}.
Moreover, as far as we know,
the dimensions given by Theorems \ref{thmintro1} and \ref{thmintro2} 
seem to recover all the previously known\footnote{More precisely, the only cases not covered by these two theorems
seem to be the vanishing of ${\rm S}_{\uk}(\Gamma_2)$ for $k_2 = 1$, and for the pairs $\uk=(k_1,2)$ with $k_1 \leq 50$ \cite[Thm. 1.3]{cleryvandergeer}.
However, this vanishing for $k_2=1$ can also be proved by arguments in the spirit of the ones employed here, as explained by the first author in an appendix of \cite{cleryvandergeer}, and we can actually prove it as well for all $\uk=(k_1,2)$ with $k_1 \leq 54$: see \S \ref{webeatcleryvandergeer}. } dimensions of spaces of
Siegel modular cuspforms for $\Gamma_g$. 

Our proof of Theorems \ref{thmintro1} and \ref{thmintro2} will use automorphic
methods, building on a strategy developed in the recent works
\cite{ChRe,CheLan,Taibi_dimtrace}:
we will review this strategy in \S \ref{sub:intro_effortless} and \S \ref{secsmf}.
An important ingredient is Arthur's {\it endoscopic classification} of the
discrete automorphic spectrum of classical groups in terms of general linear
groups \cite{Arthur_book}, including the so-called {\it multiplicity formula}.
A special feature of this approach is that even if we were only interested in
$\dim {\rm S}_{\uk}(\Gamma_g)$, we would be forced to compute first the
dimension of various spaces of automorphic forms for all the split classical
groups over $\Z$ of smaller dimension.
By a {\it split classical group over $\Z$} we will mean here either the group scheme
${\rm Sp}_{2g}$ over $\Z$, or the special orthogonal group scheme ${\rm SO}_n$
of the integral quadratic form $\sum_{i=1}^{n/2} x_i x_{n+1-i}$ (case $n$
even $\neq 2$) or $\sum_{i=1}^{(n-1)/2} x_i x_{n+1-i} + x_{(n+1)/2}^2$ (case $n$ odd)
over $\Z^n$.
An important gain of this approach,
however, is that in the end we do not only compute $\dim {\rm S}_{\uk}(\Gamma_g)$,
but also the dimension of its subspace of cuspforms of any possible endoscopic type,
a quantity which is arguably more interesting than the whole dimension itself:
see Tables \ref{tab:tableleq13nonscal} \& \ref{tab:tableleq13scal} for a sample of results. \ps

\subsection{Level one algebraic cusp forms on ${\rm GL}_m$} \label{level1alg}
Let $m \geq 1$ be an integer
and $\pi$ a cuspidal automorphic representation of ${\rm PGL}_m$ over $\Q$.
We say that $\pi$ has {\it level $1$} if $\pi_p$ is unramified for each prime $p$.
We say that $\pi$ is {\it algebraic} if the infinitesimal character of $\pi_\infty$,
that we may view following Harish-Chandra and Langlands
as a semi-simple conjugacy class in ${\rm M}_m(\C)$,
has its eigenvalues in $\frac{1}{2}\Z$,
say $w_1 \geq w_2 \geq \dots \geq w_m$,
and with $w_i-w_j \in \Z$.
Those eigenvalues $w_i$ are called {\it the weights} of $\pi$,
and the important integer $w(\pi):=2w_1$ is called
the {\it motivic weight} of $\pi$.
The Jacquet-Shalika estimates imply
$w_{m+1-i}=-w_i$ for all $i$,
and in particular, $w(\pi)\geq 0$ (see \S \ref{realalgebraic}). \ps

The algebraic cuspidal $\pi$ are especially interesting to number theorists,
as for such a $\pi$ standard conjectures (by Clozel, Langlands)
predict the existence of a compatible system of pure and irreducible
$\ell$-adic Galois representations $\rho$
with same ${\rm L}$-function as $\pi |.|^{w_1}$,
the Hodge-Tate weights of $\rho$ being the $w_i+w_1$,
and its Deligne weight being $w(\pi)$.
The level $1$ assumption in this work has to be thought as a
simplifying, but still interesting, one (see \cite{CheLan} for several
motivations).\ps

An important problem is thus the following.
For an integer $m \geq 1$,
we denote by ${\rm W}_m$ the set of $\underline{w}= (w_i)$ in $\frac{1}{2}\Z^m$
with $w_1 \geq w_2 \geq \dots \geq w_m$,
$w_i-w_j \in \Z$ and $w_i=-w_{m+1-i}$ for all $1\leq i,j\leq m$.\ps

\begin{problem}
For $\underline{w} \in {\rm W}_m$,
determine the (finite) number ${\rm N}(\underline{w})$ of
level $1$ cuspidal algebraic automorphic representations of ${\rm PGL}_m$
whose weights are the $w_i$,
and the number ${\rm N}^\bot(\underline{w})$
of those $\pi$ satisfying furthermore $\pi^\vee \simeq \pi$ (self-duality).
\end{problem}\ps

Let us say that an element $(w_i)$ in ${\rm W}_m$ is {\it regular}
if for all $i\neq j$ we have either $w_i \neq w_j$, 
or $m \equiv 0 \bmod 4$, $i=j-1=m/2$ and $w_i=w_j=0$ (hence $w_1 \in \Z$).
Despite appearences, the question of determining the ${\rm N}^\bot(\underline{w})$ for regular $\underline{w}$
is very close to that discussed in \S\,\ref{introsmf}.
Indeed, as was observed and used in \cite{ChRe, CheLan, Taibi_dimtrace},
the level $1$ self-dual $\pi$ of regular weights are the exact building blocks
for Arthur's endoscopic classification of the discrete automorphic
representations of split classical groups over $\Z$ which are unramified at all
finite places and discrete series at the Archimedean place (with a very
concrete form of Arthur's multiplicity formula, relying on \cite{AMR}).
As an illustration of this slogan, the following fact was observed in \cite[Chap
9]{ChRe}.\ps

{\bf Key fact 1}. {\it Fix $g\geq 1$ and $\uk=(k_1,\dots,k_g) \in \Z^g$ with $k_1\geq k_2 \geq \dots \geq k_g>g$.
Then the dimension of ${\rm S}_{\uk}(\Gamma_g)$ is
an explicit function of the (finitely many) quantities ${\rm N}^\bot(\underline{w})$
with $\underline{w}=(w_i) \in {\rm W}_m$ regular, $m \leq 2g+1$ and $w_1 \leq k_1-1$.}\ps
\ps
See \cite[Prop.\ 1.11]{ChRe} for explicit formulas for $g \leq 3$, \cite[Chap.
5]{Taibi_dimtrace} for $g = 4$, and \cite[Thm. 5.2]{CheLan} for the general
recipe. This general recipe will actually be recalled in \S \ref{secsmf},
in which we will also apply the recent results of
\cite{MR_scalar} to give an analoguous statement for ${\rm S}_k(\Gamma_g)$
(scalar-valued case) in the case $k \leq g$.
This last case is quite more sophisticated,
in particular it also involves certain slightly irregular weights.
We will come back later on the relations
between the Problem above and Theorems  \ref{thmintro1} and \ref{thmintro2}. \ps

\subsection{Classification and inexistence results}\label{classresult}
Let us denote by $\Pi_{\rm alg}$ the set of cuspidal automorphic 
representations of ${\rm PGL}_m$, with $m\geq 1$ arbitrary, which are algebraic and of 
level $1$. The second main result of this paper is a partial
classification of elements $\pi$ in $\Pi_{\rm alg}$ having motivic weight $\leq 24$.
The
first statement of this type, proved in \cite[Thm. F]{CheLan}, asserts that
there are exactly $11$ elements $\pi$ in $\Pi_{\rm alg}$ of motivic weight $\leq 22$: the trivial
representation of ${\rm PGL}_1$, the $5$ representations $\Delta_{k-1}$ of ${\rm
PGL}_2$ generated by the $1$-dimensional spaces ${\rm S}_k({\rm SL}_2(\Z))$ for
$k=12,16,18,20,22$ (whose weights are $\pm \frac{k-1}{2}$), the Gelbart-Jacquet
symmetric square of $\Delta_{11}$ (with weights $-11,0,11$), and four other
$4$-dimensional self-dual $\pi$ with respective weights
$$\{\pm\,19/2,\,\pm \,7/2\}, \, \, \{\pm \,21/2,\,\pm \,5/2\}, \, \, \, \{\pm\, 21/2, \pm \,9/2\}\, \, \, {\rm and}\, \, \,  \{\pm \,21/2,\,\pm\,13/2\}.$$
\noindent In this paper we significantly simplify the proof of \cite[Thm. F]{CheLan}: see \S \ref{par:poids22}.
More importantly, these simplifications allow us to prove the following theorems
for motivic weights $23$ and $24$.\ps

\begin{theointro}\label{thm23} There are exactly $13$ level $1$
cuspidal algebraic automorphic representations of ${\rm GL}_m$ over $\Q$,
with $m$ varying, with motivic weight $23$,
and having the weight $23$ with multiplicity $1$:\ps
\begin{itemize}
\item[(i)] $2$ representations of ${\rm PGL}_2$
generated by the eigenforms in ${\rm S}_{24}({\rm SL}_2(\Z))$, \ps
\item[(ii)] $3$ representations of ${\rm PGL}_4$
of weights $\pm \,23/2,\,\pm \,v/2$ with $v=7,9$ or $13$,\ps
\item[(iii)] $7$ representations of ${\rm PGL}_6$
of weights $\pm\, 23/2,\, \pm \,v/2,\, \pm \,u/2$
with $$(v,u)\,=\,(13,5),\, (15,3),\, (15,7), \,(17,5),\,(17,9),\,(19,3) \, \, \, \, {\rm and}\, \, \, (19,11),$$ \ps
\item[(iv)] $1$ representation of ${\rm PGL}_{10}$
of weights $\pm \,23/2,\, \pm \,21/2,\, \pm\,17/2,\,\pm\, 11/2,\,\pm\, 3/2$.\ps
\end{itemize}
There are all self-dual (symplectic) and uniquely determined by their weights.
\end{theointro}

The representations in (ii) and (iii) above were first discovered in
\cite{ChRe}, using some conditional argument that was later made unconditional
in \cite{TaiMult}.
Their existence was also confirmed by the different computation of the second
author in \cite{Taibi_dimtrace}, who also discovered the $10$-dimensional form in (iv). \ps

Despite our efforts, we have not been able to classify the $\pi$ of motivic
weight $23$ such that multiplicity of the weight $23$ is $>1$.
We could only prove that there is an explicit list $\mathcal{L}$ of $182$
weights $\underline{w}=(w_i)$ with $w_1=w_2=23$ such that (a) the weight of any
such $\pi$ belongs to $\mathcal{L}$, (b) for any $\underline{w}$ in
$\mathcal{L}$ there is at most one $\pi$ with this weight
(necessarily self-dual symplectic), except for the single weight $\underline{w}
=(v_i/2)$ in ${\rm W}_{14}\cap \mathcal{L}$ with
$(v_1,v_2,\dots,v_7)=(23,23,21,17,13,7,1)$, for which there might also be two
such $\pi$ which are the dual of each other: see Proposition \ref{prop:23_mult}. 
Nevertheless, we conjecture that all of those putative $183$ representations do not exist, 
except perhaps one. Indeed, we prove in \S \ref{par:grh} the following result, 
assuming a suitable form of {\rm (GRH)}.\ps

\begin{theointro}\label{thm23GRH} Assume ${\rm (GRH)}$ and that there exists $\pi$ in $\Pi_{\rm alg}$
with motivic weight $23$ and having the weight $23$ with multiplicity $>1$. Then we have  $m=16$, the weights of $\pi$ are
 $\pm 1/2$, $\pm 7/2$, $\pm 11/2$, $\pm 15/2$, $\pm 19/2$, $\pm 21/2$ 
 as well as $\pm 23/2$ with multiplicity $2$, and $\pi$ is the unique element of $\Pi_{\rm alg}$ having these $16$ weights.
\end{theointro} 
\ps

We now state our (weaker) classification result in motivic weight $24$. \ps\ps

\begin{theointro}\label{thm24} There are exactly $3$ level $1$ algebraic
cuspidal self-dual automorphic representations of ${\rm PGL}_m$ over $\Q$,
with $m$ varying, with motivic weight $24$ and a regular weight.
They are self-dual (orthogonal) with respective set of weights
$$ \{\pm 12, \pm 8, \pm 4, 0\}, \{\pm 12, \pm 9, \pm 5, \pm 2\} \, \, \, {\rm and}\, \, \,
\{\pm 12, \pm 10, \pm 7, \pm 1\}.$$
\end{theointro}

Again, those three forms were first discovered in \cite[Cor.\ 1.10 \&
1.12]{ChRe} and confirmed in \cite{Taibi_dimtrace}.
Interestingly, as explained in \cite{ChRe}, we expect that their Sato-Tate
groups are respectively the compact groups ${\rm G}_2$, ${\rm Spin}(7)$ and
${\rm SO}(8)$. See \cite[Thm. 6.12]{CheG2} for a proof that the first form,
which is $7$-dimensional, has ${\rm G}_2$-valued $\ell$-adic Galois representations. \ps

{\bf Proofs.}
Our proofs of Theorems \ref{thm23} and \ref{thm24} are in the same spirit of the one
of \cite[Thm. F]{CheLan}.
As already said, all the representations appearing in the theorem were already
known to exist by the works \cite{ChRe, TaiMult, Taibi_dimtrace} 
(and we will give another proof of their existence in \S \ref{par:effortless}), so the main
problem is to show that there are no others.
The basic idea that we will use for doing so is to consider an hypothetical
$\pi$, consider an associate ${\rm L}$-function of $\pi$, and show that this
function cannot exist by applying to it the so-called explicit formula for
suitable test functions.
This is a classical method by now, inspired by the pioneering works of 
Stark, Odlyzko and Serre on discriminant bounds of number fields.  
It was developed by Mestre in \cite{Mestre} and applied {\it loc. cit.} to the standard ${\rm L}$-function of $\pi$ (see
also \cite{Fermigier}), then by Miller \cite{Miller} to the Rankin-Selberg
${\rm L}$-function of $\pi$, and developed further more recently in  \cite[\S 9.3.4]{CheLan}.\ps

Two important novelties were discovered in \cite{CheLan} in order to obtain the
aforementioned classification result in motivic weight $\leq 22$.
The first one, developed further in \cite{Chenevier_HM}, is a finiteness result
which implies that there are only finitely many level $1$ cuspidal algebraic
automorphic representations $\pi$ of ${\rm PGL}_m$, with $m$ varying, of motivic
weight $\leq 23$.
This finiteness is also valid in motivic weight $\leq 24$ assuming a suitable
form of GRH.
This result is effective and produces a finite but large list of possible
weights for those $\pi$ (for instance, it leads to $12295$ possible weights in
motivic weight $23$).
The hardest part is then to eliminate most of those remaining weights.
The second novelty found in \cite{CheLan} was the observation that we obtain
efficient constraints by applying as well the explicit formula to all the ${\rm
L}$-functions ${\rm L}(s,\pi \times \pi_i),$ where the $\pi_i$ are the known
representations. See \cite[Scholia 9.3.26 \& 9.3.32]{CheLan} for the two useful criteria obtained
there using this idea. \ps

In this paper, we discovered a criterion 
that may be seen as a natural generalisation of \cite[Scholia 9.3.26]{CheLan},
and that happened to be (in practice, and quite surprisingly) much more
efficient than the aforementioned ones. Moreover, contrary to \cite[Scholia 9.3.32]{CheLan}, 
we do not need to know any Satake parameter for the known elements of $\Pi_{\rm alg}$ 
(which allow us to use test functions with aribtrary supports). Our basic idea here is to apply 
the explicit formula to the Rankin-Selberg ${\rm
L}$-function of all linear combinations $t_1 \pi_1
\oplus \cdots \oplus t_s \pi_s$ where $\pi_1$ is unknown of given weights, the
$\pi_i$ with $i>1$ are known (in the sense that they exist and we know their weights), 
and the $t_i$ are arbitrary nonnegative {\it real} numbers.
More precisely, for any test function $F$ we associate a certain 
symmetric bilinear form ${\rm C}^F$ on the free vector space $\R \,\Pi_{\rm alg}$ 
over $\Pi_{\rm alg}$, which represents the {\it computable}
part of the explicit formula for the test function $F$. 
Assuming a certain positivity assumption on $F$, the quadratic form
$x \mapsto {\rm C}^F(x,x)$ is then $\geq 0$ on the cone of $\R\, \Pi_{\rm alg}$ generated by $\Pi_{\rm alg}$: see Proposition \ref{prop:basicinequality}.
In order to reach a contradiction we have thus to show that at least one quadratic form ${\rm C}^F$ takes a negative value 
on the cone generated by $\pi_1,\dots,\pi_s$.
See \S \ref{par:applineqalgo} for the description of the minimisation algorithm that we have used for this purpose,
as well as the homepage \cite{homepage} for related sources.
One charm of this method is that although it requires some computational work to find a concrete element $x$ 
of that cone and a test function $F$ leading to a contradiction (and all is fair for that!), once we have found it it is 
quick and easy to rigorously check that we have ${\rm C}^F(x,x)<0$: see \S \ref{numericaleval}.

\subsection{The effortless computation of masses}
\label{sub:intro_effortless}

Fix $G$ a rank $n$ split classical group over $\Z$ in the sense of \S \ref{introsmf}. In other words, $G$ belongs to one of the families
\[ (\SO_{2n+1})_{n \geq 1},\ \ (\Sp_{2n})_{n \geq 1} \ \ \text{and} \ \
(\SO_{2n})_{n \geq 2} \]
Assume that $G(\R)$ has discrete
series, {\it i.e.} that $G$ is not isomorphic to ${\rm
SO}_{2n}$ with $n$ odd, and fix $K$ a maximal compact subgroup of $G(\R)$.
There is an analogue of Key fact $1$ with ${\rm S}_{\uk}(\Gamma_g)$ replaced by
the multiplicity of any discrete series of $G(\R)$ in the space
$\mathcal{A}^2(G)$ of $K$-finite square-integrable automorphic forms over
$G(\Z)\backslash G(\R)$.
In this paper, as in \cite{Taibi_dimtrace}, we will use a variant of this
statement involving rather certain Euler-Poincar\'e characteristics. Our first aim is to state this variant (Key fact 2 below).
For any dominant
weight $\lambda$ of $G(\C)$, we denote by ${\rm V}_\lambda$ an irreducible
representation of $G(\C)$ with highest weight $\lambda$ and consider 
$${\rm EP}(G; \lambda) \,= \,\sum_{i \geq 0} (-1)^i \dim {\rm H}^i(\mathfrak{g},K; \mathcal{A}^2(G) \otimes {\rm V}_\lambda^\vee) \,\,\,\in \,\,\Z,$$
where ${\rm H}^\ast(\mathfrak{g},K;-)$ denotes $(\mathfrak{g},K)$-cohomology.
Attached to $G$ is a certain integer denoted ${\rm n}_{\widehat{G}}$, defined as the dimension of the
standard representation of $\widehat{G}$, the Langlands dual group of $G$: concretely, we have ${\rm
n}_{\widehat{G}}={2n+1}$ for $G={\rm Sp}_{2n}$, and ${\rm n}_{\widehat{G}}=2n$ for $G={\rm SO}_{2n+1}$
or ${\rm SO}_{2n}$. The infinitesimal character of ${\rm V}_\lambda$, namely ``$\lambda + \rho$'',
defines a unique regular element $w(\lambda)$ in ${\rm W}_m$ with $m={\rm n}_{\widehat{G}}$. Concretely, 
using the standard\footnote{The conditions on the $\lambda_i$ are the following:  $\lambda_i \in \Z$ for each $i$, $\lambda_1 \geq \dots \geq \lambda_n$, and either $\lambda_n\geq 0$ (cases $G={\rm Sp}_{2n}$ or ${\rm SO}_{2n+1}$) or $|\lambda_n| \leq \lambda_{n-1}$ (case $G={\rm SO}_{2n}$ with $n\geq 2$).} notation  $\lambda = \sum_{i=1}^n \lambda_i e_i$  for dominant weights of classical groups (as in \cite[\S
2]{Taibi_dimtrace}), $w(\lambda)$ is explicitely given by the following formulas:
$$ \label{eq:w_lambda}
  w(\lambda)_i = \begin{cases}
    \lambda_i+n+1/2-i & \text{ for } 1 \leq i \leq n \text{ if } G = \SO_{2n+1},
    \\
    \lambda_i+n+1-i & \text{ for } 1 \leq i \leq n \text{ if } G = \Sp_{2n}, \\
    |\lambda_i|+n-i & \text{ for } 1 \leq i \leq n \text{ if } G = \SO_{2n}\, \, {\rm and}\, n \equiv 0 \bmod 2. \\
  \end{cases}
$$

\noindent The promised second key fact,
explained in \S\S 4.1 and 4.2 of \cite{Taibi_dimtrace} is:\ps\ps

{\bf Key fact 2}. {\it Fix $G$ and $\lambda$ as above,
and set $\underline{w}=w(\lambda)=(w_i)$.
Assume we know ${\rm N}^\bot(\underline{v})$
for all regular $\underline{v}=(v_i) \in {\rm W}_m$ with $m<{\rm n}_{\widehat{G}}$ and $v_1 \leq  w_1$,
then it is equivalent to know ${\rm EP}(G;\lambda)$ and ${\rm N}^\bot(\underline{w})$.}\ps\ps

It follows from the formula above for $w(\lambda)$ that any regular $\underline{w}$ in ${\rm W}_m$, with $m\geq 1$ arbitrary, is of the form $w(\lambda)$ for a
unique split classical group $G$ over $\Z$ and some dominant weight $\lambda$ of $G$.
As a consequence, Key fact 2 paves the way for a computation of ${\rm
N}^\bot(\underline{w})$ for all regular $w$, by induction on ${\rm n}_{\widehat{G}}$.

Contrary to the case of Key fact 1, we will not reproduce here the precise claimed relation 
between ${\rm EP}(G; \lambda)$ and the quantities ${\rm N}^\bot(-)$ 
stated in Key fact 2,  and simply refer to \cite[\S 4]{Taibi_dimtrace}. 
As for Key fact 1, it crucially depends on Arthur's endoscopic
classification \cite{Arthur_book} and on the explicit description of Archimedean
Arthur packets in the regular algebraic case 
given by \cite{AdJo} and \cite{AMR} (they are the so-called Adams-Johnson packets, see \cite[\S 4.2.2]{Taibi_dimtrace}).
The full combinatorics, when written down explicitly and case-by-case, are rather complicated but computable, and implemented since
\cite{Taibi_dimtrace}.\ps\ps

\label{method}
The second ingredient is Arthur's {\it ${\rm L}^2$-Lefschetz trace formula} developed
in \cite{ArthurL2} applied to the relevant test function. 
A detailed analysis of this formula has been made in
\cite{Taibi_dimtrace} that we will follow below. We use this version of the trace
formula as it is the one with the simplest geometric terms. Its spectral side is 
exactly ${\rm EP}(G;\lambda)$, a quantity in principle harder to interpret (the price to pay for a simple geometric side), 
but this can be done precisely by Key fact 2 above.

We fix a Haar measure $dg= \prod'_v dg_v$ on $G(\A)$,
with $\A$ the ad\`ele ring of $\Q$,
such that the Haar measure $dg_p$ on $G(\Q_p)$ gives $G(\Z_p)$ the volume $1$.
Fix a dominant weight $\lambda$ of $G$.
We apply Arthur's formula \cite{ArthurITF} to a test function $\varphi$ of the
form $\varphi_\infty  \otimes'_p {\rm 1}_{G(\Z_p)}$, where ${\rm 1}_{G(\Z_p)}$
is the characteristic function of $G(\Z_p)$, and where $\varphi_\infty(g)
dg_\infty$ is the sum, over all the discrete series $\delta$ of $G(\R)$ with
same infinitesimal character as $V_\lambda$, of a pseudocoefficient of $\delta$.
According to \cite{ArthurL2} the resulting identity, which only depends on $G$
and $\lambda$, is
\begin{equation}\label{traceformula} {\rm EP}(G; \lambda) \,=\, {\rm T}_{\rm geom}(G;\lambda).\end{equation}
The geometric side ${\rm T}_{\rm geom}(G;\lambda)$ is a finite sum of terms
indexed by Levi subgroups of $G$ of the form
\begin{equation}\label{levisab}{\rm GL}_1^a \times {\rm GL}_2^b \times G'\end{equation}
with $G'$ a split classical groups over $\Z$ and $a,b \geq 0$.
The main term, corresponding to $G$ itself and called the {\it elliptic term}, is
\begin{equation} \label{eq:tell}
  \mathrm{T}_\mathrm{ell}(G; \lambda) \,=\, \sum_{\gamma}\, \vol(G_\gamma(\Q)
  \backslash G_\gamma(\A))\, \cdot\, \O_\gamma(1_{G(\widehat{\Z})})\,\cdot \,
  \tr(\gamma\,|\, \mathrm{V}_\lambda),
\end{equation}
where $\gamma$ runs over representatives of the (finitely many)
$G(\Q)$-conjugacy classes of finite order elements in $G(\Q)$
whose $G(\Q_p)$-conjugacy class meets $G(\Z_p)$ for each prime $p$ (see \S \ref{elementaryconjclass}).
For each such $\gamma$, we have denoted by $G_\gamma$ its centralizer in $G$ (defined over $\Q$),
choosed on $G_\gamma(\A)$ a signed Haar measure $dh=\prod'_v dh_v$
with $dh_\infty$ an Euler-Poincar\'e measure on $G_\gamma(\R)$ (in the sense of \cite{EPSerre}),
and we have denoted by ${\rm O}_\gamma({\rm 1}_{\rm G(\widehat{\Z})})$
the product over all primes $p$ of the classical orbital integrals
\begin{equation} \label{eq:locintorg} \int_{G(\Q_p)/G_\gamma(\Q_p)} \,{\rm 1}_{G(\Z_p)}\,(g_p \gamma g_p^{-1}) \,\frac{dg_p}{dh_p}.\end{equation}
Ta\"ibi developped in \cite{Taibi_dimtrace} a number of algorithms to enumerate
the $\gamma$, compute their local orbital integrals (with $dh_p$ Gross's
canonical measure), and the associated global volumes.
Here we shall simply write
\begin{equation} \label{eq:tellmasse}
  \mathrm{T}_\mathrm{ell}(G; \lambda) \,=\, \sum_{ c \in  \mathcal{C}(G)} \, \,
  \mathrm{m}_c \, \,  \tr(c\,|\, \mathrm{V}_\lambda),
\end{equation}
where $\mathcal{C}(G)$ denotes the set of
$G(\overline{\Q})$-conjugacy classes of finite order elements in $G(\Q)$
(essentially, a characteristic polynomial: see \S\ref{elementaryconjclass})
and ${\rm m}_c$ is a certain number depending only on $c$
and called {\it the mass} of $c$
(note that $\tr(c\, |\,\mathrm{V}_\lambda)$ is well-defined).
By definition, ${\rm m}_c$ is a concrete sum of volumes times global integral
orbitals; it essentially follows from \cite[Theorem 9.9]{GrossMot} and Siegel's
theorem on the rationality of the values of Artin L-functions at non-negative
integers \cite{SiegelLfunc} that we have $\mathrm{m}_c \in \Q$.

The character of ${\rm V}_\lambda$ may be either evaluated using the (degenerate) Weyl
character formula as in \cite{ChRe}, or much more efficiently for small
$\lambda$ using Koike-Terrada's formulas \cite{Koike_Terada} as was observed in \cite{chenevier_niemeier}.
Last but not least, the term in the sum defining ${\rm T}_{\rm geom}(G;
\lambda)$ corresponding to a proper Levi subgroup of the form \eqref{levisab}
is expressed in terms of ${\rm T}_{\rm ell}(G';\lambda')$, as well as ${\rm T}_{\rm
ell}({\rm SO}_3;\lambda'')$ if $b \neq 0$, for suitable auxilliary
$\lambda',\lambda''$: see \cite[\S 3.3.4]{Taibi_dimtrace} for the concrete
formulas, that will not be repeated here.
As a consequence, the key problem is to be able to compute the masses ${\rm
m}_c$ for $c \in \mathcal{C}(G)$. \ps

{\bf The strategy.}
We are finally able to explain our strategy.
Fix $m\geq 1$ an integer.
We may assume, by induction, that we have computed the masses ${\rm m}_c$ for
all split classical groups $H$ over $\Z$ such that $H(\R)$ has discrete series
and ${\rm n}_{\widehat{H}} <m$, and all $c \in \mathcal{C}(H)$.
By the Key fact 2 and the trace formula \eqref{traceformula}, note that we have
an explicit and computable formula for ${\rm N}^\bot(\underline{w})$ for all
regular $w \in {\rm W}_{m'}$ with $m'<m$.
Fix a split classical group $G$ over $\Z$ such that $G(\R)$ has discrete
series and ${\rm n}_{\widehat{G}}=m$.
Assume we have found a finite set $\Lambda$ of dominant weights of $G(\C)$ with
the following two properties:\ps

(P1) For all $\lambda \in \Lambda$ we have ${\rm N}^\bot(w(\lambda))=0$.\ps

(P2) the $\Lambda \times \mathcal{C}(G)$ matrix $(\tr( c \, | \, {\rm
V}_\lambda))_{(\lambda,c)}$ has rank $|\mathcal{C}(G)|$. \ps

Then from (P1) and the Key fact 2 we know
${\rm EP}(G;\lambda)$ for all $\lambda \in \Lambda$.
It follows that for all $\lambda \in \Lambda$ we know ${\rm T}_{\rm geom}(G;\lambda)$ as well, by the trace formula \eqref{traceformula},
hence also ${\rm T}_{\rm ell}(G;\lambda)$ since the non-elliptic geometric terms are also known by induction.
By \eqref{eq:tellmasse} and (P2),
we deduce the masses ${\rm m}_c$ for all $c \in \mathcal{C}(G)$ by solving a linear system.
As a consequence, for an arbitrary dominant weight $\lambda$ of $G$ we may then compute ${\rm T}_{\rm geom}(G;\lambda)$,
hence ${\rm EP}(G;\lambda)$ by \eqref{traceformula}, and ${\rm N}^\bot(w(\lambda))$ by Key fact 2.\ps\ps

Amusingly, the conclusion is that we end by proving the existence of self-dual cusp forms for ${\rm PGL}_m$ mostly by showing that 
many others do not exist ! (namely the ones with weights of the form $w(\lambda)$ with $\lambda$ in $\Lambda$.)\ps\ps
\ps

 {\bf A simple example}.
 Let us illustrate this method in the (admittedly too) simple case $m=2$ and $G={\rm SO}_3$.
 In this case $\mathcal{C}(G)$ has $5$ classes, of respective order $1,2,3,4$ and $6$:
 say $c_i$ has order $i$.
 The dominant weights of $G$ are of the form $k e_1$ for a unique integer $k\geq 0$, and will be simply denoted by $k$.
 For $k\geq 0$, we have $w(k)=(k+1/2,-k-1/2)$, ${\rm N}(w(k))=\dim {\rm S}_{2k+2}({\rm SL}_2(\Z))$,  
 the analysis of the spectral side gives
 $${\rm EP}(G; k ) = - {\rm N}(w(k)) + \delta_{k,0},$$
 with $\delta_{i,j}$ the Kronecker symbol, and the geometric side is
 $${\rm T}_{\rm geom}(G; k) = {\rm T}_{\rm ell}(G; k) + \frac{1}{2}.$$
Assume we know that there is no cuspidal modular form for ${\rm SL}_2(\Z)$ of usual weight $2,4,6,8,10$.
This may for instance be shown by applying the explicit formula to the Hecke ${\rm L}$-function of a putative eigenform of such a weight,
as observed in \cite[Rem. III.1]{Mestre}. This also follows very easily from the methods of \S \ref{classresult}.
Using $\dim {\rm V}_k=2k+1$ and the identity $\tr(c_i | {\rm V}_{k}) = {\rm
sin}(\frac{(2k+1)\pi}{i})/{\rm sin}\frac{2\pi}{i}$ for $i>1$ we obtain with
$\Lambda=\{0,1,2,3,4\}$ the linear system
{\tiny
$$\left[ \begin{array}{ccccc} 1 & 1 & 1 & 1 & 1 \\ 3 & -1 & 0 & 1 & 2 \\ 5 & 1 &
-1 & -1 & 1 \\ 7 & -1 & 1 & -1 & -1 \\ 9 & 1 & 0 & 1 & -2 \end{array}\right].
\left[\begin{array}{c} \mathrm{m}_{c_1} \\ \mathrm{m}_{c_2} \\ \mathrm{m}_{c_3}
\\ \mathrm{m}_{c_4} \\ \mathrm{m}_{c_6} \end{array} \right] \,=
\,\left[\begin{array}{c} 1/2 \\ -1/2 \\ -1/2 \\ -1/2 \\ -1/2
\end{array}\right].$$}\par

\noindent Luckily, the matrix on the left-hand side is invertible:
we find ${\rm m}_{c_1}=-\frac{1}{12}$, ${\rm m}_{c_2}=\frac{1}{4}$, ${\rm m}_{c_3}=\frac{1}{3}$ and ${\rm m}_{c_4}={\rm m}_{c_6}=0$.
As a consequence, we recover the classical formula for $\dim {\rm S}_{2k}({\rm SL}_{2}(\Z))$.\ps

\begin{rema} \label{rem:spin_norm_intro}
  In certain classes $c$ in $\mathcal{C}(G)$, there is no element $\gamma$ whose
  $G(\Q)$-conjugacy class meets $G(\Z_p)$ for each $p$ and such that
  $G_\gamma(\R)$ has discrete series: this forces ${\rm m}_c=0$ by \eqref{eq:tell}.
  This actually explains ${\rm m}_{c_4}={\rm m}_{c_6}=0$ above.
\end{rema}

  This remark will lead us in \S \ref{sec:spinornorm}, and following \cite[Remark 3.2.8]{Taibi_dimtrace}, to replace
  $\mathcal{C}(G)$ by a smaller set $\mathcal{C}_1(G)$, and
  to rather apply our strategy with $\mathcal{C}(G)$ replaced by $\mathcal{C}_1(G)$ in (P2). 
  See also \S \ref{elementaryconjclass} for other more elementary reductions, using the center of $G$ or an outer automorphism of ${\rm SO}_{2n}$.\ps
  
  The crucial last ingredient for this method to work is to be able to find sufficiently many $w \in {\rm W}_{{\rm n}_{\widehat{G}}}$ such that ${\rm N}^\bot(w)=0$.
  We will use of course for this the method explained in \S \ref{classresult} (the explicit formula for Rankin-Selberg ${\rm L}$-functions).
  Rather miraculously, it provides enough vanishing results up to rather a high rank: see \S \ref{nonexisteffortless}
   for a proof of the following final theorem, obtained by applying only this ``effortless'' strategy:\ps

\begin{theointro}\label{theomasseseffortless}
(``Effortless'' computation)
Assume $G={\rm SO}_n$ with $n \leq 17$, or $G={\rm Sp}_{2n}$ with $2n\leq 14$.
Then the masses ${\rm m}_c$, for all $c \in \mathcal{C}(G)$, are given in \cite{homepage}.
\end{theointro}

As already said, these results are in accordance with all the orbital integral
computations done in \cite{Taibi_dimtrace}. This method is both conceptually simpler and faster:
for comparison, computing all the orbital integrals for $\Sp_{14}$ takes several
weeks, whereas finding $\Lambda$ and solving the linear system to determine 
all the ${\rm m}_c$, for $c \in \mathcal{C}({\rm Sp}_{14}),$ only takes
minutes on the same computer. The cases ${\rm SO}_{m}$ with $m=15,16,17$ are new. Last but not least, if we combine the 
methods of this paper with Ta\"ibi's, we obtain the following new result in the symplectic case (see \S \ref{nonexisteffortless}).

\begin{theointro} \label{theomassesmixed}
  The masses ${\rm m}_c$, for all $c \in \mathcal{C}({\rm Sp}_{16})$, are given
  in \cite{homepage}.
\end{theointro}

Following Key fact 2, these two theorems allow to compute ${\rm EP}(G; \lambda)$ for 
all those $G$ and an arbitrary weight $\lambda$, as well as the quantity 
${\rm N}^\bot(\underline{w})$ for any $w \in {\rm W}_m$ for $m \leq 16$:
see {\it loc. cit.} for tables.

\section{Weil's explicit formula for Rankin-Selberg $L$-functions: a refined
  positivity criterion} \label{explicitformula}

\subsection{ Algebraic Harish-Chandra modules and automorphic representations.}
\label{realalgebraic}

Let $\pi$ be a cuspidal automorphic representation of ${\rm PGL}_n$ over $\Q$.
As recalled in \S \ref{level1alg}, we say that $\pi$ has level $1$ if  $\pi_p$ is unramified for every prime $p$.
We also explained {\it loc. cit.} what it means for $\pi$ to be algebraic.
It will be useful to have an alternative point of view on this last condition
in terms of the {\it Langlands correspondence} (see e.g. \cite[\S 8.2.12]{CheLan}).
Recall that $\Pi_{\rm alg}$ denotes
the set of level $1$ algebraic cuspidal automorphic representations of ${\rm PGL}_n$ (with $n\geq 1$ varying). \ps\ps

We denote by ${\rm W}_\R$ the Weil group of $\R$: we have ${\rm W}_\R = \C^\times
\,\coprod \,j \C^\times$, where $j^2$ is the element $-1$ of $\C^\times$ and
with $jzj^{-1}=\overline{z}$ for all $z \in \C^\times$. The Langlands
correspondence for ${\rm GL}_n(\R)$ is a natural bijection $V \mapsto {\rm
L}(V)$ between the set of isomorphism classes of irreducible admissible
Harish-Chandra modules for ${\rm GL}_n(\R)$, and the set of isomorphism classes
of $n$-dimensional (complex, continuous and semi-simple) representations of ${\rm
W}_\R$ \cite{knappmotives}.
We say that the Harish-Chandra module $V$ is {\it algebraic} if every element in
the center $\R^\times$ of ${\rm W}_\R$ acts as a homothety with factor $\pm {\rm
1}$ in ${\rm L}(V)$.
In particular, ${\rm L}(V)$ factors
through the (compact) quotient of ${\rm W}_\R$ by $\R_{>0}$, which is an
extension of $\Z/2$ by the unit circle. We denote by $1$ the trivial
representation of ${\rm W}_\R$, by $\varepsilon_{\C/\R}$ its unique order $2$
character, and for $w \in \Z$ we set ${\rm I}_w \,=\, \Ind_{\C^\times}^{{\rm
W}_\R}\, \eta^w$ where $\eta(z) = z/|z|$. Up to isomorphism, the irreducible
representations of ${\rm W}_\R$ trivial on $\R_{>0}$ are
$$1, \, \, \, \, \varepsilon_{\C/\R}, \, \, \,\,{\rm and}\, \,\,\, {\rm I}_w \,
\,\,{\rm for}\, \, w >0.$$

\noindent We also have ${\rm I}_0 \simeq 1 \oplus \varepsilon_{\C/\R}$, and ${\rm I}_{w} \simeq {\rm I}_{w'}$ if and only if $w = \pm w'$. \ps\ps
 If $\pi$ is a {\it cuspidal} automorphic representation of ${\rm PGL}_n$ over $\Q$, Clozel's purity lemma (or the Archimedean Jacquet-Shalika estimates) shows that the Harish-Chandra module $\pi_\infty$ is algebraic in the sense above if, and only if, $\pi$ is algebraic in the sense of \S \ref{level1alg} \cite[Prop. 8.2.13]{CheLan}. Moreover, for all $v \in \Z$ the multiplicity of the weight $v/2$ of $\pi$ is the same as the multiplicity of the character $\eta^v$ in the restriction of ${\rm L}(\pi_\infty)$ to $\C^\times$. In other words, all the weights $0$ (resp. $\pm v/2$ with $v>0$) of $\pi_\infty$ are explained by occurences of $1$ or $\varepsilon_{\C/\R}$ (resp. of ${\rm I}_v$) in ${\rm L}(\pi_\infty)$. It will be convenient to introduce: \ps\ps

\begin{itemize}
  \item the Grothendieck ring ${\rm K}_\infty$ of complex, continuous, finite
    dimensional, representations of ${\rm W}_\R$ which are trivial on its
    central subgroup $\R_{>0}$, \ps\ps
  \item for $w$ an integer, the sugroup ${\rm K}_\infty^{\leq w}$ of ${\rm
    K}_\infty$ generated by the ${\rm I}_v$ with $0 \leq v \leq w$ and $v \equiv
    w \bmod 2$, and also by $1$ and $\varepsilon_{\C/\R}$ in the case $w$ is
    even. \ps\ps
\end{itemize}

\noindent An element $U$ of ${\rm K}_\infty$ will be called {\it effective} if it is the class of a representation of ${\rm W}_\R$. The {\it weights} 
of such a $U$ are the $w/2$ in $\frac{1}{2}\Z$ such that $\eta^w$ occurs in $U_{|\C^\times}$, counted with the multiplicity of $\eta^w$ in $U_{|\C^\times}$ (a nonnegative number as $U$ is effective). \ps\ps

 It follows from this discussion that for $\pi$ in $\Pi_{\rm alg}$,
we have ${\rm L}(\pi_\infty) \in {\rm K}_\infty^{\leq w}$ for some integer $w\geq 0$,
and that the smallest such integer $w$ coincides with the motivic weight $w(\pi)$ of $\pi$ introduced in \S \ref{level1alg}.
Moreover, the weights of $\pi$ are that of ${\rm L}(\pi_\infty)$. \ps\ps

\ps\ps

For later use, we now recall Langlands' definition
for the $\varepsilon$-factors and $\Gamma$-factors of algebraic Harish-Chandra modules.
The $\Gamma$-factor of an element $U$ in ${\rm K}_\infty$
is a meromorphic function $s \mapsto \Gamma(s,U)$ on the whole complex plane,
characterized by the additivity property
$\Gamma(s,U \oplus U')=\Gamma(s,U)\Gamma(s,U')$
and the following axioms \cite{tate}:
{\small $$\Gamma(s,1)=\pi^{-\frac{s}{2}}\Gamma(\frac{s}{2})\, \,\, \, {\rm and} \, \, \, \, \, \Gamma(s,{\rm I}_{w}) = 2 (2\pi)^{-s-\frac{|w|}{2}}\Gamma(s+\frac{|w|}{2})\, \, \, {\rm for\,\, all}\,\,\, w \in \Z,$$}
\par \noindent in which $s \mapsto \Gamma(s)$ is the classical gamma function.
Similarly,\footnote{What we denote by $\varepsilon(U)$ here is
what Tate denotes $\varepsilon(U,\psi,dx)$ in \cite[(3.4)]{tate},
with choice of additive character $\psi(x)=e^{2i\pi x}$,
and with $dx$ the standard Lebesgue measure on $\R$.}
the $\varepsilon$-factor of $U \in {\rm K}_\infty$ is
the element $\varepsilon(U)$ of $\{\pm 1, \pm i\}$
characterized by the additivity property $\varepsilon(U \oplus U')=\varepsilon(U)\varepsilon(U')$
and the identities $\varepsilon(1)=1$ and $\varepsilon({\rm I}_{w})=i^{w+1}$ for every integer $w\geq 0$.\ps\ps

\subsection{Regular and self-dual elements of $\Pi_{\rm alg}$} \label{regselfpar}

Let $\pi$ be a level $1$ algebraic cuspidal automorphic representation of ${\rm
PGL}_n$ over $\Q$.
We will say that $\pi$ is {\it regular} if the representation ${\rm L}(\pi_\infty)$ of ${\rm W}_\R$ is multiplicity free.
It is thus equivalent to say that for each weight $w$ of $\pi$, either $w$ has multiplicity $1$ or 
we have $w=0$ and ${\rm L}(\pi_\infty)$ contains both $1$ and $\varepsilon_{\C/\R}$ with multiplicity $1$.
This latter case can only occur of course if both the motivic weight of $\pi$ and $n$ are even.
Moreover, we observe that:

\begin{itemize}
\item 
if all the nonzero weights of $\pi$ have multiplicity $1$,
and if the weight $0$ has multiplicity $2$,
then $\pi$ is regular if, and only if, we have $n \equiv 0 \bmod 4$. \ps\ps
\item $\pi$ is regular if, and only if, the vector $(w_i) \in \frac{1}{2} \Z^n$,
where $w_1\geq w_2 \geq \dots \geq w_n$ are the weights of $\pi$,
is regular in the sense of \S \ref{level1alg}. \ps \ps
\item if $\pi$ is regular and $n=2g+1$ is odd,
then ${\rm L}(\pi_\infty)$ contains $\varepsilon_{\C/\R}^g$.
\end{itemize}

\noindent Indeed, as $\pi$ a trivial central character,
we must have ${\rm det} \,{\rm L}(\pi_\infty)=1$,
and we conclude by the formula
$\det\, {\rm I}_v\, =\, \varepsilon_{\C/\R}^{v+1}$ for $v \in \Z$. \ps\ps

Assume now that $\pi$ is {\it self-dual}, that is, isomorphic to its contragredient (or {\it dual}) $\pi^\vee$.
Then $\pi$ is either {\it symplectic} or {\it orthogonal} in the sense of Arthur \cite[Thm. 1.4.1]{Arthur_book}.
 Moreover, if $\pi$ is symplectic (resp. orthogonal) then ${\rm L}(\pi_\infty)$ preserves
 a nondegenerate alternating (resp. symmetric) pairing  \cite[Thm. 1.4.2]{Arthur_book}.
 In particular, if $\pi \in \Pi_{\rm alg}$ is self-dual, and if some weight of $\pi$ has multiplicity $1$ (e.g. if $\pi$ is regular),
 then $\pi$ is symplectic if and only if $w(\pi)$ is odd \cite[Prop. 8.3.3]{CheLan}.

\subsection{The explicit formula for ${\rm L}$-functions of pairs of elements in $\Pi_{\rm alg}$} \label{explicitformulaforpairs}

Let $\pi$ and $\pi'$ be level $1$ cuspidal automorphic representations
of $\PGL_n$ and $\PGL_m$ respectively.
For $p$ prime we denote by ${\rm c}_p(\pi)$ the semi-simple conjugacy class in
$\SL_n(\C)$ associated with the unramified representation $\pi_p$, following
Langlands, under the Satake isomorphism.
The Rankin-Selberg ${\rm L}$-function of $\pi$ and $\pi'$ is the Euler product
$${\rm L}(s,\pi \times \pi') \, = \, \prod_p \, {\rm det} (1- p^{-s} {\rm c}_p(\pi) \otimes {\rm c}_p(\pi'))^{-1}.$$
By fundamental works of Jacquet, Piatetski-Shapiro, and Shalika
\cite{JacquetShalika, JPSS_RS_conv}, this Euler product is absolutely convergent
for ${\rm Re}\, s>1$, and the {\it completed} ${\rm L}$-function
\begin{equation} \label{completers} \Lambda(s,\pi \times \pi')\, = \, \Gamma(s, {\rm L}(\pi_\infty) \otimes {\rm
L}(\pi'_\infty))\, {\rm L}(s,\pi \times\pi'),\end{equation}
has a meromorphic continuation to $\C$ and a functional equation of the form
\begin{equation} \label{eq:functeq}
  \Lambda(s,\pi \times \pi') \, =\, \epsilon(\pi \times \pi') \, \Lambda(1-s,
  \pi^\vee \times (\pi')^\vee)
\end{equation}
where $\epsilon(\pi \times \pi')$ is a certain nonzero complex number
(it does not depend on $s$ as $\pi$ has level $1$). We set $\epsilon(\pi) = \epsilon(\pi \times 1)$.

Assuming $\pi$ and $\pi'$ are algebraic, the only case of interest here, the
$\Gamma(s,-)$ factor in \eqref{completers} is given by the recipe recalled in \S \ref{realalgebraic},
and we simply have
\begin{equation} \label{eqn:epsilon=epsilon_infty}
  \epsilon(\pi \times \pi') = \varepsilon({\rm L}(\pi_\infty) \otimes {\rm
L}(\pi'_\infty)). \end{equation}
Note that the ring structure of ${\rm K}_\infty$ is determined by the relations
${\rm I}_w \cdot {\rm I}_{w'}  =  {\rm I}_{|w+w'|} + {\rm I}_{|w-w'|}$ and
$\varepsilon_{\C/\R} \cdot {\rm I}_w = {\rm I}_w$.  \ps\ps

By Moeglin and Waldspurger \cite[App.]{MoeWal_residuel_GL}, $\Lambda(s,\pi
\times \pi')$ is entire unless we have $\pi' \simeq \pi^\vee$, in which case the
only poles are simple and at $s=0,1$.
Moreover, $\Lambda(s,\pi \times \pi')$ is
bounded in vertical strips away from its poles by Gelbart and Shahidi
\cite{GelbartShahidi_boundedness}.
All those analytic properties are key to establishing the Weil explicit formula
(for which we refer to Poitou \cite[\S 1]{Poitou_petits}) in this context.
The general formalism of Mestre \cite[\S
I]{Mestre} applies verbatim: we refer to \cite[Chap. 9, Sect.
3]{CheLan} for the details and only recall here what we need to prove our
criterion.
\ps\ps

We denote by $\R \, \Pi_{\rm alg}$ the $\R$-vector space with basis $\Pi_{\rm
alg}$. We fix a {\it test function} $F$, that is an even function $\R
\rightarrow \R$ satisfying the axioms (i), (ii) and (iii) of \cite[\S
I.2]{Mestre} with the constant $c$ loc. cit. equal to $0$ (see also \cite[\S
1]{Poitou_petits}). The reader will not lose anything here assuming simply that
$F$ is compactly supported and of class $\mathcal{C}^2$. We denote by
$\widehat{F}$ the Fourier transform of $F$, with the convention
$\widehat{F}(\xi) = \int_\R F(x) e^{-2i\pi x\xi} dx$. Following \cite[Chap. 9,
Sect. 3]{CheLan}, we first define five symmetric bilinear forms on $\R \,
\Pi_{\rm alg}$, that we denote by ${\rm B}_f^F$, ${\rm B}_\infty^F$, ${\rm
Z}^{\rm F}$, ${\rm e}^\bot$ and $\delta$. The first three of them depend on the
choice of $F$. They are uniquely determined by their values on any $(\pi,\pi')
\in \Pi_{\rm alg} \times \Pi_{\rm alg}$: \ps\ps

\noindent (a) ${\rm B}_f^F (\pi,\pi') = \Re\, \sum_{p,k} \,\, F(k {\rm log} \,p)\,\frac{{\rm log}\, p}{p^{k/2}}\,\,\,\overline{{\rm tr}\, ({\rm c}_p(\pi)^k)} \,\, {\rm tr} \,({\rm c}_p(\pi')^k)$, the sum being over all primes $p$ and integers $k\geq 1$. \ps\ps

\noindent (b) ${\rm B}_\infty^F(\pi,\pi') = {\rm J}_F( {\rm L}(\pi_\infty)
\otimes {\rm L}(\pi_\infty'))$, where ${\rm J}_F : {\rm K}_\infty \rightarrow \R$ is the linear map defined by
\begin{equation}\label{defJF} {\rm J}_F(U)  =  -  \int_\R \,\,\frac{\Gamma'}{\Gamma}(\,\frac{1}{2}\,+\,2\pi t \,i,U) \,\widehat{F}(t)\, {\rm d}t.\end{equation}
We will also denote abusively by ${\rm B}_\infty^F$ the real-valued symmetric bilinear form on ${\rm K}_\infty$ defined by ${\rm B}_\infty^F(U,V)={\rm J}_F(U \cdot V)$. 
With these abusive notations we have ${\rm B}_\infty^F(\pi,\pi')={\rm
B}_\infty^F( {\rm L}(\pi_\infty), {\rm L}(\pi_\infty'))$. \ps\ps

\noindent (c) ${\rm Z}^F(\pi,\pi')$ is the limit of $\sum_{\rho}  \, ({\rm ord}_{s=\rho} \,\Lambda(s,\pi^\vee \times \pi'))\, \,\,\Re\, \widehat{F}(\frac{1-2\rho}{4 i \pi})$,  the sum being over the zeros $\rho$ of
$\Lambda(s,\pi^\vee \times \pi')$ with $0 \leq |\Im\, \rho|\leq T$ and  $0 \leq \Re\, \rho \leq 1$, when the real number $T$ goes to $+\infty$. \ps\ps

\noindent (d) $\delta(\pi,\pi')=1$ if $\pi \simeq \pi'$, and $\delta(\pi,\pi')=0$ otherwise (Kronecker symbol).\ps\ps

\noindent (e) ${\rm e}^\bot(\pi,\pi')=1$ if $\pi$ and $\pi'$ are self-dual with $\epsilon(\pi \times \pi')=-1$, and ${\rm e}^\bot(\pi,\pi')=0$ otherwise. \ps\ps

\noindent The main result is that for any test function $F$ we have the equality of bilinear forms
\begin{equation}\label{RSexplformF}  {\rm B}_f^F\,+\, {\rm B}_\infty^F \,+ \,\frac{1}{2}\, {\rm Z}^F \,= \,\,\widehat{F}(\frac{i}{4\pi}) \, \delta\, \hspace{1 cm}{\it (\, the  ``explicit\,\,formula")} \end{equation}
on the space $\R \,\Pi_{\rm alg}$: see \cite[\S I.2]{Mestre} and \cite[Prop.
9.3.9]{CheLan}. We finally define a last bilinear form on $\R \,\Pi_{\rm alg}$ by the formula
\begin{equation}\label{CF} {\rm C}^F \, \, := \, \, \widehat{F}(\frac{i}{4\pi}) \, \delta \,\,-\, \, {\rm B}^F_\infty -\,\, \frac{1}{2}\,\widehat{F}(0)  \,e^{\perp}.\end{equation}
In our applications, it will represent the "computable" part of the explicit
formula.
Note that for any test function $F$, both  $\widehat{F}(0)$ and
$\widehat{F}(i/4\pi)$ are real numbers, and if $F$ is non-negative then they are
both non-negative.
\ps

\begin{defi}
  Let $F$ be a test function. We will say that $F$ satisfies {\rm (POS)} if we have $F(x)\geq 0$ for all $x \in \R$, and $\Re
  \,\widehat{F}(\xi) \geq 0$ for all $\xi \in \C$ with $|{\rm Im}\, \xi |\leq
  \frac{1}{4\pi}$. 
\end{defi}
  
\begin{prop} \label{prop:basicinequality}
  Let $F$ be a test function satisfying {\rm (POS)}. Then for any integer $r\geq 1$, any $\pi_1,\dots,\pi_r$ in
  $\Pi_{\rm alg}$ and any nonnegative real numbers $t_1,\dots,t_r$, we have
    \begin{equation} \label{eq:ineq_CF}
    {\rm C}^F(\sum_i t_i \pi_i,\sum_i t_i \pi_i) \geq 0 .
  \end{equation}
\end{prop}

\begin{pf}
  By density of the rationals in $\R$, we may assume that the $t_i$ are rational
  numbers, and even that they are integers by homogeneity of the quadratic form
  $x \mapsto {\rm C}^F(x,x)$. But in this case, the statement is \cite[Cor.
  9.3.12]{CheLan}. As the proof is very simple, we give a direct argument. By
  \eqref{RSexplformF} we have ${\rm C}^F \,=\, {\rm B}^F_f \,+\, \frac{1}{2} \,
  ({\rm Z}^F - \widehat{F}(0) \,{\rm e}^{\bot})$. By definition (a) and the
  assumption $F\geq 0$, the symmetric bilinear form ${\rm B}_f^F$ is positive
  (semi-definite) on $\R \Pi_{\rm alg}$. It is thus enough to show that
  $\frac{1}{2} \, ( {\rm Z}^F - \widehat{F}(0)\, {\rm e}^{\bot})$ has
  nonnegative coefficients in the natural basis $\Pi_{\rm alg}$ of $\R \,
  \Pi_{\rm alg}$, {\it i.e.} that we have ${\rm Z}^F(\pi,\pi') \geq
  \widehat{F}(0) \, {\rm e}^{\bot}(\pi,\pi')$ for all $\pi,\pi' \in \Pi_{\rm
  alg}$. But this follows from the definition of ${\rm Z}^F(\pi,\pi')$, the
  assumption on $\Re\,\, \widehat{F}$, and the fact that if we have ${\rm
  e}^{\bot}(\pi,\pi')=1$ then $\Lambda(s,\pi^\vee \times \pi')$ has a zero at
  $s=1/2$ by the functional equation \eqref{eq:functeq}.
\end{pf}

\subsection{Applications}\label{par:applineqalgo}  In what follows we will apply Proposition \ref{prop:basicinequality} to disprove
the existence of representations $\pi$ in $\Pi_{\rm alg}$ such that
$\pi_{\infty}$ is a given algebraic representation, using the knowledge that
there are representations in $\Pi_{\rm alg}$ with known Archimedean
components.

\subsubsection{The basic inequalities} \label{classicalconseq} Before doing so, we first recall
the following basic but important consequence of the explicit formula, 
that we derive here as a very special
case of Proposition \ref{prop:basicinequality} (see also \cite[Cor. 9.3.12 \& 9.3.14]{CheLan}).

\begin{coro} \label{basiccor} Let $F$ be a test function satisfying {\rm (POS)} and fix $U$ in ${\rm K}_\infty$. 
If there is an element $\pi$ in $\Pi_{\rm alg}$ with ${\rm L}(\pi_\infty)=U$ then we have the inequality
\begin{equation} \label{basicineq} {\rm B}_\infty^F (U,U) \leq \widehat{F}(i/4\pi).\end{equation}
More generally, if there are distinct elements $\pi_1,\dots,\pi_m$ in $\Pi_{\rm alg}$ with ${\rm L}((\pi_j)_\infty)=U$ for all $j$, then we have
\begin{equation} \label{taibicrit} {\rm B}_\infty^F (U,U) \leq \frac{1}{m} \widehat{F}(i/4\pi).\end{equation}
\end{coro}

\begin{pf} Consider the element $x = \sum_{i=1}^m \pi_i$ of $\R \Pi_{\rm alg}$. We have ${\rm C}^F(x,x) \geq 0$ by Proposition \ref{prop:basicinequality}.
We clearly have\footnote{\label{footnoteebot} We actually have ${\rm e}^\perp(x,x)=0$. Indeed, if $\pi,\pi'$ are in $\Pi_{\rm alg}$ with ${\rm L}(\pi_\infty)={\rm L}(\pi'_\infty)$, 
then ${\rm e}^\perp(\pi,\pi')=0$. To see this, we may assume $\pi$ and $\pi'$ are self-dual, either both symplectic or both orthogonal (they have the same motivic weight by assumption),
and the assertion follows then from the general property $\epsilon(\pi \times \pi')=1$ proved in \cite[Thm. 1.5.3 (b)]{Arthur_book}. Alternatively, we can easily check $\varepsilon(U\cdot U)=1$ for $U={\rm L}(\pi_\infty)$.} $\widehat{F}(0) {\rm e}^\bot(x,x) \geq 0$, by the inequality $\widehat{F}(0)\geq 0$. We conclude by the equalities $\delta(x,x)=m$ and ${\rm B}_\infty^F(x,x)\,=\,m^2\, {\rm B}_\infty^F(U,U)$.
\end{pf}

Establishing inequality \eqref{basicineq} is the original application 
of the explicit formula for Rankin-Selberg ${\rm L}$-function 
to prove the nonexistence of certain $\pi$ in $\Pi_{\rm alg}$ with given $\pi_\infty$. 
It was used by
\footnote{In the context of Artin ${\rm L}$-functions, the advantages of
considering Rankin-Selberg ${\rm L}$-functions had already been noticed by
Serre, see \cite{Poitou_minorations} p. 150.}
Miller in \cite{Miller} to show that for $n \leq 12$ there is no $\pi$ in
$\Pi_{\rm alg} \smallsetminus \{1\}$ such that ${\rm L}(\pi_\infty)$ is either
${\rm I}_1 + {\rm I}_3+ \cdots + {\rm I}_{2n+1}$ or $\varepsilon_{\C/\R}^n +
{\rm I}_2 + {\rm I}_{4}+\cdots + {\rm I}_{2n}$.
As explained in \cite[Sect. 9.3]{CheLan} and \cite{Chenevier_HM}, 
the simple inequality \eqref{basicineq} is very constraining generally in motivic weight $\leq 23$: 
for suitable $F$ the bilinear form ${\rm B}_\infty^F$ is positive definite on ${\rm K}_\infty^{\leq w}$ for $w \leq 23$, and 
there is an explicit finite list $\mathcal{L}$ of elements of ${\rm K}_\infty$ 
such that whenever $U$ is in ${\rm K}_\infty^{\leq w} - \mathcal{L}$ with $w\leq 23$,
there is no $\pi$ in $\Pi_{\rm alg}$ with ${\rm L}(\pi_\infty)=U$. It has however 
 some limitations: as we shall see, the list $\mathcal{L}$ is quite large, and far from optimal. For instance, 
 it does not seem possible to exclude this way the possibility\footnote{An intuitive reason for that is that 
 there actually exists a $\pi'$ in $\Pi_{\rm alg}$ with very close weights, namely $\pi'=\Delta_{11}$, 
 with ${\rm L}(\pi_\infty') \simeq {\rm I}_{11}$. See the discussion in \cite[Sect. 9.3.19]{CheLan}
 for many other examples (and how to deal with this case differently), which allow to develop some intuition.
 } 
 ${\rm L}(\pi_\infty) \simeq {\rm I}_{13}$. 
Nevertheless, Inequality \eqref{basicineq} will be extremely helpful to us in \S \ref{nonexisteffortless} and Sect.~\ref{sec:mot2324}.
 Inequality \eqref{taibicrit} was first observed by Ta\"ibi. In the case ${\rm B}_\infty^F(U,U)>0$, it may be seen 
 as an effective form of Harish-Chandra's finiteness theorem. We will often use it to show that there is at most one 
 $\pi$ in $\Pi_{\rm alg}$ with given ${\rm L}(\pi_\infty)=U$; note that such a $\pi$ has to be self-dual if it exists, as we have
  ${\rm L}((\pi^\vee)_\infty)={\rm L}(\pi_\infty)^\vee={\rm L}(\pi_\infty)$. \par \medskip

\subsubsection{A general method}\label{sub:testfin}\label{par:method} For $\pi$ in $\Pi_{\rm alg}$, set 
${\rm sd}(\pi)=1$ if $\pi$ is self-dual, and ${\rm sd}(\pi)=0$ otherwise. In this section, we will develop a 
method trying to answer  {\it
in the negative} the following question.\ps

\begin{question} \label{questexplform}
  Fix an integer $r\geq 1$, and for each $1 \leq i \leq r$ elements $U_i$ in
  ${\rm K}_\infty$ and $\delta_i$ in $\{0,1\}$.
  Does there exist distinct representations $\pi_1$, \dots, $\pi_{r}$ in $\Pi_{\rm
  alg}$ with ${\rm L}((\pi_i)_\infty) = U_i$ and ${\rm sd}(\pi_i)=\delta_i$ for
  each $1 \leq i \leq r$ ?
\end{question}

To do so, assume we are given an integer $r\geq 1$
and for each $1 \leq i \leq r$, elements $U_i$ in ${\rm
  K}_\infty$,  $\delta_i$ in $\{0,1\}$, and an integer $m_i\geq 1$. In other words,
  we fix a quadruple
 \begin{equation} \label{eq:quadruple} \mathcal{Q}=(r,\underline{U},\underline{\delta},\underline{m}) \end{equation}
  with $\underline{U}=(U_i)_{1\leq i \leq r}$ in ${\rm K}_\infty^r$, $\underline{\delta}=(\delta_i)_{1 \leq i \leq r}$ in $\{0,1\}^r$ and $\underline{m}=(m_i)_{1\leq i \leq r}$ 
  in $\Z_{\geq 1}^r$. To the choice of $\mathcal{Q}$ and of a test function $F$, we associate the symmetric 
  bilinear form $\beta^F_\mathcal{Q}$ on the standard Euclidean space $\R^r$, defined by the formula
\begin{equation} \label{eq:betaqf} \beta_\mathcal{Q}^F(e_i,e_j) = \frac{1}{m_i} \widehat{F}(i/4\pi) \,\,e_i \cdot e_j - {\rm J}_F(U_i \cdot U_j) - \widehat{F}(0) \delta_i \delta_j \frac{1-\varepsilon(U_i \cdot U_j)}{4},\end{equation}
where $e_1,\dots,e_r$ is the standard orthonormal basis of $\R^r$.
We will discuss the practical numerical evaluation of $\beta_\mathcal{Q}^F$
(i.e.\ of ${\rm J}_F$, $\widehat{F}(0)$ and $\widehat{F}(\frac{i}{4\pi})$) in \S
\ref{numericaleval}.
Set $\mathbb{S}^{r-1}_{+}\, =\, \left\{ (t_i) \in \R^r \,\,\middle|\,\,
\sum_{i=1}^r t_i^2 = 1 \text{ and } \forall i,\,t_i \geq 0 \right\}$.
  
\begin{problem} \label{prob:min} Fix a test function $F$ and a quadruple $\mathcal{Q}=(r,\underline{U},\underline{\delta},\underline{m})$ as in \eqref{eq:quadruple}.
 Determine wether the map  $x \mapsto \beta_\mathcal{Q}^F(x,x)$ takes a negative value on $\mathbb{S}^{r-1}_{+}$. \end{problem}
  
  The relationship between Question \ref{questexplform} and this problem (which does not involve automorphic representations) is the following. 
  Suppose $m_i=1$ for each $1\leq i \leq r$ (the general $m_i$ will play a role only later).
  Assume there are distinct $\pi_1,\dots,\pi_r$ in $\Pi_{\rm alg}$ with ${\rm L}((\pi_i)_\infty) = U_i$ and ${\rm sd}(\pi_i)=\delta_i$ for
  each $1 \leq i \leq r$.
Denote $V = \bigoplus_{i=1}^r \R \pi_i \subset \R \Pialg$ viewed as an Euclidean
space with orthonormal basis $(\pi_1, \dots, \pi_r)$.
As the $\pi_i$ are distinct we actually have $x \cdot x = \delta(x,x)$ for all
$x \in V$. 
We also have
\begin{equation} \label{calcCFBe}
  {\rm B}_\infty^F(\pi_i,\pi_j)= {\rm B}_\infty^F(U_i,U_i)={\rm J}_F(U_i \cdot
  U_j),\,\,\, {\rm e}^\bot(\pi_i,\pi_j)= \delta_i \,\delta_j
  \,\frac{1-\varepsilon(U_i \cdot U_j)}{2}.
\end{equation}
In other words, the linear map $\iota : \R^r \rightarrow V$ defined by $e_i
\mapsto \pi_i$ is an isometry satisfying ${\rm C}^F(\iota(x),\iota(y)) =
\beta_\mathcal{Q}^F(x,y)$ for all $x,y \in \R^r$.
If we are able to find an element $\underline{t}=(t_i) \in \mathbb{S}^{r-1}_+$
with $\beta_\mathcal{Q}^F(\underline{t},\underline{t})<0$, then the element
$\iota(\underline{t})=\sum_{i=1}^r t_i \pi_i$ contadicts Proposition
\ref{prop:basicinequality}: we have answered Question \ref{questexplform} in the
negative. \ps\ps

  From now on we thus focus on Problem \ref{prob:min}.
  We fix an arbitrary quadruple $\mathcal{Q}=(r,\underline{U},\underline{\delta},\underline{m})$ as in \eqref{eq:quadruple} and a test function $F$.
  To simplify the notations we also set $E=\R^r$ and $D=\mathbb{S}_{+}^{r-1}$.
  Let us introduce, for each non-empty $I \subset
\{1, \dots, r\}$: \begin{itemize}
\par \medskip
\item  the subspace $E_I := \bigoplus_{i \in I} \R e_i$ of $E$, the intersection $D_I = D \cap E_I$ and its interior $\mathring{D}_I := \{ \sum_{i \in I} t_i e_i \in D
\,\,|\,\, \forall i \in I,\, t_i >0\}$. We have $D = \bigsqcup_I \mathring{D}_I$.
\par \medskip
\item the minimal eigenvalue $\lambda_I$ of the Gram matrix $(\beta_\mathcal{Q}^F(e_i,e_j))_{i,j \in I}$ of the
  restriction of $\beta_\mathcal{Q}^F$ to $E_I \times E_{I}$, and the corresponding eigenspace $E_{I,
  \lambda_I}$. 
\par \medskip  
\end{itemize}
We also denote by $\mu_\mathcal{Q}^F$ the minimum of $x \mapsto \beta_\mathcal{Q}^F(x,x)$ on $D$.

\begin{prop} \label{prop:algomin} Fix a test function $F$ and a quadruple $\mathcal{Q}=(r,\underline{U},\underline{\delta},\underline{m})$ as in \eqref{eq:quadruple}.
  Let $\mathcal{I}$ be the set of non-empty $I \subset \{1, \dots,
  r\}$ such that $E_{I, \lambda_I}$ intersects $\mathring{D}_I$. Then $\mathcal{I}$ is non-empty and we have $\mu_\mathcal{Q}^F= \min_{I \in \mathcal{I}} \lambda_I$.
\end{prop}

\begin{proof}
  The minimum  $\mu_\mathcal{Q}^F$ of $x \mapsto \beta_\mathcal{Q}^F(x,x)$ on
  the compact $D= \bigsqcup_{I \in \mathcal{I}} \mathring{D}_I$ is reached in
  $\mathring{D}_J$ for some $J$.
  By Lemma \ref{lemm:min_quad} below applied to $E_J$ and ${{\rm C}^F}_{| E_J
  \times  E_{J}}$, any local minimum of $x \mapsto \beta_\mathcal{Q}^F(x,x)$ on
  $\mathring{D}_J$ is an eigenvector for $\lambda_J$ and we have
  $\mu_\mathcal{Q}^F = \lambda_J$.
  We have $J \in \mathcal{I}$, and the other inequality $\mu_\mathcal{Q}^F \leq
  \lambda_I$ for any $I \in \mathcal{I}$ is obvious.
\end{proof}

\begin{lemm} \label{lemm:min_quad}
  Let $E$ be an Euclidean space with scalar product $x \cdot y$, $S$ its unit
  sphere, $b$ a symmetric bilinear form on $E$ and $u$ the (symmetric)
  endomorphism of $E$ satisfying $b(x,y) = x \cdot u(y)$ for all $x,y$ in $E$.
  Assume that the map $S \rightarrow \R, x \mapsto b(x,x)$ has a local minimum
  at the element $v$ in $S$.
  Then $v$ is an eigenvector of $u$ whose eigenvalue $b(v,v)$ is the minimal
  eigenvalue of $u$.
\end{lemm}

\begin{pf}
  Set $q(x)=b(x,x)$. We have $q(\frac{v+w}{|v+w|}) = q(v) + 2 b(w,v) + {\rm
  O}(w^2)$ when $w$ goes to $0$ in $v^\perp$. As $v$ is a local minimum of $q$,
  this shows $b(w,v)=w \cdot u(v) = 0$ for all $w$ in $v^\bot$. So $v$ is an
  eigenvector of $u$. Denote by $\lambda$ be the corresponding eigenvalue.
  Assume $u$ has an eigenvalue $\lambda' < \lambda$, and choose $v'$ in $S$ with
  $u(v') = \lambda' v'$. We have $b(v,v')=0$ and $q( (1-\epsilon^2)^{1/2} v
  + \epsilon v') = \lambda + \epsilon^2 (\lambda'-\lambda)<\lambda$ for all
  $0<\epsilon<1$, a contradiction.
\end{pf}

\begin{exam}\label{casr2} Assume $r=2$ and set 
 $(\beta_\mathcal{Q}^F(e_i,e_i))_{1\leq i,j \leq 2}= {\tiny \left( \begin{array} {cc} a & b \\ b & c \end{array} \right)}$.
 We have $\lambda_{\{1\}}=a$ and $\lambda_{\{2\}}=c$. We may assume $a$ and $c$ are $\geq 0$, 
 otherwise Problem \ref{prob:min} is solved. For $I=\{1,2\}$, the eigenvalue $\lambda_I$ is $<0$ if and only if 
the determinant $a c - b^2$ is $<0$. In this case, we have $b \neq 0$ and the eigenspace $E_{I,\lambda_I}$ is a line:
we easily check that this line meets $\mathring{D}_I$ if and only if $b<0$. 
Proposition \ref{prop:algomin} implies that assuming $a c < b^2$ and $b<0$, 
or equivalently $b + \sqrt{a c}<0$, we have $\mu_\mathcal{Q}^F<0$.

\end{exam}
\begin{lemm} \label{lemm:tiegaltj} Fix a test function $F$ and a quadruple $\mathcal{Q}=(r,\underline{U},\underline{\delta},\underline{m})$ as in \eqref{eq:quadruple}.
Assume $\mu_\mathcal{Q}^F<0$, $\widehat{F}(i/4\pi)>0$, as well as $(U_i,\delta_i,m_i)=(U_j,\delta_j,m_j)$ for some indices $i \neq j$. Then any element $\underline{t}$ in $D$
      with $\beta_\mathcal{Q}^F(\underline{t},\underline{t})=\mu$ satisfies $t_i=t_j$. 
\end{lemm}

\begin{pf}
  Set $q(x)=\beta_\mathcal{Q}^F(x,x)$.
  Consider the set $B = \bigcup_{0 \leq \lambda < 1} \lambda D$; then $B \cup D$
  is convex and we have $q(x)>\mu_\mathcal{Q}^F$ for $x \in B$.
  Fix $\underline{t} \in D$  with
  $\beta_\mathcal{Q}^F(\underline{t},\underline{t})=\mu_\mathcal{Q}^F$. 
  An inspection of Formula \eqref{eq:betaqf} shows that for any real numbers
  $s_i,s_j$ we have 
  \begin{equation} \label{eq:CF_two_var}
    q \left( s_i e_i + s_j e_j + \sum_{l \neq i,j} t_l e_l \right) =
    -\frac{2}{m_i} s_i s_j \widehat{F}(i/4 \pi) + \left( \text{function of }
    s_i+s_j \right).
  \end{equation}
  The set
  \begin{equation} \label{eq:ti_plus_tj_eq_c}
    \{ (s_i, s_j) \, | \, s_i, s_j \geq 0, \, s_i + s_j = t_i+t_j, \, s_i^2 +
    s_j^2 + \sum_{l \neq i,j} t_l^2 \leq 1 \}
  \end{equation}
  is a compact interval in $\R^2$ with end points $(t_i, t_j)$ and $(t_j, t_i)$.
  By assumption we have $\widehat{F}(i/4 \pi) > 0$, and so the minimum of
  \eqref{eq:CF_two_var} on \eqref{eq:ti_plus_tj_eq_c} is reached for $s_i = s_j
  = (t_i+t_j)/2$. 
  If we assume $t_i \neq t_j$ then $s_i e_i + s_j e_j + \sum_{l \neq i,j} t_l
  e_l$ belongs to $B$, a contradiction.
\end{pf}

This lemma leads to the following considerations. Start with a quadruple $\mathcal{Q}=(r,\underline{U},\underline{\delta},\underline{m})$
with the property $m_i=1$ for $i=1,\dots,r$. Assume we have a partition
$$\{1,\dots,r\} = \coprod_{l=1}^{r'} P_l$$
with the property that for each $1 \leq l \leq s$, and each $i,j \in P_l$, we have $(U_i,\delta_i)=(U_j,\delta_j)$.
Consider the new quadruple $\mathcal{Q}'=(r',\underline{U'},\underline{\delta'},\underline{m'})$
where for each $1 \leq l \leq r'$ we define $U'_l$ (resp. $\delta'_l$) as the element $U_i$ (resp. $\delta_i$) 
with $i \in P_l$ (this does not depend on the choice of such an $i$), and set $m_l=|P_l|.$ We have a natural inclusion 
$$i : \R^{r'} \longrightarrow \R^r$$
sending $e_l$ to $\frac{1}{\sqrt{m_l}} \sum_{i \in P_l} e_i$ for each $1 \leq l \leq r'$. This embedding is an isometry for
the standard Euclidean structures on both sides, and it follows from Formula \eqref{eq:betaqf} that we have
$\beta_\mathcal{Q}^F(i(x),i(y)) = \beta_\mathcal{Q'}^F(x,y)$ for all $x,y \in \R^{r'}$ and all test functions $F$.
Lemma \ref{lemm:tiegaltj} shows then (the inequality 
$\mu_\mathcal{Q}^F \leq \mu_\mathcal{Q'}^F$ being obvious):

\begin{coro}\label{cor:minegalti} Let $\mathcal{Q}$ and $\mathcal{Q'}$ be as above, and fix a test function $F$ with $\widehat{F}(i/4\pi)>0$.
We have $\mu_\mathcal{Q}^F<0$ if and only if $\mu_\mathcal{Q'}^F<0$, and if these inequalities hold we have
$$\mu_\mathcal{Q}^F=\mu_\mathcal{Q'}^F.$$
\end{coro} 

\par \medskip
\begin{rema} \label{remasdequi} Assume we have two quadruples of the form $\mathcal{Q}=(r,\underline{U},\underline{\delta},\underline{m})$
and $\mathcal{Q'}=(r,\underline{U},\underline{\delta'},\underline{m})$ with $\delta'_i \geq \delta_i$ for each $1 \leq i \leq r$. Choose a test function $F$
with $\widehat{F}(0) \geq 0$. Then we have $\beta_{\mathcal{Q'}}^F(x,y) \leq \beta_{\mathcal{Q}}^F(x,y)$ for all 
$x,y$ in $\R_{\geq 0}^r$ by Formula \eqref{eq:betaqf}. This shows $\mu_\mathcal{Q'}^F \leq \mu_\mathcal{Q}^F$. In particular,
$\mu_\mathcal{Q}^F<0$ implies $\mu_\mathcal{Q'}^F<0$.
\end{rema}

\subsubsection{A digression on numerical evaluation}\label{numericaleval} 
Before discussing the natural algorithm that follows from Propositions \ref{prop:algomin} and Corollary \ref{cor:minegalti}, 
let us discuss the numerical evaluation of the bilinear form ${\rm C}^F$. Given a test function $F$, 
we will have to be able to compute with enough and certified precision the quantities 
\begin{equation}  
\label{calcquant}
\widehat{F}(0), 
\hspace{.3cm} \widehat{F}(i/4\pi) 
\hspace{.3cm} \textrm{and} \hspace{.3cm} {\rm J}_F(U) \,\,\textrm{for} \,\,U=1 \,\textrm{and} \,\,U={\rm I}_w\,\,(w \in \Z).
\end{equation}
It amounts to computing certain indefinite integrals.
Numerical integration routines of computer packages such as \texttt{PARI} allow
to compute approximations of such integrals, with increasing and in principle
arbitrarily large accuracy.
Although these routines have been very useful in our preliminary computations,
and experimentally return highly accurate values when properly used, it would be
delicate to rigorously bound the differences between 
these computed values and the exact
ones.
This is why we proceed differently.\ps\ps

In this paper, we only use Odlyzko's function $F={\rm F}_\ell$ with parameter
$\ell>0$. This is the function defined by ${\rm F}_\ell(x)={\rm g}(x/\ell)/ {\rm cosh}(x/2)$,
 where ${\rm g} : \R \rightarrow \R$ is twice the convolution square of the function 
$x \mapsto {\rm cos} (\pi x) {\rm 1}_{|x|\leq 1/2}$: see \cite[Sect. 3]{Poitou_petits} and \cite[Sect. 9.3.17]{CheLan}. 
These functions satisfy {\rm (POS)}, $\widehat{\mathrm{F}_\ell}(i/4\pi)=\frac{8}{\pi^2}
\ell$, and Proposition 9.3.18 of \cite{CheLan} provides alternative closed
formulas for all the other quantities in \eqref{calcquant} (see Proposition \ref{explGRH} for similar expressions).
Each is a sum of a linear combination of a few special values of  the classical
digamma function $\psi = \Gamma' / \Gamma$ and of its derivative $\psi'(z) =
\sum_{n \geq 0} 1/(n+z)^2$, and of a simple rapidly converging series with given
tail estimates \cite[(3) p.127]{CheLan}. 
Using these formulas and estimates, we implemented functions in Python using
Sage \cite{sage} to compute certified intervals containing the real numbers
\eqref{calcquant} for $F=\mathrm{F}_\ell$.
See \cite{homepage} for the source code.
For interval arithmetic, Sage relies on the Arb library
\url{http://arblib.org/}.
Our computations only use the four operations, the exponential and logarithm
functions, the constant $\pi$, the function $\psi$ (\verb!acb_digamma! in this
library), and its derivative (a special case of \verb!acb_polygamma!).

\begin{rema} \label{rema:qfminim}
  Fix an integer $0 \leq w \leq 23$.
  For suitable $\ell>0$, the restriction of ${\rm B}_\infty^{{\rm F}_\ell}$ to
  ${\rm K}_\infty^{\leq w}$ is positive definite (see e.g. Lemma 9.3.37 and
  Proposition 9.3.40 in \cite{CheLan}, as well as \cite{Chenevier_HM}). 
  By Corollary \ref{basiccor}, it is important to be able to enumerate, for
  $c>0$, all the (finitely many) effective elements $U $ in ${\rm
  K}_\infty^{\leq w}$ satisfying ${\rm B}_\infty^{{\rm F}_\ell}(U,U) \leq c$.
  We use for this the Fincke-Pohst algorithm enumerating the short vectors in a
  lattice.
  Using interval arithmetic as explained above we can obtain rational lower
  bounds for the coefficients of the Gram matrix of
  $\mathrm{B}_\infty^{F_\ell}$, and since we are only interested in effective
  elements of ${\rm K}_\infty^{\leq w}$ we can work with this rational Gram
  matrix instead.
  Unfortunately \verb!PARI!'s \verb!qfminim! does not (yet?) include an
  \emph{exact} variant of the Fincke and Pohst algorithm for Gram matrices with
  integral entries.
  For this reason we reimplemented the first (simple) algorithm of Fincke-Pohst
  \cite{FinckePohst} using only exact computations, adding the condition of
  effectivity in the recursion to avoid unnecessary computations. 
  Of course in practice this algorithm always leads to the same conclusions as
  \verb!PARI!'s \verb!qfminim! algorithm, if the latter is properly used.
  See \cite{homepage} for our source code.
\end{rema}

\subsubsection{The algorithm} \label{paralgo} 

The following algorithm tries to solve Problem \ref{prob:min} using the method discussed in \S \ref{par:method}.\par \medskip

\begin{enumerate}
\item[{\it Input:}] A quadruple $\mathcal{Q}=(r,\underline{U},\underline{\delta},\underline{m})$ as in \ref{eq:quadruple}. 
\ps\ps

\item[{\it Output:}] (if the algorithm terminates) A triple
  $(\ell,I,\underline{t})$ with $\ell>0$, a non empty $I \subset \{1,\dots,r\}$,
  and $\underline{t} \in \R^I$ with $\beta_\mathcal{Q}^{{\rm
  F}_\ell}(\underline{t},\underline{t})<0$.
 \ps\ps

  \item[{\it Step 1}.] Choose a real number $\ell>0$ and compute an approximation $(G_{i,j})_{1 \leq i,j \leq r}$ of the Gram
    matrix {\small $(\beta_\mathcal{Q}^{{\rm F}_\ell}(e_i,e_j))_{1 \leq i,j \leq r}$}.
    We do this using the formulas \eqref{eq:betaqf} of \S \ref{par:method} and the expressions 
    of \cite{CheLan} for the quantities \eqref{calcquant} with $F={\rm F}_\ell$ discussed in \S \ref{numericaleval}.
    \par \medskip
    
      \item[{\it Step 2}.]  Choose a nonempty subset $I$ of $\{1,\dots,r\}$ and compute an
    approximation $\lambda_I$ of the minimal eigenvalue of the Gram matrix
    $(G_{i,j})_{i,j \in I}$, as well as an approximate corresponding
    eigenvector $(t_i)_{i \in I}$.
    For doing so, we apply \verb!PARI!'s \verb!qfjacobi! function to
    $(G_{i,j})_{i,j \in I}$ (an implementation of Jacobi's method).
    \par \medskip
    
    \item[{\it Step 3}.] If we have $\lambda_I<0$ and $t_i>0$ for all $i \in I$, return $\ell$, $I$ and $\underline{t}=(t_i)_{i \in
    I}$ and go to Step 4. Otherwise, go back to Step 2 and change the subset $I$. 
    If all the $I$ have been tried, go back to Step 1 and change the parameter $\ell$.  
    \par \medskip
  \item[{\it Step 4}.] Check rigorously, using interval arithmetic as discussed
    in \S \ref{numericaleval}, that we have indeed
    $\beta_\mathcal{Q}^{\mathrm{F}_\ell} (\underline{t},\underline{t})<0$.
    If it fails go back to the second part of Step 3.
\end{enumerate}
\par \smallskip

Let us comment this algorithm and discuss the unexplained choices involved: \par\medskip

\begin{itemize}
\item The choice of $\ell$ in Step 1 is based on some preliminary experiments, 
and it seems quite hard to guess a priori a range for the best ones. 
In our applications, we will choose $\ell$ in $[\frac{1}{2},15] \cap 10^{-2}\Z$. 
\par \medskip
\item The loop consisting of Steps 2 and 3, for a given $\ell$, 
can be very long if $r$ is large, as there are $2^r-1$ possibilities for $I$.  In practice, we order the subsets $I$ by increasing cardinality, 
often restrict to $I$ of small cardinality. In practice again, the eigenspace $E_{I, \lambda_I}$ is just a line. 
\par \smallskip
\item  Whenever we reached Step 4, the rigorous check with interval arithmetic of the inequality
$\beta_\mathcal{Q}^{{\rm F}_\ell}(\underline{t},\underline{t})<0$ never fails in practice.
    {\it This single check is enough to justify that the algorithm does show that $x \mapsto \beta_\mathcal{Q}^{{\rm F}_\ell}(x,x)$
    takes a negative value on $\mathbb{S}_+^{r-1}$}. This is the most important remark regarding this algorithm.
In particular, we do not have to justify any of the computations done in Steps 1, 2 and 3 before: all is fair in order to find
   a candidate $(\ell,I,\underline{t})$. Of course, the experimental fact that the last check in Step 4 never fails just reflects that the computations made with  \verb!PARI!
   are highly accurate. 
\end{itemize}

In the end, a charm of this algorithm is that even if the loop of Steps 1, 2 and 3 can be very long, 
   once we get the candidate $(\ell,I,\underline{t})$ we just have to store it, and then inequality $\beta_\mathcal{Q}^{{\rm F}_\ell}(\underline{t},\underline{t})<0$
   can be rechecked instantly. 
   
\subsubsection{Final algorithm} \label{paralgo2}

For our applications in \S \ref{par:poids22}, \S \ref{sec:mot2324} and \S
\ref{nonexisteffortless}, it will be convenient to apply Algorithm \ref{paralgo}
in the following slightly more restrictive context.\ps\ps

{\it Set up.}
We fix $U$ in ${\rm K}_\infty$, $\delta$ in $\{0,1\}$, and an integer $m\geq 1$. 
We fix as well a {\it known} set ${\bf S}$ of elements of $\Pi_{\rm alg}$ and
our aim is to show that there does not exist distinct elements
$\pi_1,\dots,\pi_m$ in $\Pi_{\rm alg}\setminus {\bf S}$ with ${\rm
L}((\pi_i)_\infty)=U$ and ${\rm sd}(\pi_i)\geq \delta$ for each $1 \leq i \leq
m$.
By known we mean that we assume given ${\rm L}(\varpi_\infty)$ and
\footnote{Let us mention that, at present, the authors are not aware of the
existence of any non self-dual element in $\Pi_{\rm alg}$, so in practice will
always actually have ${\rm sd}(\varpi)=1$ for $\varpi$ in ${\bf S}$.}
${\rm sd}(\varpi)$ for all $\varpi \in S$.
We denote by $S$ the set of triples $(U',\delta',m')$ in ${\rm K}_\infty \times
\{0,1\} \times \Z_{\geq 1}$ such that there are exactly $m'$ elements $\varpi$
in $S$ with $({\rm L}(\varpi_\infty),{\rm sd}(\varpi))=(U',\delta')$.
\ps\ps

{\it Algorithm.} Set $r=1+|S|$. Assuming $|S|\geq 1$ it is convenient choose a bijection
\begin{equation} \label{identS} S \overset{\sim}{\rightarrow} \{2,\dots,r\}\end{equation}
and write $S=\{(U_i,\delta_i,m_i)\, \, |\, \, 2 \leq i \leq r\}$. Set also  $(U_1,\delta_1,m_1)=(U,\delta,m)$. This defines a quadruple $\mathcal{Q}=(r,\underline{U},\underline{\delta},\underline{m})$. 
We now apply Algorithm \ref{paralgo} to $\mathcal{Q}$. In Step $2$ we obviously may, and do, 
 restrict to subsets $I$ containing $1$, {\it i.e.} of the form $I=\{1\} \coprod S'$ with $S' \subset S$, via the identification \eqref{identS}. \ps\ps
 
{\it Output.}
When this algorithm terminates, it produces $(\ell,I,\underline{t})$ such that $\beta_\mathcal{Q}^{{\rm F}_\ell}(\underline{t},\underline{t})<0$.
For $j=2,\dots,m$, set $x_j=\frac{1}{\sqrt{m_j}} \sum \varpi$, the sum being
over the $\varpi \in {\bf S}$ with $(\mathrm{L}(\varpi_\infty),
\delta(\varpi))=(U_j,\delta_j)$.
Assume there are distinct elements $\pi_1,\dots,\pi_m$ in $\Pi_{\rm
alg} \setminus \mathbf{S}$ with $\mathrm{L}((\pi_i)_\infty)=U$ and ${\rm
sd}(\pi_i)\geq \delta$ for each $1 \leq i \leq m$.
Then for the element $x= t_1 \frac{1}{\sqrt{m}} (\pi_1+\dots+\pi_m) +
\sum_{i=2}^r t_i x_i$ \,\,of\,\, $\R\, \Pi_{\rm alg}$ we have ${\rm C}^{{\rm
F}_\ell}(x,x) \leq \beta_\mathcal{Q}^{{\rm F}_\ell}(\underline{t},\underline{t})
<0$ (see Remark \ref{remasdequi} for the first inequality), contradicting
Proposition \ref{prop:basicinequality}. 

\begin{rema} \label{em:critchelan}
  In the case $S=\emptyset$, this method just amounts to applying Corollary
  \ref{basiccor}.
  In the case $|S|=1$, it amounts to applying Scholium 9.3.26 of \cite{CheLan},
  by the discussion of Example \ref{casr2}.
  The case of arbitrary $|S|$ can thus be viewed as a generalisation of these
  criteria {\it loc. cit.}
  See \cite{homepage} for our source code in \texttt{PARI} of the algorithm
  above.
\end{rema}

\subsubsection{An illustration}\label{par:poids22}

Algorithm \ref{paralgo2} can be used to give another proof of the 
Chenevier-Lannes classification theorem \cite[Thm. 9.3.3]{CheLan} mentionned in \S \ref{classresult} of the introduction,
which is both very fast (a few seconds of computations) and systematic.
Although this alternative proof shares many steps with the one {\it loc. cit.}, it bypasses 
the geometric criterion involving Satake
parameters explained in \S 9.3.29 therein (and does not rely at all 
on any computation of Satake parameters of known elements in $\Pi_{\rm alg}$).
Such an improvement, although not decisive here, will be crucial in the proof of Theorem \ref{thm23}, because
at present we only know rather few Satake parameters for the known elements of $\Pi_{\rm alg}$ of dimension $>3$
(see however \cite{bfgwebsite} and \cite{megarbane}). 
\ps \ps

For the convenience of the reader, and in order to illustrate our new method, let us now explain the aforementioned proof of 
 \cite[Thm. 9.3.3]{CheLan} in the most complicated case of motivic weight $22$. So we want to prove that there is a unique $\pi$ 
 in $\Pi_{\rm alg}$ of  motivic weight $22$, namely $\pi = {\rm Sym}^2 \Delta_{11}$
 (for which we have  ${\rm L}(\pi_\infty)={\rm I}_{22}+\varepsilon_{\C/\R}$). 
 We refer to the working sheet in \cite{homepage} for the numerical verifications used below. \ps\ps
 
 {\it Step 1}. We first observe that ${\rm B}_\infty^{{\rm F}_\ell}$ is positive definite on the lattice ${\rm K}_\infty^{\leq 22}$ for $\ell =4.38$ (Lemma \cite[9.3.37]{CheLan}).
 Using the PARI $\texttt{qfminim}$ command, or better Remark \ref{rema:qfminim}, we may and do 
 list all the effective elements $U$ in ${\rm K}_\infty^{\leq 22}$ satisfying 
 $${\rm B}_\infty^{{\rm F}_\ell}(U,U) \leq \widehat{{\rm F}_\ell}(i/4\pi)$$
for $\ell = 4.38$. We retain furthermore only those satisfying $\det U=1$ and containing ${\rm I}_{22}$. 
The resulting list $\mathcal{U}$ has $158$ elements. If $\pi$ in $\Pi_{\rm alg}$ has motivic weight $22$, 
then ${\rm L}(\pi_\infty)$ is in $\mathcal{U}$ by Corollary \ref{basiccor}. We will study each of these $158$ possibilities for ${\rm L}(\pi_\infty)$
mostly case by case.
\ps\ps

{\it Step 2}. Denote by ${\rm N}(U)$ be the number of elements $\pi$ in $\Pi_{\rm alg}$ 
with ${\rm L}(\pi_\infty)=U$. We want to bound 
${\rm N}(U)$ for each $U$ in $\mathcal{U}$ by applying Inequality \eqref{taibicrit} of Corollary \ref{basiccor}.
For this we check that for all $U$ in $\mathcal{U}$ we have 
${\rm B}_\infty^{{\rm F}_\ell}(U,U) > \frac{1}{2} \widehat{{\rm F}_\ell}(i/4\pi)$, 
unless $U$ belongs to the subset $\mathcal{U}'\,=\,\{{\rm I}_{22}+{\rm
I}_{12},\, {\rm I}_{22}+{\rm I}_{10},\, {\rm I}_{22}+{\rm I}_8\}$, in which case
we only have  ${\rm B}_\infty^{{\rm F}_\ell}(U,U) > \frac{1}{3} \widehat{{\rm
F}_\ell}(i/4\pi)$ (here $\ell$ is still $4.38$). 
This shows ${\rm N}(U) \leq 1$ for $U$ in $\mathcal{U}\setminus \mathcal{U}'$, and ${\rm N}(U) \leq 2$ for $U$ in $\mathcal{U}'$.
\ps\ps

{\it Step 3.} Fix $U$ in $\mathcal{U}'$. We want to show ${\rm N}(U) \leq 1$.
Assume ${\rm N}(U)=2$, {\it i.e.} that there exist distinct $\pi_1,\pi_2$ in
$\Pi_{\rm alg}$ with ${\rm L}((\pi_1)_\infty) = {\rm L}((\pi_2)_\infty)=U$.
We apply Algorithm \ref{paralgo2}
to $U$, $\delta=0$  (see Remark \ref{remasdequi}), $m=2$ and to
the known set ${\bf S}=\{1,\, \Delta_{11},\, \Delta_{15},\, \Delta_{17},\,
\Delta_{19},\, \Delta_{21},\, {\rm Sym}^2 \Delta_{11}\}$.
For $U={\rm I}_{22}+{\rm I}_{12}$ and $\ell=3.5$ 
it returns for instance an element close to 
$$x= \,\,\,0.924 \,\,\frac{1}{\sqrt{2}} (\pi_1+ \pi_2) \,\,\,+\,\,\, 0.383 \,\,\Delta_{11}.$$
We verify (using interval arithmetic, see \S \ref{numericaleval}) that 
we have ${\rm C}^{{\rm F}_\ell}(x,x) \simeq -0.173$ up to $10^{-3}$: this
contradicts Proposition \ref{prop:basicinequality}. 
The algorithm produces a similar element $x$ in the case $U={\rm I}_{22}+{\rm
I}_{10}$, with $(0.924,0.383)$ replaced by $(0.900,0.436)$, and we have then
${\rm C}^{{\rm F}_\ell}(x,x) \simeq -0.198$ up to $10^{-3}$.
Nevertheless, it does not seem to produce any contradiction in the remaining
case $U={\rm I}_{22}+{\rm I}_{8}$, even if we let $\ell$ vary.
To deal with this last $U$ we add to ${\bf S}$ the known element $\Delta_{21,9}$
of $\Pi_{\rm alg}$, whose Archimedean ${\rm L}$-parameter is ${\rm I}_{21}+{\rm
I}_9$ (which is ``close'' to $U$). 
The algorithm returns for $\ell=3.5$ the element 
$x\,\,\,:= \,\,\,0.942 \,\,\frac{1}{\sqrt{2}} (\pi_1+ \pi_2) \,\,\,+\,\,\, 0.335 \,\,\Delta_{21,9}$
and we verify that we have ${\rm C}^{{\rm F}_\ell}(x,x) \simeq  -0.147$ up to
$10^{-3}$, which is indeed $<0$.\ps\ps

{\it Step 4.} We have proved so far ${\rm N}(U)\leq 1$ for all $U\in \mathcal{U}$. In particular, 
any $\pi$ in $\Pi_{\rm alg}$ with ${\rm L}(\pi) \in \mathcal{U}$ is self-dual. 
Fix $U$ in $\mathcal{U}$. We now apply Algorithm \ref{paralgo2} to 
$U$, $\delta=1$, $m=1$ and to the same set $S$ as above (with $|S|\leq 7$). Using the nine $\ell$
in $[3,5]\cap \frac{1}{4}\Z$, it yields a contradiction in each case! 
Actually, if we restrict to subsets $S' \subset S$ with $|S'|=1$ in Step 2 of the algorithm 
(in other words, if we only apply the Scholium of \cite{CheLan} mentioned in Remark \ref{em:critchelan})
we already get a contradiction for all but 
 the $7$ elements $U$ mentioned in Table \ref{tableauelim3} below. These remaining cases 
were exactly the ones dealt with using the geometric criterion involving Satake
parameters explained in \cite[\S 9.3.29]{CheLan}. In these $7$ cases, our algorithm produces contradictions 
for subsets $S'$ of size $2$, such as the ones gathered in Table \ref{tableauelim3}. This concludes the proof.
$\square$

\begin{table}[htp]
\renewcommand{\arraystretch}{1.5}
\begin{center}{\scriptsize {\begin{tabular}{c|c|c} 
$U$ & $x$ & ${\rm C}^{{\rm F}_4}(x,x)$ up to $10^{-3}$ \cr \hline
${\rm I}_{22}+{\rm I}_{16}+1$ & $0.625 \,\,\pi_1\,\,\, + \,\,\,0.611\,\, \Delta_{19}\,\,+\,\,\, 0.485\, \, {\rm Sym}^2 \Delta_{11}$  & $-0.427$ \cr \hline
${\rm I}_{22}+{\rm I}_{12}$ & $0.640 \,\,\pi_1\,\,\, + \,\,\,0.582\,\, \Delta_{15}\,\,+\,\,\, 0.502\, \, \Delta_{17}$ & $-0.511$ \cr \hline
${\rm I}_{22}+{\rm I}_{12}+1$ & $0.709 \,\,\pi_1\,\,\, + \,\,\,0.432\,\, \Delta_{11}\,\,+\,\,\, 0.558\, \, \Delta_{15}$ & $-0.204$ \cr \hline
${\rm I}_{22}+{\rm I}_{20}+{\rm I}_{10}+\epsilon_{\C/\R}$ & $0.636 \,\,\pi_1\,\,\, + \,\,\,0.393\,\, \Delta_{19}\,\,+\,\,\, 0.664\, \, {\rm Sym}^2 \Delta_{11}$ & $-0.037$ \cr \hline
${\rm I}_{22}+{\rm I}_{16}+{\rm I}_{10}+\epsilon_{\C/\R}$ & $0.701 \,\,\pi_1\,\,\, + \,\,\,0.531\,\, \Delta_{19}\,\,+\,\,\, 0.476\, \, {\rm Sym}^2 \Delta_{11}$ & $-0.246$ \cr \hline
${\rm I}_{22}+{\rm I}_{4}$ & $0.630 \,\,\pi_1\,\,\, + \,\,\,0.608\,\, \Delta_{11}\,\,+\,\,\, 0.483\, \, {\rm Sym}^2 \Delta_{11}$ & $-0.204$ \cr \hline
${\rm I}_{22}+{\rm I}_{20}+{\rm I}_{14}+{\rm I}_4$ & $0.696 \,\,\pi_1\,\,\, + \,\,\,0.297\,\, 1\,\,+\,\,\, 0.654\, \, \Delta_{21}$ & $-0.047$ \cr
\end{tabular}}}
\end{center}
\caption{{\small Some elements $x$ with ${\rm L}((\pi_1)_\infty)=U$ and ${\rm C}^{{\rm F}_\ell}(x,x)<0$ for $\ell=4$.}}
\label{tableauelim3}
\end{table}
\renewcommand{\arraystretch}{1}

\subsubsection{Another illustration: a strengthening of a vanishing result in \cite{cleryvandergeer}} \label{webeatcleryvandergeer}

As another example, let us show that for all odd $1 \leq w \leq 53$, there is no cuspidal selfdual algebraic level $1$ automorphic representation $\pi$ of 
${\rm PGL}_4$ with ${\rm L}(\pi_\infty)={\rm I}_w+{\rm I}_w$. We apply for this Algorithm \ref{paralgo2} to $U=2\, {\rm I}_w$, $\delta=m=1$ and to the set 
${\bf S}$ of $\dim {\rm S}_{w+1}({\rm SL}_2(\Z))$ cuspidal automorphic representations generated by level $1$ cuspidal eigenforms for ${\rm SL}_2(\Z)$. 
Note that we have $|S|=0$ for $w=13$ and $w<11$, and $|S|=1$ otherwise. We obtain a contradiction in each case using $S'=S$ and $\ell=5$.
This shows ${\rm S}_{(k_1,2)}(\Gamma_2)=0$ for all $k_1\leq 54$ by \cite[Lemma A.2]{cleryvandergeer}.

\section{Effortless computation of masses in the trace formula} \label{par:effortless}

Let $G$ be a split classical group over $\Z$ such that $G(\R)$ admits discrete series.
In other words, $G$ belongs to one of the three families
\[ (\SO_{2n+1})_{n \geq 1},\ \ (\Sp_{2n})_{n \geq 1} \ \ \text{and} \ \
(\SO_{4n})_{n \geq 1} \]
In this section, we explain how to implement the strategy explained in \S \ref{sub:intro_effortless} 
in order to determine the masses ${\rm m}_c$ for $c \in \mathcal{C}(G)$. 
In \S \ref{elementaryconjclass}, we first make 
elementary observations that will allow us to replace
$\Ccal(G)$ by a concrete set $\mathcal{P}(G)/\sim$ of equivalence classes of polynomials,
and to rewrite the elliptic terms accordingly.
In \S \ref{sec:spinornorm}, we restrict further $\mathcal{C}(G)$ 
by an explicit subset containing all conjugacy classes $c$
such that $\mathrm{m}_c \neq 0$.
It turns out that finding a subset significantly smaller than $\Ccal(G)$ is
possible in the case of special orthogonal groups, using spinor norms considerations.
In the last \S  \ref{nonexisteffortless}, we finally prove Theorems \ref{theomasseseffortless} 
and \ref{theomassesmixed}, by discussing how to produce sets $\Lambda$ of dominant weights satisfying
a variant of conditions (P1) and (P2)  alluded in \S \ref{sub:intro_effortless}.

\subsection{Conjugacy classes and characteristic polynomials: elementary
observations}\label{elementaryconjclass}

Let $G$ be one of $\SO_{2n+1}$, $\Sp_{2n}$ or
$\SO_{4n}$. We shall denote by ${\rm n}_G$ 
the dimension of the standard (or tautological)
representation of $G$, so ${\rm n}_G$ is respectively $2n+1, 2n$ or $4n$.
(Do not confuse ${\rm n}_G$ with the ${\rm n}_{\widehat{G}}$ introduced in \S \ref{sub:intro_effortless}).

The indexing set for the sum defining the elliptic part
$\mathrm{T}_{\mathrm{ell}}(G; \lambda)$ of the geometric side in
Arthur's ${\rm L}^2$-Lefschetz trace formula \cite{ArthurL2} recalled in
\eqref{eq:tell}, is the set of conjugacy classes of semi-simple elements $\gamma
\in G(\Q)$ which are {\it $\R$-elliptic} (i.e.\ $\gamma$ belongs to an
anisotropic maximal torus of $G_{\R}$, in particular the eigenvalues of $\gamma$
have absolute value $1$) and such that the conjugacy class of $\gamma$ in
$G(\A_f)$ meets the compact support of the smooth function we put in the trace
formula, in our case the characteristic function of $G(\Zhat)$. In particular, 
the characteristic polynomial $P_\gamma$ of such a $\gamma$, 
a monic polynomial of degree ${\rm n}_G$ in $\Q[X]$, 
belongs to $\Z[X]$ and has all its complex roots of absolute value $1$. 
Using a celebrated theorem of Kronecker, these conditions 
imply that the roots of $P_\gamma$ are roots of unity, hence that the semi-simple element 
$\gamma$ has finite order. This explains the discussion of 
Formula \eqref{eq:tell} in \S \ref{sub:intro_effortless}, and the indexing set $\mathcal{C}(G)$
of finite order elements of $G(\Q)$ taken up to conjugacy by
$G(\Qbar)$ in the sum \eqref{eq:tellmasse}.\ps

\begin{defi} \label{def:PG}  Let $\mathcal{P}(G)$ be the set of polynomials $P$ in $\Q[X]$ having degree
  $\nG$, which are products of cyclotomic polynomials and in which $X+1$ has
  even multiplicity (or equivalently, with $P(0)=(-1)^{\nG}$).
\end{defi}

If $c$ is a class in $\mathcal{C}(G)$, then all the elements $\gamma \in c$ have the same $P_\gamma$, 
and we will denote by $P_c$ this polynomial. It is an element of $\mathcal{P}(G)$ by the above discussion.
We have thus defined a map 
\begin{equation} \label{eq:mapchar}{\rm char} : \Ccal(G) \rightarrow \mathcal{P}(G), \,\,c \mapsto P_c.\end{equation}
It is well-known that if two semi-simple conjugacy classes $c_1,c_2$ in the classical group $G(\overline{\Q})$
have the same characteristic polynomial $P$, then they are equal, except in the case $G = {\rm SO}_{4n}$, $P(-1)P(1)\neq 0$ and $c_1$ and $c_2$ 
are conjugate under ${\rm O}_{4n}(\overline{\Q})/{\rm SO}_{4n}(\overline{\Q}) \simeq \Z/2\Z$. In particular, for $P \in \mathcal{P}(G)$
the fiber\footnote{Beware that the map ${\rm char}$ is not surjective in general. For instance, Corollary \ref{coro:disc} shows that 
for any prime $p \equiv 1 \bmod 4$, there is no order $p$ element in ${\rm SO}_{p-1}(\Q)$, as $\Phi_p(1)\Phi_p(-1)=p$ is not a square.}
${\rm char}^{-1}(P)$ has at most $1$ element if $G \neq {\rm SO}_{4n}$ or $P(-1)P(1)=0$, and $0$ or $2$ elements otherwise.
The following elementary lemma (see \cite[Remark 3.2.11]{Taibi_dimtrace})
shows that this latter case does not create
complications:

\begin{lemm} \label{lem:mcouter}For $G=\SO_{4n}$ and $c, c' \in \mathcal{C}(G)$ with $P_c=P_{c'}$,  we have $\mathrm{m}_c = \mathrm{m}_{c'}$.
\end{lemm}

Thus we may write
\[ \mathrm{T}_{\mathrm{ell}}(G; \lambda) = \sum_{P \in
\mathcal{P}(G)} \mathrm{m}_P \, \tr(P; \lambda) \]
with
\[ \mathrm{m}_P := \begin{cases}
    \mathrm{m}_c & \text{ if there is} \,\,c \in \mathcal{C}(G) \text{ with } \text{char}(c)=P
    \\
    0 & \text{ if } P \text{ does not belong to }
    \text{char}(\mathcal{C}(G))
\end{cases} \]
and
\[ \tr(P; \lambda) := \begin{cases}
  \tr( c \,|\, {\rm V}_{\lambda} ) & \text{ if } G \neq
    \SO_{4n} \text{ or } P(1)P(-1)=0 \\
  \tr( c \,|\, {\rm V}_{\lambda}) + \tr( c' \,|\, {\rm V}_{\lambda} ) & \text{
    otherwise}
\end{cases} \]
with ${\rm char}^{-1}(P)=\{c\}$ in the first case
and ${\rm char}^{-1}(P)=\{ c, c' \}$ in the second case.
This also implies $\mathrm{T}_{\mathrm{ell}}(\SO_{4n};
\theta(\lambda)) = \mathrm{T}_{\mathrm{ell}}(\SO_{4n}; \lambda)$ where
$\theta$ is the non-trivial outer automorphism of $\SO_{4n}$ induced by ${\rm O}_{4n}(\Z)/{\rm SO}_{4n}(\Z)=\Z/2\Z$.
This invariance is fortunate also because Koike and Terada's simple (and most
importantly very effective for small weights) formulas \cite{Koike_Terada} for
traces in algebraic representations apply to irreducible representations of
symplectic and orthogonal (rather than \emph{special} orthogonal) groups. 
Equivalently, their formula gives $\tr(P; \lambda)$ in terms of $P$, but not
$\tr(c \,|\, {\rm V}_{\lambda})$ in the second case above if $\theta(\lambda) \neq
\lambda$.\ps

There is another obvious invariance property of masses. 
For $G={\rm SO}_{4n}$ or ${\rm Sp}_{2n}$, the element $-1$ of $G(\Z)$ is in the center of $G$, 
and $c \mapsto -c$ preserves $\mathcal{C}(G)$. Formula \eqref{eq:locintorg} thus shows:

\begin{lemm} \label{lem:mccenter}For $G={\rm SO}_{4n}$ or ${\rm Sp}_{2n}$, and $c \in \mathcal{C}(G)$, we have  $\mathrm{m}_{-c} = \mathrm{m}_c$.
\end{lemm}
For $G$ and $c$ as above, we have $P_{-c}(X)=(-1)^{\deg P_c} P_c(-X)$,
${\rm tr}(P_{-c};\lambda)=\lambda(-1) {\rm tr}(P_c;\lambda)$,  as well 
as ${\rm m}_{P_c}={\rm m}_{P_{-c}}$ by the lemma.
A consequence is that $\mathrm{T}_{\mathrm{ell}}(G; \lambda) = 0$ if
the restriction of $\lambda$ to the center ${\rm Z}(G)$ of $G$ is non-trivial.\footnote{
\label{foot:p1prime}
From the perspective of the strategy discussed in \S \ref{sub:intro_effortless}, 
this vanishing is in agreement with the vanishing of ${\rm N}^\bot(w(\lambda))$ for all
dominant weights $\lambda$ of ${\rm Sp}_{2n}$ or ${\rm SO}_{4n}$ such that $\lambda(-1)=-1$,
which is a consequence of the property $\epsilon(\pi)=\varepsilon({\rm L}(\pi_\infty))=1$ for orthogonal $\pi$ in $\Pi_{\rm alg}$:
see  \cite[Thm. 1.5.3]{Arthur_book} and \cite[Prop. 1.8]{ChRe}.}
Define the following equivalence relation on $\mathcal{P}(G)$: $P_1
\sim P_2$ if $P_1 = P_2$ or $P_2(X) = (-1)^{\nG} P_1(-X)$. Assuming $\lambda_{|{\rm Z}(G)}=1$ we may thus finally write 
\begin{equation} \label{finalformtell}  \mathrm{T}_{\mathrm{ell}}(G; \lambda) = \sum_{P \in
\mathcal{P}(G)/\sim} {\rm e}_P \,\mathrm{m}_P \, \tr(P; \lambda) \end{equation}
where ${\rm e}_P \in \{1,2\}$ denotes the size of the equivalence class of $P$.
\ps\ps

To sum up, for the purpose of implementing our strategy introduced in \S
\ref{sub:intro_effortless} we can replace the indexing set
$\mathcal{C}(G)$ by $\mathcal{P}(G)/\sim$, which is
computable, and we may as well restrict to dominant weights $\lambda$ such that
$\lambda|_{{\rm Z}(G)} = 1$, and even to a set of representatives
for the orbits under $\{1,\theta\}$ in the even orthogonal case. \ps


\begin{table}[htp] \centering
\renewcommand*{\arraystretch}{1.2}
\caption{}
\makebox[\linewidth]{
\begin{tabular}{ccccccccc}
  $G$ & $\Sp_2$ & $\Sp_4$ & $\Sp_6$ & $\Sp_8$ & $\Sp_{10}$ &
  $\Sp_{12}$ & $\Sp_{14}$ & $\Sp_{16}$ \\
  $|\mathcal{P}(G)/\sim|$ & $3$ & $12$ & $32$ & $92$ & $219$ & $530$ & $1157$ &
  $2521$ \\
\end{tabular}}
\end{table}

\subsection{Conjugacy classes and characteristic polynomials in the orthogonal
case: spinor norms}
\label{sec:spinornorm}

As announced in Remark \ref{rem:spin_norm_intro}, it turns out that in the
orthogonal cases we can further reduce the set parametrizing conjugacy classes.
Let $\mathcal{C}_0(G) \subset \mathcal{C}(G)$ be the subset of equivalence classes 
containing a finite order element in $G(\Q)$ whose $G(\A_f)$-conjugacy class meets
$G(\Zhat)$.
In particular $\Ccal_0(G)$ contains the set of $c \in \Ccal(G)$ such that
$\mathrm{m}_c \neq 0$.
A priori it may happen that $\mathcal{C}_0(G)$ is smaller
than $\mathcal{C}(G)$.
Using the analysis in \cite[\S 3.2.2]{Taibi_dimtrace} and Jacobson's hermitian
analogue of the Hasse-Minkowski theorem \cite{JacobsonHerm}, one can argue that
$\mathcal{C}_0(\Sp_{2n}) = \mathcal{C}(\Sp_{2n})$ for any $n
\geq 1$.
Since this fact is rather unfortunate for our strategy, we leave the details to
the interested reader.

We now focus on special orthogonal groups.
Proposition \ref{prop:restr_char_poly} below gives an explicit subset
$\mathcal{P}_1(G)$ of $\mathcal{P}(G)$ such that its
preimage $\mathcal{C}_1(G) \subset \mathcal{C}(G)$ under the map char \eqref{eq:mapchar} contains
$\mathcal{C}_0(G)$.
In contrast with the symplectic case we will see that $\mathcal{C}_1(G)
\subsetneq \mathcal{C}(G)$ in general, owing to the fact that special
orthogonal groups are not simply connected.

\begin{rema}
  The second author had already observed that there was such a restriction on
  classes $c$ satisfying $\mathrm{m}_c \neq 0$ in \cite[Remark
  3.2.8]{Taibi_dimtrace}, unfortunately without giving details or proofs \dots
  He was also unaware of related previous work of Gross and McMullen:
  \cite[Theorem 6.1]{Gross_McMullen} is similar to Proposition
  \ref{prop:restr_char_poly}.
  Unfortunately we could not deduce Proposition \ref{prop:restr_char_poly} from
  the results of \cite{Gross_McMullen}, so we give a slightly different proof
  below, relying on the Zassenhaus formula for spinor norms.
\end{rema}

Let $R$ be a commutative ring, $V$ a projective $R$-module of finite constant rank $n$ and $q: V \rightarrow R$ 
a quadratic form. We say that $q$ is {\it non-degenerate} if the associated $R$-bilinear form $\beta_q(x,y)=q(x+y)-q(x)-q(y)$
is a perfect pairing on $V$. We say that  $q$ is {\it regular} if either $q$ is non-degenerate, or $n$ is odd and Zariski-locally on $R$
the half-discriminant of $q$ is invertible: see \cite[Ch. IV \S 3]{Knus} (who rather uses the terminology {\it semi-regular} in this case). 
When $V$ is a 
free $R$-module and $q$ is non-degenerate, we denote by $\disc(q) \in R^{\times}/R^{\times,2}$ the class of the determinant 
of a Gram matrix of $\beta_q$, where $R^{\times,2}$ denotes the subgroups of squares in $R^\times$.\ps

First we recall a few definitions from \cite[Appendix C]{Conrad_luminy} or \cite[Ch. IV]{Knus}.
Assume $n\geq 3$ and $q$ regular. Associated to $(V, q)$ are (reductive) groups schemes $\Spin(V, q) \subset
\GSpin(V, q)$ and $\SO(V, q)$ over $R$.
The group $\GSpin(V, q)$ is the group of even degree invertible elements
in the Clifford algebra ${\rm C}(V, q)$ which stabilize the submodule $V \subset {\rm C}(V,
q)$ under conjugation.
This conjugation action gives a morphism $\pi : \GSpin(V,q) \rightarrow
\SO(V,q)$, with kernel the central $\GL_1$ (invertible scalars in the Clifford
algebra).
See {\it e.g.} Propositions C.2.8 and C.4.6 of \cite{Conrad_luminy} for these properties and the fact
that $\pi$ factors through the special orthogonal group.
The \emph{Clifford norm} morphism $\nu : \GSpin(V,q) \rightarrow \GL_1$ is
defined in (C.4.2) and (C.4.4) {\it loc.\ cit.}
The restriction of $\nu$ to the central $\GL_1$ is $t \mapsto t^2$.
The group $\Spin(V,q)$ can be defined as the kernel of the Clifford norm:
see the proof of Lemma C.4.1 {\it loc.\ cit.}\ for the case $n$ even and the proof of
Proposition C.4.10 loc.\ cit.\ and the paragraph following it for the case $n$
odd. We have \cite[(6.2.3) p.231]{Knus} an exact sequence of sheaves in groups on the
Zariski site of $R$
\begin{equation} \label{eq:surjgspin} 1 \rightarrow \GL_1 \rightarrow \GSpin(V, q) \rightarrow
\SO(V, q) \rightarrow 1, \end{equation}
and thus an exact sequence of sheaves in groups on the fppf site of $R$
\[ 1 \rightarrow \mu_2 \rightarrow \Spin(V, q) \rightarrow
  \SO(V, q) \rightarrow 1. \]
The (not so) long exact sequence in cohomology associated to the second short
exact sequence above gives the spinor norm $\sn: \mathrm{SO}(V,q) \rightarrow
H^1_{\mathrm{fppf}}(R, \mu_2)$.
If $\mathrm{Pic}(R) = 1$, which will always be the case in this paper, the fppf
exact sequence $1 \rightarrow \mu_2 \rightarrow \GL_1 \rightarrow \GL_1 \rightarrow 1$ gives
the isomorphism $H^1_{\mathrm{fppf}}(R, \mu_2) \simeq R^\times / R^{\times, 2}$, 
and we will implicitly consider the spinor norm in this last group.
The spinor norm of $\gamma \in \mathrm{SO}(V,q)(R)$ is then represented by
$\nu(\tilde{\gamma})$ where $\tilde{\gamma} \in \mathrm{GSpin}(V,q)(R)$ is any
lift of $\gamma$; such a lift exists by \eqref{eq:surjgspin} and $\mathrm{Pic}(R) = 1$.
The spinor norm is additive: if $(V, q) \simeq (V_1, q_1) \perp (V_2, q_2)$ and
$\gamma \in \mathrm{SO}(V,q)(R)$ stabilizes $V_1$ and $V_2$ then we have
$\sn \gamma = \sn \gamma|_{V_1} \times \sn \gamma|_{V_2}$. \ps

\begin{theo}[Zassenhaus] \label{theo:Zassenhaus}
  Let $k$ be a field of characteristic different from $2$.
  Let $V$ be a finite-dimensional vector space over $k$, endowed with a
  non-degenerate quadratic form $q$.
  Let $\gamma \in \mathrm{SO}(V,q)(k)$ and write the characteristic polynomial of
  $\gamma$ as $(X-1)^a (X+1)^{2b} Q(X)$ with $Q(1)Q(-1) \neq 0$.
  Then the spinor norm $\sn \gamma$ of $\gamma$ is represented by $\disc(
  q\,| \, \ker\,(\gamma + 1)^{2b}) \,Q(-1)$ in $k^\times / k^{\times, 2}$.
\end{theo}


\begin{proof}
  This is just a reformulation of the main theorem of \cite{Zassenhaus_spin}.
  Here is a short argument for the convenience of the reader.
  Using the orthogonal decomposition $V \,=\, \ker\, (\gamma-1)^a \, \perp \ker \, (\gamma+1)^{2b} \,\perp\, \ker Q(\gamma)$,
  and the additivity of spinor norms, we may assume $a=0$, so $\dim V \equiv 0 \bmod 2$, and either $Q=1$ or $b=0$.
  In the case $Q=1$, we have $\sn \gamma = \sn (-1)$ as unipotent elements are squares,
  and we conclude by the classical formula $\sn (-1) \,=\, 2^{\dim V} \,\disc(q)$ (use an orthogonal basis of $V$). 
  The arguments so far have used that the characteristic of $k$ is $\neq 2$, but the following ones will not.

   Assume $b=0$, write $Q(X) = \prod_{i=1}^n (X-t_i)(X-t_i^{-1})$ in $\overline{k}[X]$, 
   and choose $\tilde{\gamma} \in \mathrm{GSpin}(V, q)(k)$ a lift of $\gamma$. 
  Write $\tilde{\gamma}=du=ud$ its Jordan decomposition in 
  $\mathrm{GSpin}(V, q)(\overline{k})$, with $d$ semi-simple and $u$ unipotent.
  There is a decomposition 
  $V \otimes_k \overline{k} = P_1 \perp \cdots \perp P_n$, 
  where each $P_i$ is a $d$-stable hyperbolic plane on which the two eigenvalues 
  of $d$ are $t_i^{\pm 1}$. Using the natural isomorphism 
  between ${\rm C}(V,q) \otimes_k \overline{k}$ 
  and the graded tensor product of the Clifford algebras of the $P_i$ 
  (see e.g. \cite[IV. Prop. 1.3.1]{Knus}) we easily sees that there is a pair 
  $(s,\lambda)$ in $\overline{k}^\times \times \overline{k}^\times$ 
  such that:  $s^2 = t_1 \dots t_n$, the Clifford norm of $d$ (or equivalently, of $\widetilde{\gamma}$) is $\lambda^2$,
  and the trace of $d$  (or equivalently, of $\widetilde{\gamma}$) in the spin representation of ${\rm GSpin}(V,q)(k)$ is $\lambda s
  \prod_{i=1}^n (1 + t_i^{-1})$. 
  The spinor norm of $\gamma$ is thus represented by $\lambda^2 \in k^\times$.
  Note that although the spin representation may not be defined over $k$, its trace is. 
  Indeed, this representation is defined as the tautological morphism $\mathrm{GSpin}(V, q)(k) \subset {\rm C}(V,
  q)^\times$ and ${\rm C}(V, q)$ is a central simple algebra over $k$, whose reduced trace is $k$-valued.   Since we have $1 + t_i^{-1} \neq 0$ for each $i$ as $Q(-1) \neq 0$, the spinor norm of $\gamma$ is
  represented by $s^2 \prod_{i=1}^n (1+t_i^{-1})^2 = \prod_{i=1}^n t_i (1+t_i^{-1})^2 =
  \prod_{i=1}^n (1+t_i)(1+t_i^{-1}) = (-1)^{2n} Q(-1)$.
\end{proof}

We deduce the following discriminant formula, which can also be proved directly
(see \cite[Proposition A.3]{Gross_McMullen}).

\begin{coro} \label{coro:disc}
  Under the same assumptions, assume moreover that $a=b=0$.
  Then $\disc(q) \in k^\times / k^{\times, 2}$ is represented by $Q(1)Q(-1)$.
\end{coro}
\begin{proof}
  The discriminant of $q$ is the spinor norm of $-\mathrm{Id}_V$, or equivalently of  $\gamma
  (-\gamma)$. We conclude by applying the theorem to $\gamma$ and $-\gamma$.
\end{proof}

\begin{prop} \label{prop:restr_char_poly}
  Let $V$ be a free $\Z$-module of rank $n$ endowed with a regular quadratic form $q$,
   $G={\rm SO}(V,q)$, and $P(X) = (X-1)^a (X+1)^{2b} Q(X)$ 
  a monic polynomial of degree $n$ in $\Q[X]$ with $Q(1)Q(-1) \neq 0$.
  Assume that for every prime $p$ there exists $\gamma_p \in G(\Zp)$ having
  characteristic polynomial $P$.
  If $b=0$ then the integer $|Q(-1)|$ is square, and if $a=0$ then the
  integer $|Q(1)|$ is a square.
\end{prop}

Note that the existence of $\gamma_p$ for all primes $p$ implies that $Q$ has
integer coefficients.

\begin{proof} 
  Fix a prime $p$. Since $\gamma_p \in G(\Zp)$ and $\operatorname{Pic} \Zp = 1$ the element
  $\gamma_p$ can be lifted to an element of $\GSpin(V,q)(\Zp)$ by \eqref{eq:surjgspin}, so the
  spinor norm of $\gamma_p$ lies in the image of $\Zp^{\times}$ in $\Qp^{\times}
  / \Qp^{\times, 2}$.  Together with Theorem \ref{theo:Zassenhaus}, this implies that
 $$\disc(q \,|\, \ker(\gamma_p + 1)^{2b}) \times Q(-1) \, \, \in \Qp^\times/\Qp^{\times,2}$$
   lies as well in the image of $\Zp^{\times}$. Assuming $b=0$, it follows that the integer $Q(-1)$ has  
   an even valuation at each prime $p$, so $|Q(-1)|$ is a square. Assume now $a=0$. In particular, $n$ is
  even and we have $\disc(q) \in \Z^{\times}$. By Corollary \ref{coro:disc} applied to the orthogonal of 
  $\ker(\gamma_p + 1)^{2b}$ in $V \otimes\Qp$, we have 
  $\disc(q)\times\disc(q \,|\, \ker(\gamma_p+1)^{2b}) \equiv Q(-1) Q(1)$ in $\Qp^\times/\Qp^{\times, 2},$ or equivalently:
  $$\disc(q) \times \disc(q \,|\, \ker(\gamma_p+1)^{2b}) \times Q(-1)  \equiv Q(1) \mod \Qp^{\times, 2}$$
But have seen that the left-hand side is in the image of $\Zp^\times$. So the integer $Q(1)$ has  
  an even valuation at each prime $p$, and $|Q(1)|$ is a square.\end{proof}

The astute reader will notice that the second statement could also be proved by
considering $(-1)^{\deg P}P(-X)$ and $(- \gamma_p)_p$.
In odd dimension this would require the introduction of $\mathrm{GPin}$ and
$\mathrm{Pin}$ groups to define the spinor norm (see complications in
\cite[Appendix C]{Conrad_luminy}).

\begin{defi}
  For $G = \SO_{2n+1}$ or $\SO_{4n}$ let
  $\mathcal{P}_1(G)$ be the subset of $\mathcal{P}(G)$
  consisting of all polynomials of the form
  $(X-1)^a(X+1)^{2b}Q(X)$, where $Q(X)$ is a product of cyclotomic polynomials
  $\Phi_m$ with $m \geq 3$, which satisfy
  \begin{itemize}
    \item $b>0$ if the positive integer $Q(-1)$ is not a square, and
    \item $a>0$ if the positive integer $Q(1)$ is not a square.
  \end{itemize}
  For $G = \Sp_{2n}$ denote $\Pcal_1(G) = \Pcal(G)$ (Definition
  \ref{def:PG}).
\end{defi}

The positive integers $\Phi_m(\pm 1)$ for $m\geq 3$ may be computed inductively in terms of
the prime decomposition of $m$: see \cite[Theorem 7.1]{Gross_McMullen}. In terms of the notation ${\rm m}_P$ 
defined in \S \ref{elementaryconjclass}, Proposition \ref{prop:restr_char_poly} asserts:

\begin{coro}\label{cor:constraint} If $P \in \mathcal{P}(G)$ satisfies ${\rm m}_P \neq 0$, then we have $P \in \Pcal_1(G)$. 
\end{coro}

This constraint is very useful, particularly in the even case, as the following
table shows. In practice, we will see that it is almost sharp: see Remark \ref{rem:P1sharp}.
Observe that $\mathcal{P}_1(G)$ is stable under the 
equivalence relation $\sim$ introduced in the end of \S \ref{elementaryconjclass}.

{\small
\begin{table}[htp] \label{tab:card_cc_SO} \centering
\renewcommand*{\arraystretch}{1.2}
\caption{}
\makebox[\linewidth]{
\begin{tabular}{ccccccccccccc}
  $G$ & $\SO_3$ & $\SO_4$ & $\SO_5$ & $\SO_7$ & $\SO_8$ &
  $\SO_9$ & $\SO_{11}$ & $\SO_{12}$ & $\SO_{13}$ & $\SO_{15}$ &
  $\SO_{16}$ & $\SO_{17}$ \\
  $|\Pcal(G) / \sim|$ & $5$ & $12$ & $19$ & $59$ & $92$ & $165$ & $419$ &
  $530$ & $1001$ & $2257$ & $2521$ & $4877$ \\
  $|\Pcal_1(G) / \sim|$ & $3$ & $6$ & $12$ & $34$ & $40$ & $99$ & $244$ &
  $211$ & $598$ & $1339$ & $992$ & $2948$
\end{tabular}}
\end{table}
}

\subsection{Non-existence of level one regular algebraic automorphic cuspidal
representations} \label{nonexisteffortless}
In this paragraph, we prove Theorems  \ref{theomasseseffortless} 
and \ref{theomassesmixed} of the introduction.
To implement the strategy explained in \S \ref{sub:intro_effortless}, 
 taking into account the reduced Formula \eqref{finalformtell}
and Corollary \ref{cor:constraint}, it remains to actually produce, for as many
``small rank'' groups $G$ as possible in the families $(\SO_{2n+1})_{n \geq 1}$, $(\Sp_{2n})_{n \geq 1}$ and 
$(\SO_{4n})_{n \geq 1}$,
sets $\Lambda$ of dominant weights for $G$ satisfying the following properties:
\begin{itemize}\ps
  \item[(P1')]
    For all $\lambda \in \Lambda$ we have $\lambda_{|{\rm Z}(G)}=1$ and $\mathrm{N}^\perp(w(\lambda)) = 0$.\ps\ps
  \item[(P2')] 
    For $\mathcal{P}=\mathcal{P}_1(G)/\sim$, the $\Lambda \times \mathcal{P}$ matrix $({\rm e}_P \, \tr(P; \lambda))_{\lambda \in \Lambda, P \in \mathcal{P}}$ has rank $|\mathcal{P}|$.\ps
\end{itemize}

\noindent Of course (P2') implies $|\Lambda| \geq |\mathcal{P}|$ so our aim is
roughly to produce as many dominant weights satisfying (P1') as possible. 
See also Footnote \ref{foot:p1prime} for an important remark regarding the condition on $\lambda_{|{\rm Z}(G)}$ in (P1').\ps

{\sc Notations:} For $w\geq 0$ an integer we denote by $\Lambda_G(w)$ the (finite) set of all dominant weights $\lambda$ of $G$ such that: $2 w(\lambda)_1 \leq w$, 
$\lambda_{|{\rm Z}(G)}=1$, as well as $\lambda_m \geq 0$ in the case $G={\rm SO}_{2m}$.
For a dominant weight $\lambda$ of $G$, there is a unique effective
element $U(\lambda) \in \mathrm{K}_\infty$ with $\det U(\lambda) = 1$ and such that the multi-set of weights of $U(\lambda)$ 
(as defined in \S \ref{realalgebraic}) coincides with $w(\lambda)$ (viewed of course as the multi-set $\{ w(\lambda)_i\, \, |\, \, 1 \leq i \leq {\rm n}_{\widehat{G}}\}$).
The representation $U(\lambda)$ is multiplicity free, and for any $\pi \in
\Pialg$ having weights $w(\lambda)$ we have $\mathrm{L}(\pi_\infty) =U(\lambda)$.  \ps

In order to produce $\Lambda$ we will first use the inequality \eqref{basicineq} in Corollary
\ref{basiccor}. We choose $w$ big enough, and for every $\lambda \in {\rm \Lambda}_G(w)$, and for all
parameters $\ell \in \frac{1}{4} \Z \cap [1/2, 20]$, we compute
$\widehat{\mathrm{F}_\ell}(i/4\pi)-{\rm B}_{\infty}^{\mathrm{F}_\ell}(U(\lambda), U(\lambda))$.
Whenever we find a negative value (certified using interval arithmetic as
explained in \S \ref{numericaleval}), we know that $\mathrm{N}^\perp(w(\lambda))
= 0$ by Corollary \ref{basiccor} and thus we add $\lambda$ to $\Lambda$. In other words, we choose
$$\Lambda_w^{\rm test} := \{ \lambda \in \Lambda_G(w) \,\,|\,\, \exists \ell \in  \frac{1}{4} \Z \cap [1/2, 20], \, \, \widehat{\mathrm{F}_\ell}(i/4\pi) < {\rm B}_{\infty}^{\mathrm{F}_\ell}(U(\lambda), U(\lambda))\},$$
for the set $\Lambda$. For our purpose this simple method is already very effective.
Table \ref{tab:setlambda} below displays all groups for which it works, {\it i.e.}\ for which, 
using the given $w$, the set $\Lambda_w^{\rm test}$ satisfies the rank condition (P2').

\begin{table}[htp] \label{tab:setlambda} \centering
\renewcommand*{\arraystretch}{1.2}
\caption{}
\makebox[\linewidth]{
\begin{tabular}{cccc}
$G$ & $|\Pcal_1(G)/\sim|$ & $|\Lambda_w^{\rm test}|$ & $w$ 
  \\
$\SO_3$    & $3$    & $3$    & ${}5$ \\
$\Sp_2$    & $3$    & $4$    & ${}14$ \\
$\SO_5$    & $12$   & $44$   & ${}23$ \\
$\SO_4$    & $6$    & $30$   & ${}26$ \\
$\Sp_4$    & $12$   & $28$   & ${}28$ \\
$\SO_7$    & $34$   & $183$  & ${}27$ \\
$\Sp_6$    & $32$   & $97$   & ${}28$ \\
$\SO_9$    & $99$   & $498$  & ${}27$ \\
$\SO_8$    & $40$   & $335$  & ${}28$ \\
$\Sp_8$    & $92$   & $255$  & ${}28$ \\
$\SO_{11}$ & $244$  & $923$  & ${}27$ \\
$\Sp_{10}$ & $219$  & $446$  & ${}28$ \\
$\SO_{13}$ & $598$  & $1294$ & ${}27$ \\
$\SO_{12}$ & $211$  & $1061$ & ${}28$ \\
$\Sp_{12}$ & $530$  & $597$  & ${}28$ \\
$\SO_{15}$ & $1339$ & $1924$ & ${}35$ \\
\end{tabular} }
\end{table}

See \cite{homepage} for the Sage program which checks that each set $\Lambda$ in
the table satisfies (P1') (using Corollary \ref{basiccor} and interval
arithmetic) and inductively computes, for each group $G$ appearing in the table:\ps
\begin{enumerate}
  \item
    all masses $(\mathrm{m}_P)_{P \in \mathcal{P}_1(G)}$,\ps
  \item
    $\mathrm{T}_\mathrm{ell}(G; \lambda)$, $\mathrm{EP}(G; \lambda)$ and
    $\mathrm{N}^\perp(w(\lambda))$ for all dominant weights $\lambda$ in any
    desired range (only limited by computer memory).
\end{enumerate}
\ps
Note that we obtain in particular, independently of the computation of masses in
\cite{ChRe} and \cite{Taibi_dimtrace}, the existence of $27$ self-dual
elements of $\Pialg$ having regular weight and motivic weight $\leq 24$.
In \S \ref{sec:mot2324} we will prove that there is no other such element in
$\Pialg$.
The only case which was obtained in \cite{Taibi_dimtrace} and that we cannot
recover using this much simpler method is $\Sp_{14}$.
The case of $\SO_{15}$ is new. For $G=\Sp_{14}$, considering all dominant weights $\lambda$ in $\Lambda_{G}(90)$, 
we only find a set $\Lambda_{90}^{\rm test}$ of cardinality $974$, whereas we have
$|\Pcal_1(\Sp_{14})/\sim| = 1157$.
Higher motivic weights do not seem to provide any new non-existence results.
Similarly this method does not allow us to conclude either in the case of
$\SO_{17}$.\ps\ps

To go further we use the algorithm explained in \S \ref{paralgo2} to find larger
sets $\Lambda$ satisfying (P1').
More precisely, for a large enough $w$ and each dominant weight $\lambda \in \Lambda_G(w)$ we applied this
algorithm with $U = U(\lambda)$, $\delta=1$, $m=1$ and taking for $\mathbf{S}$
\footnote{We actually have $|S|=26$, explained by the equality $\dim {\rm
S}_{23}({\rm SL}_2(\Z))=2$.}
the set of $27$ known elements of $\Pialg$ having motivic weight $\leq 24$ found above.
As before we try various $\ell \in \frac{1}{4} \Z \cap [1/2, 20]$.
Using this refined method we obtain the following three new cases:
\begin{itemize}
  \item For $G={\rm Sp}_{14}$ we have $|\mathcal{P}(G)/\sim| = 1157$ and we found a subset $\Lambda \subset \Lambda_{G}(36)$ of cardinality $1274$ satisfying (P1') and
    (P2').\ps
  \item The case $G=\SO_{16}$ is easier: we have
    $|\mathcal{P}_1(G)/\sim| = 992$ and we found a subset $\Lambda \subset \Lambda_{G}(28)$ of cardinality
     $1810$ satisfying (P1') and (P2').\ps
  \item For $G={\rm SO}_{17}$ we have $|\Pcal_1(G)| = 2948$ and we found a subset $\Lambda \subset \Lambda_{G}(63)$ of cardinality
     $3477$ satisfying (P1') and (P2'). (Restricting to $\Lambda_G(61)$ is not enough, as it yields a set of dominant
    weights of cardinality $3461$ which does not satisfy (P2')).\ps
\end{itemize}
Again our program checking rigorously that these sets satisfy (P1') and the
inductive computation of masses and of the numbers
$\mathrm{N}^\perp(w(\lambda))$ can be found at \cite{homepage}.
This concludes the proof of Theorem \ref{theomasseseffortless}. \ps

\begin{rema} \label{rem:P1sharp} Let $G$ be as in Theorem \ref{theomasseseffortless}.
An inspection of the masses found above shows
${\rm m}_P \neq 0$ for all $P \in \mathcal{P}_1(G)$, except for $6$ polynomials $P$
in the case $G={\rm SO}_{13}$, and $6$ others in the case $G={\rm SO}_{17}$. This shows that the spinor norms constraints
established in \S \ref{sec:spinornorm} are almost sharp. 
\end{rema}


This second method only gives us these three additional cases for which we can
compute all masses by solving a linear system.
For example the set $\Pcal_1(\Sp_{16})/\sim$ has $2521$ elements, but we are not even
close to producing $\Lambda$'s with enough elements: we were only able to produce a subset $\Lambda \subset \Lambda_{{\rm Sp}_{16}}(116)$ having $1427$ elements
 satisfying (P1').
To overcome this scarcity of dominant weights satisfying (P1'), we computed a lot
of masses for $\Sp_{16}$, namely for all $P$ in a certain subset
$\mathcal{P}(\Sp_{16})_{\mathrm{easy}}$ of $\mathcal{P}(\Sp_{16})$, using the
method of \cite{Taibi_dimtrace}, {\it i.e.}\ by computing orbital integrals directly,
and then we computed the remaining ones by solving a linear system.\ps
To describe the set $\mathcal{P}(\Sp_{16})_{\mathrm{easy}}$ explicitly, for $P
\in \mathcal{P}(\Sp_{16})$ and $p$ a prime  write
\[ P = \prod_m \prod_{k \in S(p,m)} \Phi_{m p^k}^{d(p,m,k)} \]
where the first product is over all integers $m$ coprime to $p$, $S(p,m) \subset
\Z_{\geq 0}$ and $d(p,m,k) \geq 1$.
Then $P \in \mathcal{P}(\Sp_{16})_{\mathrm{easy}}$ if and only if for any prime
number $p$ and any $m$ coprime to $p$ we have $|S(p,m)| \leq 2$ and
\[ \begin{cases}
  0 \in S(p,m) & \text{ if } p>2 \text{ and } |S(p,m)|=2, \\
  0 \in S(p,m) \text{ or } 1 \in S(p,m) & \text{ if } p=2 \text{ and }
  |S(p,m)|=2.
\end{cases} \]
For such a polynomial $P$ the computation using the method
explained {\it loc.\ cit.}\ of the orbital integrals \eqref{eq:locintorg} occurring in
the mass $\mathrm{m}_P$ is purely a combinatorial matter and does not require
any bilinear algebra.
To be more precise, in general computing an orbital integral using the method
loc.\ cit.\ involves enumerating totally isotropic subspaces stable under a
given unipotent automorphism $\gamma$ in (possibly degenerate) symplectic or
hermitian spaces $(V, \langle \cdot, \cdot \rangle)$ over a finite field,
enumerating isomorphisms between such triples $(V, \langle \cdot, \cdot \rangle,
\gamma)$, and/or computing the complete invariants attached by Wall \cite{Wall}
to isomorphism classes of such triples with $\langle \cdot, \cdot \rangle$
non-degenerate; we restrict to cases where no such computation is necessary.
Although these easier cases have the obvious benefit of being much easier to
implement, the second advantage here is that these orbital integrals are computed (by a
computer) in a matter of seconds.
In contrast, there are relevant orbital integrals for $\Sp_{16}$ for which the
implementation of \cite{Taibi_dimtrace} does not terminate in any reasonable
time.\ps

Denoting $\mathcal{P}(\Sp_{16}) = \mathcal{P}(\Sp_{16})_\mathrm{easy} \sqcup
\mathcal{P}(\Sp_{16})_\mathrm{hard}$, we have
$|\mathcal{P}(\Sp_{16})_\mathrm{hard}/\sim| = 766$.
We found a subset of dominant weights $\Lambda \subset \Lambda_G(36)$ for
$G=\Sp_{16}$ of cardinality $1086$ satisfying (P1') and the analogue of (P2') for
$\mathcal{P}(\Sp_{16})_\mathrm{hard}$.
This concludes the proof of Theorem \ref{theomassesmixed}.

\section{Classification results in motivic weights $23$ and $24$}\label{sec:mot2324}

This section is in the natural continuation of Sect. \ref{explicitformula}, of which we shall use freely the notations. 
We have decided to postpone it here because, in a few places, we will use below existence or inexistence results of certain self-dual regular 
elements of $\Pi_{\rm alg}$, results which have been proved in Sect. \ref{par:effortless}.\ps\ps

\subsection{Motivic weight $23$} \label{par:mot23}We now prove Theorem \ref{thm23} along with the following supplementary result. 

\ps

\begin{prop} \label{prop:23_mult}
  Let $U$ be an effective element of ${\rm K}_{\infty}^{\leq 23}$ containing
  $\mathrm{I}_{23}$ with multiplicity $\geq 2$.
  Let $T$ be the subset of elements $\pi$ in $\Pialg$ with ${\rm
  L}(\pi_{\infty})=U$.
  \begin{enumerate}\ps
    \item
      If $|T| \geq 2$ then we have  $U = {\rm I}_1 + {\rm I}_7 + {\rm I}_{13} +
      {\rm I}_{17} + {\rm I}_{21} +\, 2\, \, {\rm I}_{23}$ and $T = \{ \pi,
      \pi^\vee \}$ for some non-self-dual $\pi$.\ps
    \item
      If $T=\{\pi\}$ then $U$ belongs to an explicit set of $181$ elements and
      $\pi$ is of symplectic type.
  \end{enumerate}
\end{prop}

The set of $181$ possible $U$ mentioned above can be found in \cite{homepage}. They all satisfy $14 \leq \dim U \leq 42$.

\begin{proof}[Proof of Theorem \ref{thm23} and Proposition \ref{prop:23_mult}]
  For $\ell = 9.74$ the restriction of the symmetric bilinear form
  $\widehat{{\rm F}_{\ell}}(i/4 \pi)^{-1} {\rm B}_{\infty}^{{\rm F}_\ell}$ to
  ${\rm K}_\infty^{\leq 23}$ is positive definite.
  As explained in Remark \ref{rema:qfminim}, using interval arithmetic we obtain
  rational lower bounds (we take them in $10^{-6} \Z$) for the coefficients of
  its Gram matrix in the basis ${\rm I}_1,\dots,{\rm I}_{23}$.
  Applying the Fincke-Pohst algorithm, we obtain the set $\mathcal{U}_2$ of all
  $265$ effective elements $U$ in ${\rm K}_{\infty}^{\leq 23}$ containing ${\rm
  I}_{23}$ and satisfying ${\rm B}_\infty^{{\rm F}_\ell}(U, U) \leq
  \widehat{{\rm F}_\ell}(i/4 \pi)/2$.
  By Corollary \ref{basiccor}, $\mathcal{U}_2$ contains all the elements $U$
  such that there exist two distinct elements $\pi_1, \pi_2$ in $\Pi_{\rm alg}$
  of motivic weight $23$ and with ${\rm L}((\pi_1)_\infty) = {\rm L}((\pi_2)_\infty) = U$.
  \ps\ps

  For each $U$ in $\mathcal{U}_2$, we systematically applied Algorithm
  \ref{paralgo2} to $U$, $\delta=0$, $m=2$ and to the set ${\bf S}$ of $27$
  known elements of $\Pialg$ having motivic weight $\leq 24$, and various
  $\ell$.
  For all but one $U$, namely the one of assertion (1), it led to a
  contradiction with Inequality \eqref{eq:ineq_CF}.
  Let us be more precise about the choices of $\ell$ and of the subset $S'
  \subset S$ that we can make {\it a posteriori} in order to reach these
  contradictions more quickly (see also the source code \cite{homepage} for a
  working sheet).
  We first replace for the rest of the proof the ${\bf S}$ above by its subset
  whose elements have motivic weight $\leq 23$.
  We now have $|{\bf S}|=24$ and $|S|=23$.
  If we apply Algorithm \ref{paralgo2} with all $\ell$ in $[3,12]\cap \Z$ and
  all subsets $S'\subset S$ with $1 \leq |S'|\leq 2$, we obtain in a few seconds
  (on a personal computer) a contradiction for all but $12$ elements of
  $\mathcal{U}_2$.
  Using then all $\ell$ in $[7,9.5]\cap \frac{1}{2}\Z$ and all $S' \subset S$
  with $3 \leq |S'| \leq 4$ for those $12$ elements, the algorithm finds again
  in a few seconds a contradiction in all but $6$ cases.
  Two of these six are eliminated in a minute about using $\ell=11$ and all the
  $33649$ subsets $S'$ with $|S'|=5$.
  The remaining $4$ elements have the form $U=U'+{\rm I}_{21}+\,2\, {\rm
  I}_{23}$  with $U'$ in the following list: ${\rm I}_3 + {\rm I}_7 + {\rm
  I}_{13}+ {\rm I}_{17}, \,\,\,  {\rm I}_1 + {\rm I}_5 + {\rm I}_9 + {\rm
  I}_{11}+ {\rm I}_{15} + {\rm I}_{17} + {\rm I}_{19},\,\,\, {\rm I}_1 + {\rm
  I}_7 + {\rm I}_{11}+ {\rm I}_{15} + {\rm I}_{19},\,\,\, {\rm I}_1 + {\rm I}_7
  + {\rm I}_{13}+ {\rm I}_{17}$.
  In the case $U'={\rm I}_1 + {\rm I}_7 + {\rm I}_{11}+ {\rm I}_{15} + {\rm I}_{19}$ we use 
$\ell=8.75$ and all the $100947$ subsets $S'$ with $|S'|=6$. To give an example, the algorithm produces in about $2$ minutes a linear combination close to\par
{\tiny $$x \,\,= \,\,0.860\,\frac{1}{\sqrt{2}} (\pi_1 + \pi_2) + 0.0834 \,\,1\,\, + 0.150 \,\, \Delta_{11} \,\,+\,\, 0.108\,\, \Delta_{15}\,\, + \,\,0.335\,\, \Delta_{19,7}\,\, +\,\, 0.172\,\, \Delta_{23,7}\,\, +\,\, 0.280 \,\,\Delta_{23,15,7}$$
}
\par \noindent with ${\rm C}^{{\rm F}_{8.75}}(x,x) = -0.0023$ up to $10^{-4}$. In the case $U'={\rm I}_1 + {\rm I}_5 + {\rm I}_9 + {\rm I}_{11}+ {\rm I}_{15} + {\rm I}_{17} + {\rm I}_{19}$, we use similarly $\ell=11.75$ and $|S'|=6$. The case 
$U'={\rm I}_3 + {\rm I}_7 + {\rm I}_{13}+ {\rm I}_{17}$ is quite harder to discard. After many tries, we found a contradiction using $\ell=10.25$ and a certain $11$ element subset $S'$ of $S$: see the source code in \cite{homepage} for the details. So far, we have thus proved the following:\ps\ps
  
(a)  For any $U \neq  {\rm I}_1 + {\rm I}_7 + {\rm I}_{13} + {\rm I}_{17}
      + {\rm I}_{21} +\, 2\, \, {\rm I}_{23}$ there is at most one element $\pi$
  of $\Pialg$ with motivic weight $23$ and $\rmL(\pi_\infty) = U$. In particular any such $\pi$ is self-dual.\ps\ps
  \noindent Despite our efforts, we could not find a contradiction in the case
  of the last element $U = {\rm I}_1 + {\rm I}_7 + {\rm I}_{13} + {\rm I}_{17} +
  {\rm I}_{21} +\, 2\, \, {\rm I}_{23}$.
  We have however ${\rm B}_\infty^{F_\ell}(U, U) > \widehat{{\rm F}_\ell}(i/4
  \pi)/3$ for $\ell = 9.74$.
  By Corollary \ref{basiccor}, this shows: \ps\ps
  
  (b) For $U={\rm I}_1 + {\rm I}_7 + {\rm I}_{13} + {\rm I}_{17} + {\rm I}_{21}
  +\, 2\, \, {\rm I}_{23}$, there are at most $2$ elements $\pi$ of $\Pialg$ of
  motivic weight $23$ and with ${\rm L}(\pi_\infty)=U$.\ps\ps
      
  Note that we have proved assertion {\it (1)} except for the non self-duality
  assertion.
  To go further, we determine the set of effective elements $U$ in ${\rm K}_{\infty}^{\leq 23}$ containing ${\rm I}_{23}$
  and satisfying ${\rm B}_\infty^{{\rm F}_\ell}(U, U) \leq \widehat{{\rm F}_\ell}(i/ \pi)$ for $\ell=9.74$. For this we proceed as in the first paragraph of the proof and obtain an explicit set $\mathcal{U}_1$ with $12230$ elements. By Corollary \ref{basiccor}, $\mathcal{U}_1$ contains all the elements $U$
  such that there exists $\pi$ in $\Pi_{\rm alg}$ of motivic weigth $23$ with
  ${\rm L}(\pi_\infty) = U$. For each $U$ in $\mathcal{U}_1$ we applied
  Algorithm \ref{paralgo2} to $U$, $\delta=1$ (we restrict to self-dual elements), $m=1$, 
to the set ${\bf S}$ of all $27$ known elements of $\Pialg$ having motivic weight $\leq 24$,
for various choices of $\ell$ and subsets $S'\subset S$. We obtained contradictions with Inequality \eqref{eq:ineq_CF} for all but
 $187$ elements of $\mathcal{U}_2$. We refer to \cite{homepage} for an explicit list of $12293-187=12106$ triples $(\ell,S',\underline{t})$ leading 
 to a contradiction in each case (checked using interval arithmetic). It would be tedious to explain here in details which $\ell$ and $S'$ we did choose to find these triples: this is unnecessary anyway as all the necessary information for our proof is contained in the aforementioned list! We nevertheless refer to \cite{homepage} for the log file of our computations (which took several months).  \ps\ps
%
 Among the $187$ aforementioned ``resistant'' elements of $\mathcal{U}_2$, six of them are multiplicity free:
 ${\rm I}_3+{\rm I}_{11}+{\rm I}_{17}+{\rm I}_{21}+{\rm I}_{23}, {\rm I}_7+{\rm I}_{15}+{\rm I}_{21}+{\rm I}_{23}, {\rm I}_3+{\rm I}_{9}+{\rm I}_{15}+{\rm I}_{21}+{\rm I}_{23}, {\rm I}_1+{\rm I}_{9}+{\rm I}_{17}+{\rm I}_{21}+{\rm I}_{23},  {\rm I}_5+{\rm I}_{13}+{\rm I}_{17}+{\rm I}_{21}+{\rm I}_{23}$ and ${\rm I}_5+{\rm I}_{13}+{\rm I}_{19}+{\rm I}_{23}$. These six regular weights have dimension $\leq 10$, and we know from the results of  Sect. \ref{par:effortless} that there is no self-dual $\pi$ with these Archimedean components. An inspection of the list $\mathcal{V}$ of remaining $181$ elements reveals that for any $U$ in $\mathcal{V}$:
\ps
\begin{itemize} 
 \item[(i)] $U$ contains ${\rm I}_{23}$ with multiplicity $\geq 2$, \ps
 \item[(ii)] $U$ contains ${\rm I}_w$ for some $w \in \{1,3,5\}$,\ps
 \item[(iii)] for any $w \in \{1,3,5,7,9\}$, the multiplicity of ${\rm I}_w$ in $U$ is at most one.\ps
 \end{itemize}
  \ps
  Assertion (i) concludes the proof of Theorem \ref{thm23}.
  Assertion (b) above and the fact that ${\rm I}_1 + {\rm I}_7 + {\rm I}_{13} +
  {\rm I}_{17}+ {\rm I}_{21} +\, 2\, \, {\rm I}_{23}$ is not in $\mathcal{V}$
  imply assertion {\it (1)} of Proposition \ref{prop:23_mult}.
  By (ii) and (iii) above, for any $U$ in $\mathcal{V}$ there is some ${\rm
  I}_w$ which occurs in $U$ with multiplicity $1$.
  In particular, such a $U$ has no ${\rm W}_\R$-equivariant nondegenerate
  symmetric pairing.
  This shows that any self-dual $\pi$ with ${\rm L}(\pi_\infty)=U$ is of
  symplectic type by \cite[Theorem 1.4.2]{Arthur_book}, and proves assertion
  {\it (2)} of Proposition \ref{prop:23_mult}.
\end{proof}

\begin{rema} For a given $(U,\delta,m)$, it seems hard to us to guess {\it a priori} what will be the best choices of $\ell$ and ${\bf S}$ (or $S'$) 
to plug into Algorithm \ref{paralgo2} for the purpose of reaching a contradiction with Inequality \eqref{eq:ineq_CF}. Athough the authors have developped their own intuition and artisanal methods to find good $\ell$ and ${\bf S}$, they are mostly based on numerical experiments. 
In the same vein, in the cases where we did not find any contradiction, it seems difficult to prove that there cannot be any, as it is always possible to let $\ell$ vary and increase the size of ${\bf S}$. However, based on the large number of experiments we made, we find it likely that it is not possible to discard any of the elements of the remaining list $\mathcal{V}$ by changing $\ell$ or ${\bf S}$.
\end{rema}

\subsection{Motivic weight $24$} \label{par:mot24}
The following lemma is the first step in the proof of Theorem \ref{thm24}.

\begin{lemm}\label{lem:regmot24}
  Let $n \geq 13$.
  Let $\pi$ be a self-dual level $1$ cuspidal algebraic regular automorphic
  representation of $\PGL_n$ over $\Q$ of motivic weight $24$.
  Then $\rmL(\pi_\infty)$ belongs to the following list:
  \begin{itemize}
    \item $1 \,+\, {\rm I}_6 \,+\, {\rm I}_8 \,+\, {\rm I}_{14} \,+\, {\rm I}_{20} \,+\, {\rm I}_{22}
      \,+\, {\rm I}_{24}$ for $n=13$,
    \item $1 \,+\, {\rm I}_6 \,+\, {\rm I}_{10} \,+\, {\rm I}_{16} \,+\, {\rm I}_{20} \,+\, {\rm I}_{22}
      \,+\, {\rm I}_{24}$ for $n=13$,
    \item ${\rm I}_2 \,+\, {\rm I}_4 \,+\, {\rm I}_{12} \,+\, {\rm I}_{14} \,+\, {\rm I}_{18} \,+\,
      {\rm I}_{20} \,+\, {\rm I}_{22} \,+\, {\rm I}_{24}$ for $n=16$,
    \item ${\rm I}_2 \,+\, {\rm I}_8 \,+\, {\rm I}_{12} \,+\, {\rm I}_{14} \,+\, {\rm I}_{18} \,+\,
      {\rm I}_{20} \,+\, {\rm I}_{22} \,+\, {\rm I}_{24}$ for $n=16$.
  \end{itemize}
  In particular we have $n \leq 16$.
\end{lemm}
\begin{proof}

  Let $\pi$ be as in the lemma and set $U={\rm L}(\pi_\infty)$. This is a multiplicity free effective element of ${\rm K}_\infty^{\leq 24}$ containing ${\rm I}_{24}$  and with $\det U =1$. There are only finitely many such elements, with $\dim U \leq 25$ in all cases.   Moreover, $\pi$ is orthogonal 
  as $w(\pi)$ is even (see the last paragraph of \S \ref{regselfpar}). 
  By \cite[Theorem 1.5.3]{Arthur_book} loc.\ cit.\ we have thus $\epsilon(\pi) = +1$.
  Since $\epsilon(\pi) = \varepsilon(U)$ this gives an extra constraint on $U$. \ps\ps
  
  A straightforward computer-aided enumeration gives us the list of the $1260$
  effective multiplicity free elements $U$ in ${\rm K}_\infty^{\leq 24}$ containing ${\rm I}_{24}$, 
  and satisfying $\dim U>12$, $\det U = 1$ and $\varepsilon(U)  = 1$. 
  We applied Algorithm \ref{paralgo2} to each such $U$, $\delta=1$, $m=1$, 
  to the set ${\bf S}$ of $15$ elements of $\Pialg$ having motivic weight $\leq 23$ and dimension
  $\leq 4$.  Except in the four cases given in the statement, we obtained a contradiction with
  Inequality \eqref{eq:ineq_CF}. It is actually enough to choose $\ell$ in 
  $[3,7] \cap \frac{1}{2}\Z$ and to restrict to the subsets $S' \subset S$ with $|S'|\leq 7$.
  We refer to \cite{homepage} for an explicit list of $1256$ triples $(\ell,S',\underline{t})$ leading 
 to a contradiction in each case (checked using interval arithmetic). 
\end{proof}

\begin{proof} (Of Theorem \ref{thm24}) Using Theorem \ref{theomasseseffortless} we may compute, for 
any effective, multiplicity free, element $U \in {\rm K}_\infty^{\leq 24}$ containing ${\rm I}_{24}$, and with $\dim U \leq 16$,
the number of self-dual $\pi$ in $\Pi_{\rm alg}$ with ${\rm L}(\pi_\infty)=U$ (this uses ${\rm SO}_n$ for $n\leq 16$ and ${\rm Sp}_{2n}$ for $2n \leq 8$). Remarkably, we find only three such $\pi$, namely the ones in the statement of Theorem \ref{thm24}. We conclude by Lemma \ref{lem:regmot24}.
\end{proof}


\subsection{Classification results conditional to ${\rm (GRH)}$} \label{par:grh}
By ${\rm (GRH)}$ we shall mean here: for all $\pi,\pi' \in \Pi_{\rm alg}$, 
the zeros $s \in \C$ of $\Lambda(s,\pi \times \pi')$ satisfy $\Re\, s\, \, = \frac{1}{2}$. 
Assuming {\rm (GRH)}, Proposition \ref{prop:basicinequality} holds more generally for any test 
function $F$ satisfying $F(x) \geq 0$ and $\widehat{F}(\xi) \geq 0$ 
for all $x$ and $\xi$ in $\R$ (a condition weaker than ${\rm (POS)}$).
This condition is trivially satisfied by the function ${\rm G}_{\ell}(x)={\rm
g}(x/\ell)$, where ${\rm g}$ is the function recalled in \S \ref{numericaleval}
and $\ell$ is a positive real number (these are the classical functions of
Odlyzko ``under (GRH)''). 
In order to apply Algorithms \ref{paralgo} and \ref{paralgo2} with ${\rm G}_{\ell}$ instead of ${\rm F}_{\ell}$, we need the following variant 
of \cite[Prop. 9.3.18]{CheLan}. We set $\phi(z) = \frac{1}{2} \psi(\frac{z+1}{2})-\frac{1}{2} \psi(\frac{z}{2})$ and
${\rm r}(z)=2 \pi^{2} \frac{e^{-z}}{(z^{2}+\pi^{2})^{2}}$.

\begin{prop}\label{explGRH} Let $\ell>0$ be a real number. For any integer $w\geq 0$ we have
{\small
$${\rm J}_{{\rm G}_\ell}({\rm I}_w) = \log \, 2\pi  \,- \,\Re \,\psi(b+\frac{i \pi}{\ell}) \,+\,\frac{1}{\pi}\, \Im \,\psi(b+\frac{i \pi}{\ell})\,-\, \frac{1}{\ell} \Re\, \psi'(b+\frac{i \pi}{\ell}) \, + \,s_1(b,\ell),$$
}
with $b=\frac{1+w}{2}$ and $s_1(b,\ell)= \ell \sum_{n=0}^\infty {\rm r}(\ell (b+n))$. Moreover, we also have
{\small $${\rm J}_{{\rm G}_\ell}(1-\varepsilon_{\C/\R})=  \Re\,\phi\,(\frac{1}{2}+\frac{i \pi}{\ell})\, -\,  \frac{1}{\pi}\, \Im \,\phi(\frac{1}{2}+\frac{i \pi}{\ell})\, +\, \frac{1}{\ell}\, \Re \,\phi'(\frac{1}{2}+\frac{i \pi}{\ell})\,+\, \,s_2(\ell)$$
}
\par \noindent with $s_2(\ell)= \ell \sum_{n=0}^\infty (-1)^{n} {\rm r}(\ell(n+1/2))$, as well as $\widehat{{\rm G}_{\ell}}(0) = 8 \ell/\pi^{2}$ and
$$\widehat{{\rm G}_{\ell}}(i/4\pi)= 4\pi^{2}\ell \frac{1+{\rm cosh}(\ell/2)}{(\ell^2/4+\pi^{2})^{2}}.$$ \end{prop} 

\begin{pf} We follow the proof of \cite[Prop. 9.3.18]{CheLan} and omit the straightforward details.
For real numbers $b,\ell>0$, set $S(b,\ell) = \int_0^\infty
(\mathrm{g}(x/\ell)\frac{e^{-bx}}{1-e^{-x}} - \frac{e^{-x}}{x} ){\rm d}x$.
A computation almost identical to p. 276 {\it loc. cit.} shows that we have
$S(b,\ell)=- \,\Re \,\psi(b+\frac{i \pi}{\ell}) \,+\,\frac{1}{\pi}\, \Im
\,\psi(b+\frac{i \pi}{\ell})\,-\, \frac{1}{\ell} \Re\, \psi'(b+\frac{i
\pi}{\ell}) \, + \,s_1(b,\ell)$.
On the other hand, by \cite[Prop. 9.3.8]{CheLan} we have ${\rm J}_{{\rm G}_\ell}({\rm I}_w)\,=\, {\rm log}\, 2\pi\,  + S(\frac{1+w}{2},\ell)$ for any integer $w\geq 0$ and
${\rm J}_{{\rm G}_\ell}(1-\varepsilon_{\C/\R})= \frac{1}{2}(S(\frac{1}{4},2\ell) - S(\frac{3}{4},2\ell))$. This shows the first two formulas. Set $h(\alpha) =  \int_0^\infty {\rm g}(x) e^{-\alpha x} {\rm d}x$ for $\alpha$ in $\C$. By p. 275 {\it loc. cit.} we have $h(\alpha) = \frac{\alpha}{\alpha^2+\pi^2}+2 \pi^2 \frac{1+e^{-\alpha}}{(\alpha^2+\pi^2)^2}$. We conclude by the relations
$\widehat{{\rm G}_\ell}(0)\,=\,2 \ell \,h(0)$ and $\widehat{{\rm G}_\ell}(i/4\pi) \,= \,\ell \,(h(\ell/2)+h(-\ell/2))$.
\end{pf}

\noindent Upper bounds for the tails of the series  $s_1$ and $s_2$ are given in \cite[(3) p. 277]{CheLan}.\ps

\begin{pf} (Of Theorem \ref{thm23GRH}) In this proof, whenever we apply Algorithms \ref{paralgo} \& \ref{paralgo2} we do it using ${\rm G}_\ell$ instead of ${\rm F}_\ell$.
Applying Algorithm \ref{paralgo2} to the element $U$ of Proposition
\ref{prop:23_mult} (1), $\delta=m=2$ and the set ${\bf S}$ of $27$ known
elements of $\Pi_{\rm alg}$ with motivic weight $\leq 24$,  we obtain a
contradiction with Inequality \eqref{eq:ineq_CF} with $\ell=5$ and
$S'=\{\Delta_{23,7},\, \Delta_{23,13,5},\, \Delta_{23,15,7}\}$ (three elements
in the list of Thm.\ \ref{thm23}).
It thus only remains to show that for any of the $181$ elements of the list
$\mathcal{V}$ of Proposition \ref{prop:23_mult} (2), there is no selfdual $\pi$
in $\Pi_{\rm alg}$ with ${\rm L}(\pi_\infty)=U$.
For each $U$ in $\mathcal{V}$, we applied Algorithm  \ref{paralgo2} to $U$, $\delta=m=1$ and the set ${\bf S}$ above, using various $\ell$.
We found a contradiction in all but one cases. More precisely, we may reach all these contradictions but one using $\ell \in [4,7] \cap \frac{1}{2}\Z$, 
$S'$ of size $\leq 7$, and $S'$ not containing any of the $3$ elements of $\Pi_{\rm alg}$ with motivic weight $24$ (see \cite{homepage} 
for a working sheet). The two remaining elements of $\mathcal{V}$ are then 
$$A={\rm I}_1+{\rm I}_7+{\rm I}_{11}+{\rm I}_{15}+{\rm I}_{19}+{\rm I}_{21}+\,2 \,{\rm I}_{23}\, \, \, {\rm and}\,\,\,B={\rm I}_1+{\rm I}_9+{\rm I}_{15}+{\rm I}_{19}+\,2\, {\rm I}_{23}.$$
For $U=B$ we eventually found a contradiction using $\ell=6.36$ and a certain subset $S' \subset S$ with $13$ elements ! (see {\it loc. cit.}) 
The remaining case $U=A$ is the one of the statement of Theorem \ref{thm23GRH}.
\end{pf}

\section{Siegel modular cusp forms for ${\rm Sp}_{2g}(\Z)$} \label{secsmf}

In this section, we explain how to use our classification Theorems \ref{thm23} \& \ref{thm24} to prove
Theorem \ref{thmintro2}.
Along the way, we will also reformulate much more precisely the {\bf Key fact 1}
stated in the introduction. \ps

\subsection{Brief review of Arthur's results for ${\rm Sp}_{2g}$} \label{arthurreview}
Fix $g\geq 1$ be an integer. 
We denote by $\Pi_{\rm disc}({\rm Sp}_{2g})$
the set of isomorphism classes of discrete automorphic representations of ${\rm Sp}_{2g}$
with $\pi_p^{{\rm Sp}_{2g}(\Z_p)} \neq 0$ for all primes $p$.
Recall that the Langlands dual group of ${\rm Sp}_{2g}$ ``is'' ${\rm SO}_{2g+1}(\C)$;
it has a tautological (often called {\it standard}) representation ${\rm St}$ of dimension $2g+1$.
Let $\pi$ be  in $\Pi_{\rm disc}({\rm Sp}_{2g})$.
For each prime $p$, the Satake parameter ${\rm c}(\pi_p)$ of $\pi_p$
will be viewed following Langlands as a semi-simple conjugacy class in ${\rm SO}_{2g+1}(\C)$.
Similarly, the infinitesimal character ${\rm c}(\pi_\infty)$ of $\pi_\infty$
will be viewed as a semi-simple conjugacy class in the Lie algebra of ${\rm SO}_{2g+1}(\C)$
(and most of the time, as the collection of its $2g+1$ eigenvalues). \ps

Let $\Psi({\rm Sp}_{2g})$ denote the set of level $1$ {\it global Arthur parameters} for ${\rm Sp}_{2g}$.
An element of $\Psi({\rm Sp}_{2g})$ is by definition
a finite collection $\psi$ of distinct triples $(\pi_i,n_i,d_i)$, for $i$ in $I$,
with $n_i,d_i \geq 1$ a collection of integers satisfying $2g+1=\sum_{i \in I} n_i d_i$,
and with $\pi_i$ a level $1$ self-dual cuspidal automorphic representation of ${\rm PGL}_{n_i}$ over $\Q$
which is orthogonal if $d_i$ is odd, symplectic otherwise.
It suggestive to view $\psi$ as the isobaric automorphic representation of ${\rm GL}_{2g+1}$ over $\Q$ defined as
\begin{equation} \label{psiisobaric} \psi \,\,=\,\, \underset{i \in I}{\boxplus} \,\, \underset{0 \leq r_i \leq d_i-1}{\boxplus} \pi_i\, |.|^{\frac{d_i-1}{2}-r_{i}}.\end{equation}
We often simply write for short
\footnote{For typographical reasons we also replace the symbol $\pi_i[d_i]$ with $[d_i]$
if we have $\pi_i=1$, and by $\pi_i$ if we have $d_i=1$ and $\pi_i \neq 1$.}
$$\psi = \underset{i \in I}{\oplus} \pi_i[d_i].$$

To any $\psi$ in $\Psi({\rm Sp}_{2g})$,
viewed as in \eqref{psiisobaric} as an irreducible admissible representation of ${\rm GL}_{2g+1}(\A)$ over $\Q$,
we may attach a collection of Satake parameters $\psi_p$
(semi-simple conjugacy classes in ${\rm GL}_{2g+1}(\C)$),
as well as an infinitesimal character $\psi_\infty$
(a semi-simple conjugacy class in ${\rm M}_{2g+1}(\C)$).
We shall say that $\psi$ is {\it algebraic}
when the $2g+1$ eigenvalues of $\psi_\infty$ are in $\Z$.
In this case, the only one that we shall need to study here,
all the $\pi_i$ are algebraic (see \S \ref{realalgebraic}). \ps

Assume $\psi \in \Psi({\rm Sp}_{2g})$ is algebraic.
Using the local Langlands correspondence for the ${\rm GL}_{n_i}(\R)$,
we may attach to $\psi$ a morphism
$\psi_\R : {\rm W}_\R \times {\rm SL}_2(\C) \longrightarrow {\rm SO}_{2g+1}(\C)$,
uniquely defined up to ${\rm SO}_{2g+1}(\C)$-conjugacy, with the property
$${\rm St} \circ \psi_\R \simeq \bigoplus_i {\rm L}((\pi_i)_\infty) \boxtimes {\rm Sym}^{d_i-1} \C^2.$$
(Recall the notation ${\rm L}(-)$ from \S \ref{realalgebraic})
By \S \ref{realalgebraic}, note that $\psi_\R$ is trivial on $\R_{>0} \times 1$,
and in particular, $\psi_\R({\rm W}_\R)$ is bounded
(it is thus an Archimedean {\it Arthur parameter}).
If $r$ is a representation of ${\rm W}_\R$, and $d\geq 1$ is an integer,
it will be convenient to write $r[d]$ for the representation
$r \boxtimes {\rm Sym}^{d-1} \C^2$ of ${\rm W}_{\R} \times {\rm SL}_{2}(\C)$.\ps

Arthur's first main result \cite[Thm. 1.5.2]{Arthur_book} attaches to any $\pi$ in $\Pi_{\rm disc}({\rm Sp}_{2g})$
a unique $\psi(\pi)$ in $\Psi({\rm Sp}_{2g})$ such that
we have $\psi(\pi)_v={\rm St} \circ {\rm c}(\pi_v)$ for all place $v$ of $\Q$
(see also \cite[Lemma 4.1.1]{Taibi_dimtrace}).
Arthur's second main result is a converse statement,
the so-called {\it multiplicity formula},
on which we shall focus from now on and until the end of this section. \ps

Fix $\psi = \oplus_{i \in I} \pi_i[d_i]$ in $\Psi({\rm Sp}_{2g})$.
We assume that $\psi$ is algebraic for our purposes.
There are both a local and a global ingredient in the multiplicity formula.\ps

We start with the global one.
Write $I = I_{\rm even} \coprod I_{\rm odd}$
with $i \in I_{\rm even}$ if, and only if, $n_i d_i$ is even.
Define ${\rm C}_\psi$ as the abelian group generated
by the symbols $s_i$ for $i \in I_{\rm even}$,
and by the symbols $s_{ij}$ for all $i, j \in I_{\rm odd}$,
with relations $1=s_i^2=s_{ij}^2$ and $s_{ij}s_{jk}=s_{ik}$ (note $s_{ii}=1$).
This is an elementary abelian $2$-group of order $2^{|I|-1}.$
Arthur defines a global character $\epsilon_\psi$ of this group in \cite[p. 48]{Arthur_book}, that we now recall.
For each $i \in I$ consider the sign
\begin{equation}\label{defepsiloni} \epsilon(i) = \prod_{j \neq i} \epsilon(\pi_i \times \pi_j)^{{\rm Min}(d_i,d_j)}.\end{equation}
The term $\epsilon(\pi_i \times \pi_j)$ here is the Rankin-Selberg root number
already encountered in \S \ref{explicitformulaforpairs}, a (purely Archimedean)
sign that we already explained how to compute {\it loc. cit.}
It is necessarily $+1$ by \cite[Theorem 1.5.3]{Arthur_book} if $\pi_i$ and
$\pi_j$ are both orthogonal or both symplectic.
As the adjoint representation of ${\rm SO}_{2g+1}(\C)$ is isomorphic to $\Lambda^2 {\rm St}$,
Arthur's definition reads (see e.g. \cite[Sect. 8.3.5]{CheLan} for more details):
\begin{equation}\label{defepspsi} \epsilon_\psi(s_i)=\epsilon(i)\,\,\,\, \, \forall \, i \in I_{\rm even} \,\,\, \, \, {\rm and}\,\,\, \, \, \epsilon_\psi(s_{ij})=\epsilon(i)\epsilon(j)\,\,\,\,\,\forall \, i,j \in I_{\rm odd}.\end{equation}

We now describe the local ingredient.
Fix $K$ a maximal compact subgroup of ${\rm Sp}_{2g}(\R)$
and $\mathfrak{g}$ the Lie algebra of ${\rm Sp}_{2g}(\R)$.
Arthur associates to $\psi_\R$ a finite multi-set $\Pi(\psi_\R)$,
also called an {\it Arthur packet},
of unitary irreducible admissible $(\mathfrak{g},K)$-modules.
One important property he shows is that we have
$\pi_\infty \in \Pi(\psi(\pi)_\R)$ for all $\pi \in \Pi_{\rm disc}({\rm Sp}_{2g})$.
Moreover, $\Pi(\psi_\R)$ is equipped with a map
$$\Pi(\Psi_\R) \rightarrow {\rm Hom}({\rm C}_{\psi_\R},\{\pm 1\}), \, \, \, U \mapsto \chi_U,$$
where ${\rm C}_{\psi_\R}$ denotes the centralizer of the image of $\psi_\R$ in ${\rm SO}_{2g+1}(\C)$.

\begin{rema}\label{whittaker} The map $U \mapsto \chi_{U}$ depends on
the choice of an equivalence class of Whittaker datum for ${\rm Sp}_{2g}(\R)$.
From now on we fix a global Whitaker datum ${\rm Wh}$ for ${\rm Sp}_{2g}$
such that ${\rm Wh}_{p}$ is unramified with respect to ${\rm Sp}_{2g}(\Z_{p})$, for each prime $p$,
in the sense of Casselman and Shalika.
Up to conjugating ${\rm Wh}$ if necessary by the outer action of ${\rm GSp}_{2g}(\Z)$,
its Archimedian component ${\rm Wh}_{\infty}$ can belong to
any of the two classes of Whittaker data for ${\rm Sp}_{2g}(\R)$.
\end{rema}

We can now state Arthur's multiplicity formula.
Fix an algebraic $\psi$ in $\Psi(G)$.
There is a natural group embedding $\iota : {\rm C}_\psi \hookrightarrow {\rm C}_{\psi_\R}$
(``local-global'' map).
Choose $U$ in $\Pi(\psi_\R)$ and assume for simplicity that it has multiplicity
one in this multiset (this assumption will be satisfied in the cases that we
will consider below).
Then there is a $\pi$ in $\Pi_{\rm disc}(G)$ such that $\pi_\infty \simeq U$ if, and only if, we have
\begin{equation}\label{AMF} \epsilon_\psi(s_i)=\chi_U(\iota(s_i))\,\,\,\, \, \forall \, i \in I_{\rm even} \,\,\, \, \, {\rm and}\,\,\, \, \, \epsilon_\psi(s_{ij})=\chi_U(\iota(s_{ij}))\,\,\,\,\,\forall \, i,j \in I_{\rm odd}.\end{equation}
Moreover, if these equalities are satisfied
then the multiplicity of $\pi$ in the automorphic discrete spectrum of ${\rm Sp}_{2g}$ is equal to $1$.
There is a slightly more complicated statement when we do not assume $U$ has multiplicity one in $\Pi(\psi_\R)$.
This multiplicity one property will always be the case in our applications (see \S \ref{lowestweight}).
It is believed but not known that it holds in general,
although Moeglin and Renard have a number of results in this direction. \ps

\begin{rema} \label{stableamf}It is important to remark that \eqref{AMF} trivially holds
when we have $\psi = \varpi[d]$ for some cuspidal $\varpi$ of ${\rm PGL}_{(2g+1)/d}$,
because the group ${\rm C}_{\psi}$ is trivial. \end{rema}

\subsection{Lowest-weight modules: results of Arancibia-Moeglin-Renard and of Moeglin-Renard} \label{lowestweight}
For $\uk=(k_1,k_2,\dots,k_g) \in \Z^g$ with $k_1 \geq k_2 \geq \dots \geq k_g \geq 0$
we denote by $\rho_{\uk}$ the holomorphic, unitary, lowest weight $(\mathfrak{g},K)$-module of (lowest) weight $\uk$.
\footnote{These modules have been classified by Enright, Howe and Wallach:
they exist if, and only if, we have $k_g \geq g-(u+v/2)$,
with $u=|\{i, k_i=k_g\}|$ and $v=|\{i, k_i=k_g+1\}|$,
and they are unique up to isomorphism if they exist.}
The precise meaning here for "lowest" or "holomorphic" is a convention
that we may fix as in \cite[\S 3]{MR_scalar} to fix ideas,
nevertheless this choice will play no role in the sequel as we shall see.
We are interested in $\rho_{\uk}$ for the following classical reason.
Let us denote by ${\rm M}_{\uk}(\Gamma_{g})$ the vector-space of
vector-valued Siegel modular forms of weight $\uk$ for $\Gamma_g$,
and by ${\rm L}^{2}_{\uk}(\Gamma_g)$ its subspace of square-integrable forms. We have
$${\rm S}_{\uk}(\Gamma_{g})  \subset {\rm L}^{2}_{\uk}(\Gamma_g) \subset {\rm M}_{\uk}(\Gamma_{g}).$$
Assume $F$ is a Hecke eigenform in ${\rm L}^{2}_{\uk}(\Gamma_{g})$.
Then $F$ generates an element $\pi(F)$ in $\Pi_{\rm disc}({\rm Sp}_{2g})$
with $\pi(F)_\infty \simeq \rho_{\uk}$.
Better, $\dim {\rm L}^{2}_{\uk}(\Gamma_{g})$
(resp. $\dim {\rm S}_{\uk}(\Gamma_{g})$)
is exactly the number of $\pi$ in $\Pi_{\rm disc}({\rm Sp}_{2g})$
with $\pi_\infty \simeq \rho_{\uk}$ counted with their global discrete (resp. cuspidal) multiplicity. \ps\ps

An important property to have in mind is that
the $2g+1$ eigenvalues of the infinitesimal character of $\rho_{\uk}$
are $0$ and the $2g$ elements $\pm (k_i-i)$ for $i=1,\dots,g$.
Note that these $2g+1$ integers are distinct if, and only if, we have $k_g>g$.
This is also exactly the condition under which $\rho_{\uk}$ is a (holomorphic) discrete series.
If $F$ is a Hecke eigenform in ${\rm L}^{2}_{\uk}(\Gamma_{g})$ as above,
the shape of the infinitesimal character of $\rho_{\uk}$ implies that $\psi(\pi(F))$ is always algebraic.
Moreover, $\rho_{\uk}$ is an element of the Arthur packet $\Pi(\psi(\pi(F))_\R)$. \ps \ps

Conversely, let us fix until the end of \S \ref{lowestweight} a global Arthur parameter $$\psi=\oplus_{i \in {\rm I}} \pi_{i}[d_{i}]$$ in $\Psi({\rm Sp}_{2g})$
such that the eigenvalues of $\psi_\infty$ are $0$ and the $\pm (k_i-i)$,
with $i=1,\dots,g$ (in particular, $\psi$ is algebraic).
In order to apply the Arthur multiplicity formula, we want to know under which
condition on $\psi$ the module $\rho_{\uk}$ belongs to $\Pi(\psi_\R)$,
whether it has multiplicity one in this multi-set, and if so, we want to know
$\chi_{\rho_{\uk}}$.
We shall consider only the two following special, but important, cases.
\ps\ps

\subsubsection{Vector-valued case, with $k_g>g$.} \label{regularvectorvalued}
This situation is studied at length in \cite[Chap. 9]{ChRe} and \cite[Sect. 8.4.7]{CheLan}.
In this case, $\rho_{\uk}$ is a discrete series,
the eigenvalues of $\psi_\infty$ are distinct,
and we have ${\rm S}_{\uk}(\Gamma_{g})  = {\rm L}^{2}_{\uk}(\Gamma_g)$
by a general result of Wallach.
We have $I_{\rm odd}=\{i_0\}$ (a singleton) and $\pi_i$ is regular for all $i$ in $I$,
so we have $n_i d_i \equiv 0 \bmod 4$ for $i \neq i_0$ by \S \ref{regselfpar}.
The parameter $\psi_\R$ is necessarily an Adams-Johnson parameter
(see e.g. \cite[\S 3.8, App. A]{ChRe},\cite[Sect. 8.4.15]{CheLan},\cite[\S 4.2.2]{Taibi_dimtrace}),
and the main result of \cite{AMR} shows that $\Pi(\psi_\R)$ coincides
with the packet that Adams and Johnson associate to $\psi_\R$ in \cite{AdJo}
(any element of this packet having multiplicity one).
Arancibia, Moeglin and Renard also prove the expected form of the map $U \mapsto \chi_U$.
As was observed in \cite[\S 9]{ChRe} (see also \cite[Sect 8.5.1]{CheLan}),
this packet contains $\rho_{\uk}$ if and only if we have $d_{i_0}=1$,
and in this case the corresponding character $\chi_{\rho_{\uk}}$ is given by the formula, for all $i$ in $I_{\rm even}$:
\begin{equation} \label{chiukbasic} \chi_{\rho_{\uk}}(\iota(s_i)) \,\,=\,\,\left\{ \begin{array}{ll} (-1)^{\frac{n_i d_i}{4}} & \, \, {\rm if}\, \, d_i \equiv 0 \bmod 2, \\ (-1)^{{\rm e}_i}& \, \, {\rm otherwise},\end{array}\right. \end{equation}
where ${\rm e}_i$ is the number of {\it odd} integers $1 \leq i \leq g$
such that $k_i-i$ is a weight of $\pi_i$.
Note that the quantity ${\rm e}_i \bmod 2$ does not change
if we replace {\it odd} with {\it even} in the definition of ${\rm e}_i$,
as we have $n_i \equiv 0 \bmod 4$ for $d_i$ odd.
This property expresses the fact that the character above
does not depend on the choice of the Whittaker datum ${\rm Wh}_{\infty}$ in Remark \ref{whittaker}.
All in all, we have explained fully, and much more precisely,
the Key fact 1 of the introduction. \ps\ps

\subsubsection{Scalar-valued case, arbitrary genus}\label{scalarvalued}\ps\ps

In this case we have $\uk=(k,k,\dots,k)$ in $\Z^g$ with $k\geq 1$, and we rather
write $\rho_k(g)$ for $\rho_{\uk}$.
If we have $k>g$ we are in the case of \S \ref{regularvectorvalued},
so from now on we assume $g \geq k$.
The $2g+1$ eigenvalues of the infinitesimal character of $\rho_{\uk}$ are now
$0$ and the $2g$ elements $\pm (k-i)$ for $i=1,\dots,g$:
the eigenvalue $0$ has thus the multiplicity $3$,
and for $g>k>1$ the eigenvalues $\pm 1, \pm 2, \dots, \pm \min(k-1,g-k)$ have multiplicity $2$. \ps

We will use as a key ingredient the recent local results of Moeglin and Renard \cite{MR_scalar}
\footnote{Note that those authors call $n,m,\pi_n(m)^*$ what we call
$g,k,\rho_k(k)$ respectively. },
that we will specialize in what follows to this level $1$ situation.
The first main result of \cite{MR_scalar}
is that $\rho_g(k)$ belongs to $\Pi(\psi_\R)$
if, and only if, we are in one of the two cases called $({\rm I})$ and $({\rm H})$ below.
In both cases they show that $\rho_g(k)$ has multiplicity $1$ in $\Pi(\psi_\R)$ and
they determine $\chi_{\rho_g(k)}$.
We use the letter ${\rm I}$ for the case reminiscent of Ikeda lifts, and the
letter ${\rm H}$ for those related to the Howe (or theta) correspondence.
The formula for $\chi_{\rho_g(k)}$ given in \cite[Prop. 18.3]{MR_scalar}
depends on the class of ${\rm Wh}_{\infty}$, which is represented there by a
certain sign $\delta$ loc.\ cit.\ and that we represent the same way here (it
may be either one of $\pm 1$: see Remark \ref{whittaker}).
We will express below only the restriction of $\chi_{\rho_g(k)}$ to ${\rm C}_{\psi}$,
which is the information we need for applying the global multiplicity formula \eqref{AMF}.
In all cases we will see in particular that this restriction
does not depend on ${\rm Wh}_{\infty}$, i.e. on $\delta$,
hence neither on the choices discussed in the beginning of \S \ref{lowestweight}
that we made (or rather didn't) to define $\rho_{\uk}$:
changing of choice amounts to replace $\delta$ with $-\delta$ by \cite{MR_scalar}.\ps\ps

\begin{center} {\bf Preliminary general notations and remarks}\end{center}
-- Recall we have already defined a partition $I=I_{\rm even} \coprod I_{\rm odd}$
according to the parity of $n_id_i$ for $i$ in $I$.
We now define $I_0 \subset I$ as
the subset of elements $i$ in $I$ such that $0$ is a weight of $\pi_i$.
We clearly have $I_{\rm odd} \subset I_0$, $I_{\rm odd} \neq \emptyset$ and $|I_0| \leq 3$.
Set $$d_{\rm max}={\rm max}_{i \in I_0} d_i.$$

\noindent -- Say that an algebraic cuspidal automorphic representation $\pi$ of
${\rm PGL}_{n}$  satisfies {\rm (R)} if its weights are $\leq k-1$, its nonzero
weights have multiplicity $1$, the multiplicity ${\rm r}(\pi)$ of the weight $0$
of $\pi$ is $\leq 3$, and if $1$ and $\varepsilon_{\C/\R}$ have multiplicity at
most $2$ in ${\rm L}(\pi_\infty)$.
We shall see below that all the $\pi_i$ satisfy property (R), and that at most one
of them is not regular.\ps\ps

\noindent -- To decipher the formulas of  \cite[Prop. 18.3]{MR_scalar},
we have found it useful to observe, for $\epsilon=\pm 1$, $a \in \Z_{\geq 1}$ and $b \in \Z$,
the following equalities of signs:
\begin{equation}\label{deciphersign}(-1)^{\lfloor \epsilon a/2 \rfloor}= \prod_{i=1}^a (-1)^{i-1}\epsilon \, \, \, {\rm and}\, \, \, \prod_{i=b+1}^{b+a} (-1)^{i-1}\epsilon = (-1)^{\lfloor (-1)^{b}\epsilon a/2 \rfloor}.\end{equation}
This first is the product of $a$ alternating signs starting with $\epsilon$;
it only depends on $a \bmod 4$. The second follows from the first by replacing $\epsilon$ with $\epsilon(-1)^b$.
\begin{center} {\bf Case ({\rm I})}\end{center}
\ps

This corresponds to case (i) of \cite[Théorème 7.1]{MR_scalar}. 
We have $I_0=\{i_0\}$ (a singleton),
the weight $0$ of $\pi_{i_0}$ has multiplicity $1$,
$d_{i_0}=d_{\rm max}=1$, $k-1>g-k$,
and the $g$ numbers $w_i+\frac{d_i-1}{2}-r_i$,
where $i$ is in $I$, $w_i$ is a positive weights of $\pi_i$,
and with $0 \leq r_i \leq d_i-1$, fill the length $g$ segment $[k-g,k-1]$ (hence are distinct).
The representation $\pi_i$ is regular for each $i$, with weights $\leq k-1$,
hence satisfies $n_i d_i \equiv 0 \bmod 4$ for $i \neq i_0$ by \S \ref{regselfpar}.
Moreover, for $i \neq i_0$ the representation $\pi_i$ does not have $0$ among its
weights.
\ps\ps

In this case we must have $I_0=I_{\rm odd}$
so ${\rm C}_\psi$ is generated by the $s_i$ with $i \neq i_0$.
Fix such an $i$, necessarily in $I_{\rm even}$.
For any sign $s=\pm 1$ we define ${\rm e}_s(\pi_i)$ as the number of integers $1 \leq j \leq k-1$
with $(-1)^j = s$ such that $k-j$ is a weight of $\pi_i$.
The first assertion of Proposition 18.3 of  \cite{MR_scalar} (we are in case (1) of \S 18 {\it loc. cit.}), together with
Formula \eqref{deciphersign}, show that $\chi_{\rho_{k}(g)}(\iota(s_i))$ is given by the formula:
\begin{equation} \label{chiukbasic2} \chi_{\rho_{k}(g)}(\iota(s_i)) \,\,=\,\,\left\{ \begin{array}{ll} (-1)^{\frac{n_i d_i}{4}} & \, \, {\rm if}\, \, d_i \equiv 0 \bmod 2, \\ (-1)^{{\rm e}_{\delta}(\pi_i)} & \, \, {\rm otherwise}.\end{array}\right. \end{equation}

Indeed, consider the sequence of $g$ alternating signs $\underline{\frak{s}}=(\frak{s}_{k-1},\frak{s}_{k-2},\dots,\frak{s}_{k-g})$ 
starting with $\frak{s}_{k-1}=\delta$, i.e. set $\frak{s}_{k-i}=\delta (-1)^{i-1}$. Formulas  \eqref{deciphersign} 
show that, for any $i$ in $I_{\rm even}$ and any positive weight $w$ of $\pi_i$, the sign $\epsilon({\rm I}_{2w}[d_i])$ 
of Proposition  \cite[Prop. 18.3]{MR_scalar} is the product of the $\frak{s}_j$ with $j$ in the interval 
$[w-\frac{d_i-1}{2},w+\frac{d_i-1}{2}]$ of length $d_i$. 
Note that when $d_i$ is even this product is $(-1)^{d_i/2}$; when $d_i$ is odd, and thus $w \in \Z$, this product coincides with $\frak{s}_w$. 
The formula $\frak{s}_w=\delta(-1)^{k-w-1}$ shows that we have $\frak{s}_w=-1$ if and only if $w=k-j$ with $(-1)^j=\delta$. Formula \eqref{chiukbasic2} follows, 
as for $i$ in $I_{\rm even}$ the sign
$\chi_{\rho_{k}(g)}(\iota(s_i))$ is by definition the product, over all the positive weights $w$ of $\pi_i$, of $\epsilon({\rm I}_{2w}[d_i])$.\ps

Note that when $d_i$ is odd we have ${\rm e}_\delta(\pi_i)+{\rm e}_{-\delta}(\pi_i)= n_i/2\,\, \equiv 0 \bmod 2$,
as $\pi_i$ is regular and does not have the zero weight,
so ${\rm e}_{1}(\pi_{i}) \equiv {\rm e}_{-1}(\pi_{i}) \bmod 2$.
As a consequence, Formula \eqref{chiukbasic2} does not depend on $\delta$.\ps\ps
\ps\ps
\begin{center} {\bf Case ({\rm H})} \end{center}\ps
\ps
\noindent
This corresponds to case (ii) in \cite[Théorème 7.1]{MR_scalar}.
According to Theorem 7.2 loc.\ cit.\ there are two subcases:\ps
\noindent {\bf ({\rm H}1)}\,There is $i_0$ in $I_0$ with $d_{i_0}=d_{\rm max}=2(g-k)+1$,
and ${\rm L}((\pi_{i_0})_\infty)$ contains $\varepsilon_{\C/\R}^k$.\ps

\noindent {\bf ({\rm H}2)}\,There is $i_0$ in $I_0$ with $d_{i_0}=d_{\rm max}=2(g-k)+3$,
and ${\rm L}((\pi_{i_0})_\infty)$ contains $\varepsilon_{\C/\R}^{k-1}$.  \par \medskip

\noindent
Note that $i_0$ is not unique in general, so we fix any $i_0$ satisfying (H1) or
(H2). We set $k'=k$ in case (H1) and $k'=k-1$ in case (H2).
An inspection of $\psi_\R$ shows that in case (H2) we must have $g-k+1 \leq
k-1$, that is $g \leq 2k'$ (hence $k'\geq 1$). In both cases we may write
\[ \psi_\R\, \simeq \,\varepsilon_{\C/\R}^{k'}\, [2(g-k')+1] \,\oplus \, \psi'. \]
We have $\dim \psi' = 2k'$, $\det \psi' = \varepsilon_{\C/\R}^{k'}$ and
the eigenvalues of $\psi_\infty$ contributing to $\psi'$ are
the $\pm i$ for $i=0,\dots,k-1$ in case (H1),
and the same ones except $\pm (g-k+1)$ in case (H2). It follows that $\psi'$ is an Adams-Johnson parameter for the compact group ${\rm SO}(2k')$,
and in particular, is multiplicity-free.
This implies: \ps

--  $\pi_i$ satisfies (R) for all $i$, and is regular for $i \neq i_0$.\par\medskip

\noindent In particular, for $i \neq i_0$,
we have $n_i d_i \equiv 0 \bmod 4$ if $n_i$ is even,
and ${\rm L}((\pi_i)_\infty)$ contains $\varepsilon_{\C/\R}^{(n_i-1)/2}$ if $n_i$ is odd (see \S \ref{regselfpar}).
Moreover, either $\pi_{i_0}$ is regular
or we are in the case (H1) (see Remark \ref{iomult1caseh2}) and in one of the two following situations: \ps

-- $n_{i_0} \equiv 2 \bmod 4$, $0$ is a double weight of $\pi_{i_{0}}$,
and ${\rm L}((\pi_{i_0})_\infty)$ contains $\varepsilon_{\C/\R}^{k}$ twice,\ps

-- $n_{i_{0}}$ is odd, $d_{i_0}=1$, $g=k$, $0$ is a triple weight of $\pi_{i_{0}}$,
and ${\rm L}((\pi_{i_0})_\infty)$ contains $\varepsilon_{\C/\R}^{k}$ twice
and $\varepsilon_{\C/\R}^{k-1}$ once.\ps\ps

\begin{rema} \label{iomult1caseh2}  Assume we are in the case (H2). 
Then the weight $0$ of $\pi_{i_0}$ has multiplicity $1$, since the eigenvalue $g-k+1$ does not contribute to $\psi'$.
In particular we have $i_0 \in I_{\rm odd}$ and $\pi_{i_0}$ is regular. Moreover, we also have $k\geq 2$. Indeed, $k=2$ implies $g=2$, $\dim \pi_{i_0}=1$ hence $\pi_{i_0}=1$, which contradicts (H2).
\end{rema}

\noindent We now describe the restriction of the character $\chi_{\rho_k(g)}$ to ${\rm C}_\psi$.
For any $i$ in $I$ and any sign $s=\pm 1$
we define an integer ${\rm e}_s(\pi_i)$ as follows.
If we are in case (H1),
then ${\rm e}_s(\pi_i)$ is  the number of integers $1 \leq j \leq k-1$
with $(-1)^j= s$ such that $k-j$ is a weight of $\pi_i$ (as in case (I)).
If we are in case (H2),
we first consider the decreasing sequence
$(w_1,w_2,\dots,w_{k'-1})=(k-1,k-2,\dots,\widehat{g-k+1},\dots,1)$
where $g-k+1$ is omitted (this makes sense as we have $1 \leq g-k+1 \leq k-1$ by Remark \ref{iomult1caseh2}), and rather define ${\rm e}_s(\pi_i)$ as
the number of integers $1 \leq j \leq k'-1$ with $(-1)^j= s$
such that $w_j$ is a weight of $\pi_i$.
In all cases we have by property (R):
\begin{equation} \label{equadim}  {\rm r}(\pi_i)\,+\,2\,{\rm e}_{1}(\pi_i) \,+ \, 2\,{\rm e}_{-1}(\pi_i)\, =\, n_i. \end{equation}
\ps\ps

-- Assume first we have $i \in I_{\rm even}$ and $i \neq i_0$.
For $i \notin I_0$ we have: 
\begin{equation} \label{chiukbasic2bis} \chi_{\rho_{k}(g)}(\iota(s_i)) \,\,=\,\,\left\{ \begin{array}{ll} (-1)^{\frac{n_i d_i}{4}} & \, \, {\rm if}\, \, d_i \equiv 0 \bmod 2,
\\(-1)^{{\rm e}_{\delta}(\pi_i)} & \, \, {\rm otherwise}.\end{array}\right. \end{equation}
Indeed, this follows from \cite[Proposition 18.3]{MR_scalar} by a similar argument as in Case (I).
The only difference is to replace the alternating sequence 
of signs $\underline{\frak{s}}$ defined in Case (I) by the length $k'$ alternating
 sequence $\underline{\frak{s}}=(\frak{s}_{k-1},\frak{s}_{k-2},\dots,\frak{s}_0)$ starting with $\delta$ but {\it with the index $g-k+1$ omitted in case (H2)}; 
 in other words, we still set $\frak{s}_{k-i}=\delta (-1)^{i-1}$ in Case (H1), and in Case (H2) we set 
 $\frak{s}_{k-i}=\delta (-1)^{i-1}$ for $k-i>g-k+1$, and $\frak{s}_{k-i}=\delta (-1)^i$ otherwise. 
 The same arguments as in case (I) show Formula \eqref{chiukbasic2bis} and its independence on $\delta$.\ps

For $i$ in $I_0$, in which case we must have $d_i=1$, we rather find
\begin{equation} \label{chirhoio} \chi_{\rho_{k}(g)}(\iota(s_i)) \,=\,   (-1)^{{\rm e}_\delta(\pi_i)}\,\delta\,(-1)^{k'-1}. \end{equation}
Indeed, we are in the situation (2) of \S 18 {\it loc. cit.} and in the notations there we have
$a=1$ and $\epsilon_1\epsilon_2=(-1)^{\lfloor \delta (-1)^{k'-1}/2 \rfloor}= \delta
(-1)^{k'-1}=\frak{s}_0$  \cite[Remarque 18.4]{MR_scalar}.
The congruence $n_i \equiv 0 \bmod 4$, the equality ${\rm r}(\pi_i)=2$,
and \eqref{equadim} show that \eqref{chirhoio} does not depend on the sign $\delta$. \ps

-- Assume $i=i_0$ is in $I_{\rm even}$ (se we are in case (H1) by Remark \ref{iomult1caseh2}).
We have $n_{i_0} \equiv 0 \bmod 4$ if and only if $\pi_{i_0}$ is regular by  \S \ref{regselfpar}.
By \cite[Proposition 18.3]{MR_scalar} (including the formula for $\epsilon_2\epsilon_3$ in Case (2) there with $a=g-k-1$) we find:
\begin{equation}\label{casioeven} \chi_{\rho_{k}(g)}(\iota(s_{i_0}))  \,\,=\,\,\left\{ \begin{array}{ll} (-1)^{{\rm e}_\delta(\pi_{i_{0}})}\,\delta\,(-1)^{k-1} \, \, & \, \, {\rm if}\, \, n_{i_{0}}\equiv 0 \bmod 4,\\
(-1)^{{\rm e}_\delta(\pi_{i_{0}})} \,\,&\, \, {\rm otherwise}.\end{array}\right.\end{equation}
Again, these two formulas do not depend on $\delta$ by \eqref{equadim} and ${\rm r}(\pi_{i_0})=2$.

\ps

-- We are left to consider the case $|I_{\rm odd}|>1$.
We must have $|I_{\rm odd}|=3$ and $I_{\rm odd}=I_0$.
We want to give the value of $\chi_{\rho_{k}(g)}(\iota(s_{ij}))$ for $i \neq j$ in $I_{\rm odd}$.
We may write $I_0=\{i_0,i_1,i_2\}$ with $d_{i_1}=1$ and set $d_{i_2}=2a-1$.
By \cite[Proposition 18.3]{MR_scalar} (including the formula for $\epsilon_1\epsilon_2$ there and Remark 18.4 (ii)), we have:
\begin{equation} \label{eq:chi_sij}
  \chi_{\rho_{k}(g)}(\iota(s_{i_1i_2})) = (-1)^{{\rm e}_{\delta}(\pi_{i_1})+{\rm
  e}_\delta(\pi_{i_2})+ \lfloor \delta \,(-1)^{k'-a}\,a/2 \rfloor }.
\end{equation}
Observe that we also have $\delta \,(-1)^{k'-a}=\frak{s}_{a-1}$, as $a-1<(g-k)+1$ in case (H2), hence 
 $(-1)^{\lfloor \delta \,(-1)^{k'-a}\,a/2 \rfloor}=\frak{s}_{a-1}\cdots \frak{s}_1\frak{s}_0$ by Formula \eqref{deciphersign}.
As $\pi_{i_0}$ is regular of odd dimension,
and ${\rm L}((\pi_{i_0})_{\infty})$ contains $\varepsilon_{\C/\R}^{k'}$,
we have $n_{i_0} \equiv 2 k' +1 \bmod 4$.
This implies
\begin{equation}\label{propioi1i2} n_{i_0}d_{i_0} \equiv 2g+1 \bmod 4\, \,\, \, \, {\rm and}\, \, \, n_{i_1} + n_{i_2}(2a-1)  \equiv 0 \bmod 4,\end{equation}
which may also be written $\frac{n_{i_1}-1}{2}+\frac{n_{i_2}-1}{2} \equiv a \bmod 2$.
The relation ${\rm e}_\delta(\pi_l)+{\rm e}_{-\delta}(\pi_l) \equiv \frac{n_{i_l}-1}{2} \bmod 2$,
from \eqref{equadim},
and the trivial identity $(-1)^{\lfloor -ea/2 \rfloor} = (-1)^a
(-1)^{\lfloor ea/2 \rfloor}$ for $e=\pm 1$,
show that Formula \eqref{eq:chi_sij} does not
depend on $\delta$. \ps

--  We still assume $I_{\rm odd}=\{i_0,i_1,i_2\}$ as above.
It only remains to give $ \chi_{\rho_{\uk}}(\iota(s_{i_0 i_2}))$.
It will depend on the following property on $\{\pi_{i_0},\pi_{i_2}\}$:
$${\rm (P}_{i_0,i_2}{\rm )} \hspace{1 cm}{\rm L}((\pi_{i_0})_\infty) \oplus
{\rm L}((\pi_{i_2})_\infty) \,\,{\rm contains}\,\, 1\oplus \varepsilon_{\C/\R}.$$
Similarly to \eqref{casioeven} we find
{\small
\begin{equation} \label{eq:chi_sioi2} \chi_{\rho_{k}(g)}(\iota(s_{i_0 i_2}))  \,\,=\,\,\left\{ \begin{array}{ll}  (-1)^{{\rm e}_\delta(\pi_{i_{0}})+{\rm e}_\delta(\pi_{i_{2}})}\,\delta\,(-1)^{k-1}\,\,& \, \, {\rm if}\, \, {\rm (P}_{i_0,i_2}{\rm )} \,\,{\rm holds},\\
(-1)^{{\rm e}_\delta(\pi_{i_{0}})+{\rm e}_\delta(\pi_{i_{2}})}\,\,(-1)^{k-k'}\,\, & \, \, {\rm otherwise}.\end{array}\right.\end{equation}}

\par \noindent These values do not depend on $\delta$ by \eqref{equadim}:
the integer $n_{i_0}+n_{i_2}$ is $\equiv 0 \bmod 4$ if  ${\rm (P}_{i_0,i_2}{\rm )}$ holds,
and $\equiv 2 \bmod 4$ otherwise.
\ps

\begin{rema} \label{rema:I_and_H_overlap}
  \begin{enumerate}
    \item
      The cases (I) and (H) are disjoint for $k\neq g$, and for $k=g$ case (I)
      is a special case of (H1), and the two formulas \eqref{chiukbasic2} and
      \eqref{chiukbasic2bis} for $\chi_{\rho_k(g)}$ are identical.
    \item
      The parameter $\psi$ does not always determine $k$: when $g$ is even,
      parameters of type (H1) in weight $k=g/2$ coincide with parameters of type
      (H2) in weight $k=g/2+1$.
  \end{enumerate}
\end{rema}

\subsection{Proof of theorem \ref{thmintro2}}\label{pfthm2}

Fix $g\geq 1$, $\uk=(k_1,\dots,k_g) \in \Z^g$
and assume either that $\uk$ is scalar or $k_g>g$.
Arthur's multiplicity formula, as well as the multiplicity one results of \cite{AMR} and \cite{MR_scalar},
show that two Hecke eigenforms in ${\rm S}_{\uk}(\Gamma_g)$ with same Hecke eigenvalues,
or equivalently with the same standard parameter,
are proportional.
It follows that $\dim {\rm S}_{\uk}(\Gamma_g)$ is
the number of possible standard parameters of Hecke eigenforms in ${\rm S}_{\uk}(\Gamma_g)$.
In what follows we enumerate these parameters in the case $k_1 \leq 13$.
We thus fix a Hecke eigenform in ${\rm S}_{\uk}(\Gamma_g)$
and denote by $\psi \in \Psi({\rm Sp}_{2g})$ its standard parameter.
As in \S \ref{arthurreview} we write
$$\psi=\oplus_{i \in {\rm I}} \pi_{i}[d_{i}].$$
{\sc Notation:} -- Assume $v_1>v_2>\dots>v_r$ are positive odd (resp. even) integers
and that there is a unique self-dual regular $\pi$ in $\Pi^{\rm alg}$
with weights $\pm \frac{v_1}{2},\pm \frac{v_2}{2},\dots,\pm \frac{v_r}{2}$,
then we shall denote by $\Delta_{v_1,v_2,\dots,v_r}$ (resp. ${\rm O}^{\rm e}_{v_1,v_2,\dots,v_r}$)
this unique element $\pi$.
Similarly, when $v_1>v_2>\dots>v_r$ are even positive integers and there is a
unique self-dual $\pi$ in $\Pi^{\rm alg}$ with weights $0$ and $\pm
\frac{v_1}{2},\pm \frac{v_2}{2},\dots,\pm \frac{v_r}{2}$, then we shall denote
by ${\rm O}^{\rm o}_{v_1,v_2,\dots,v_r}$ this element.
The $\Delta$'s are symplectic and the ${\rm O}$'s are orthogonal.
These notations are compatible
with the ones introduced in \S \ref{classresult},
and we have for instance ${\rm O}^{\rm o}_{22}= {\rm Sym}^2 \Delta_{11}$.
We shall also denote by $1$
the trivial representation of ${\rm PGL}_1$,
and by $\Delta_{23}^1$ and $\Delta_{23}^2$
the two cuspidal representations of ${\rm PGL}_2$ generated by
the two normalized eigenforms in ${\rm S}_{24}({\rm SL}_2(\Z))$.

-- We will denote by $\mathcal{L}_{24}$ the subset of $\Pi^{\rm alg}$ whose elements are either of motivic weight $\leq 22$, or of motivic weight $23$ with the weight $23$ having multiplicity $1$, or regular self-dual of motivic weight $24$. By the classification theorems (\cite[Thm. F]{CheLan}, Theorems \ref{thm23} and \ref{thm24}), there are $11+13+3=27$ elements in $\mathcal{L}_{24}$, all regular self-dual, and according to the notations above we have
{\tiny
$$\mathcal{L}_{24} \, \, = \, \, \{1,\Delta_{11},\Delta_{15},\Delta_{17},\Delta_{19},\Delta_{19,7},\Delta_{21},\Delta_{21,5},\Delta_{21,9},\Delta_{21,13},{\rm Sym}^2 \Delta_{11},\Delta_{23}^1,\Delta_{23}^2,\Delta_{23,7},\Delta_{23,9},\Delta_{23,13},$$ $$\Delta_{23,13,5},\Delta_{23,15,3},\Delta_{23,15,7},\Delta_{23,17,5},\Delta_{23,19,3},\Delta_{23,17,9},\Delta_{23,19,11},\Delta_{23,21,17,11,3},{\rm O}^{\rm o}_{24,16,8},{\rm O}^{\rm e}_{24,18,10,4},{\rm O}^{\rm e}_{24,20,14,2}\}.$$
}

\ps

\subsubsection{Case $\uk=(k_1,k_2,\dots,k_g)$ with $k_1 \leq 13$ and $k_g>g$.}
We have $1 \leq g \leq 12$.
We apply \S \ref{regularvectorvalued}.
In this case, each $\pi_i$ is regular of motivic weight $\leq 2 (k_1-1) \leq 24$,
there is a unique $i_0 \in I$ with $n_{i_0}$ odd,
and we have $d_{i_0}=1$.
In particular, all the $\pi_i$ are in $\mathcal{L}_{24}$
and $\pi_{i_0}$ is either $1$, ${\rm Sym}^2 \Delta_{11}$ or ${\rm O}^{\rm o}_{24,16,8}$.
It is a boring but trivial exercise to enumerate all the $\psi$ in $\Psi({\rm
Sp}_{2g})$ with these properties and such that the eigenvalues of
$\psi_\infty$ are distinct and $\leq 12$.
We find exactly $199$ such parameters.
The possible $\psi$ are then exactly the ones in this list satisfying the Arthur multiplicity formula \eqref{AMF},
using Formulas \eqref{chiukbasic}, \eqref{defepspsi} and \eqref{defepsiloni}.
We find that only $59$ of these $199$ do satisfy this formula,
and obtain Table \ref{tab:tableleq13nonscal},
as well as the part of Table \ref{tab:tableleq13scal} concerning the case $k>g$.
All those computations can be done easily with the help of a computer:
see {\color{red}\cite{homepage}} for a \texttt{PARI} code doing it.
They can also be made by hand as follows.

We only treat the case $\pi_{i_0}={\rm Sym}^2 \Delta_{11}$, the two other ones
being similar.
Note that for $i \neq i_0$ such that $\pi_i$ is symplectic, we have $w(\pi_i)
\leq 19$, and either $\pi_i[d_i]=\Delta_{19,7}[2]$ or $\pi_i=\Delta_w[d]$ with
$w+d-1 \leq 20$.
Assume first that $\pi_i$ is symplectic for all $i \neq i_0$.
We have $\epsilon(\Delta_{19,7} \times {\rm Sym}^2 \Delta_{11})=1$, so if we have $\pi_i[d_i]=\Delta_{19,7}[2]$ then Arthur's multiplicity formula $\epsilon_\psi(s_i)=\chi_{\rho_{\uk}}(\iota(s_i))$ simply reads $1=1$.
If we have $\pi_i[d_i]=\Delta_w[d]$, it rather asserts
$-(-1)^{(w+1)/2}=(-1)^{d/2}$, i.e.\ $w \equiv d +1 \bmod 4$ (note
$\epsilon(\Delta_w \times {\rm Sym}^2 \Delta_{11})=-\epsilon(\Delta_w)$ for
$w<22$).
This justifies the existence of the $18$ $\psi$ in Tables
\ref{tab:tableleq13nonscal} and  \ref{tab:tableleq13scal} containing ${\rm
Sym}^2 \Delta_{11}$.
Assume now there is $i_1 \neq i_0$, necessarily unique, such that $\pi_{i_1}$ is
orthogonal.
We will show that this case cannot happen.
We have either $\pi_{i_1}={\rm O}^{\rm e}_{24,18,10,4}$ or $\pi_{i_1}={\rm
O}^{\rm e}_{24,20,14,2}$, and $d_{i_1}=1$.
If $I=\{i_0,i_1\}$ we have $\epsilon_\psi=1$ and we compute
$\chi_{\rho_{\uk}}(\iota(s_{i_1}))=-1$, so there is $i \neq i_0,i_1$ in $I$.
For weight reasons we must have $\pi_i[d_i]=\Delta_w[2]$, with $w \in \{17,11\}$
if $\pi_{i_1}={\rm O}^{\rm e}_{24,20,14,2}$, and $w=15$ otherwise.
This implies $\chi_{\rho_{\uk}}(\iota(s_i))=-1$.
Note that $\epsilon(\Delta_w \times \pi_{i_1})$ is $-1$ if $\pi_{i_1}$ has an
odd number of weights $>w/2$, and is $1$ otherwise.
This shows $\epsilon_\psi(s_i)=-\epsilon(\Delta_w)\epsilon(\Delta_w \times
\pi_{i_1})=1$ for $w=17,11$ and $\pi_{i_1}={\rm O}^{\rm e}_{24,20,14,2}$, a
contradiction.
We finally exclude the last possible case $\pi_{i_1}={\rm O}^{\rm
e}_{24,18,10,4}$ and $w=15$ as we have $I=\{i_0,i_1,i\}$,
$\chi_{\rho_{\uk}}(\iota(s_{i_1}))=-1$ and
$\epsilon_\psi(s_{i_1})=\epsilon(\Delta_w \times \pi_{i_1})=1$.

\subsubsection{Scalar-valued case with $k\leq 13$ and $g \geq k$.}
We apply \S \ref{scalarvalued}.
We are either in case (I), (H1) or (H2). \ps\ps

{\bf Case (I)}.
Assume first we are in case (I),
in particular we have $I_0=I_{\rm odd}=\{i_0\}$, $d_{i_0}=1$
and the $\pi_i$ are regular for all $i$ in $I$.
For each $i$ in $I$, we have $\frac{w(\pi_j)+d_j-1}{2} \leq 12$,
and in particular $\pi_i$ has motivic weight $\leq 24$:
it belongs to the list $\mathcal{L}_{24}$.
By inspection, $\pi_{i_0}$ has thus to be
$1, {\rm Sym}^2 \Delta_{11}$ or ${\rm O}^{\rm o}_{24,16,8}$.
As the weight $0$ has multiplicity $3$ in $\psi_\infty$ and $d_{i_0}=1$
we have $I \neq \{i_0\}$.
There is then a unique $j$ in $I-\{i_0\}$ such that,
if $a_i$ denotes the smallest positive weight of $\pi_i$,
we have $a_j-\frac{d_i-1}{2}\leq 0$;
we must have $a_j-\frac{d_i-1}{2} = k-g$.
We have both $\frac{d_j-1}{2} \geq a_j$ and $\frac{d_j-1}{2} \leq 12-w(\pi_j)/2$,
which implies $a_j + w(\pi_j)/2 \leq 12$.
By an inspection of $\mathcal{L}_{24}$,
this forces $\pi_j=\Delta_{11}$ and $d_j=12$.
But this implies that the positive weights of the $\pi_i$
with $i$ in $I-\{j\}$ are $\geq 12$.
Only the trivial representation has this property in $\mathcal{L}_{24}$.
This shows $\pi_{i_0}=1$ and that the unique possibility for $\psi$ is
$$\psi =  \Delta_{11}[12] \oplus [1].$$
We recognize the standard parameter of the genus $12$ Ikeda lift of $\Delta_{11}$ \cite{Ikeda01},
a well-known element of ${\rm S}_{12}(\Gamma_{12})$, hence $\psi$ does exist.
Alternatively, Arthur's multiplicity formula \eqref{AMF} is satisfied
as we have $\epsilon_\psi(s_j)=\epsilon(\Delta_{11})=1=\chi_{\rho_{12}(12)}(\iota(s_j))$ by \eqref{chiukbasic2},
so $\psi$ is indeed the standard parameter of an eigenform in ${\rm L}^2_{12}(\Gamma_{12})$.
The shape of $\psi$ and the Zharkovskaya relation
\footnote{This asserts that if $F$ is an eigenform in ${\rm M}_k(\Gamma_g)$,
and if $\Phi_g F$ in ${\rm M}_k(\Gamma_{g-1})$ is non zero,
with $\Phi_g$ the Siegel operator, then $\Phi_g F$ is an eigenform and the standard $L$-function of $F$
and $\Phi_g F$ satisfy ${\rm L}(s,{\rm St},F) = {\rm L}(s, {\rm St}, \Phi_g F) \zeta(s-(g-k))\zeta(s+(g-k))$. }
imply that this eigenform has to be cuspidal.

\ps\ps

{\bf Case (H1)}.
Assume we are in case (H1).
We will show again $\psi = \Delta_{11}[12] \oplus [1]$.
Write $\psi = \pi_{i_0} [ 2(g-k)+1] \oplus \bigoplus_{i \neq i_0} \pi_i[d_i]$
as in the definition of (H1).

\begin{lemm}
The representation $\pi_i$ is in $\mathcal{L}_{24}$ for all $i \in I$,
and we have $\pi_{i_0}=1$.
\end{lemm}

\begin{pf} As a general fact, all the $\pi_i$ satisfy condition (R) of \S \ref{scalarvalued}.
They have motivic weight $\leq 2(k-1) \leq 24$ and their nonzero weights have multiplicity $1$.
Moreover $\pi_i$ is regular for $i \neq i_0$.
It follows that all the $\pi_i$ are in $\mathcal{L}_{24}$, except perhaps
$\pi_{i_0}$ in the case $w(\pi_{i_0})=24$.
But for each $i$ and each positive weight $\lambda$ of $\pi_i$
we have $\lambda + \frac{d_i-1}{2}  \leq k-1 \leq 12$.
Thus $w(\pi_{i_0}) =24$ implies $k=13$ and $d_{i_0}=1=2(g-k)+1$, so $g=k=13$:
this is absurd as there is no nonzero Siegel modular form for $\Gamma_g$ with odd weight and genus.
So $\pi_{i_0}$ is in $\mathcal{L}_{24}$ with motivic weight $<24$,
and the unique possibilities are thus $\pi_{i_0}=1$ or $\pi_{i_0}={\rm Sym}^2 \Delta_{11}$
since $0$ is a weight of $\pi_{i_0}$.
Assume we have $\pi_{i_0}={\rm Sym}^2 \Delta_{11}$.
Then ${\rm L}((\pi_{i_0})_\infty)$ contains $\varepsilon_{\C/\R}$ so
$k$ is odd by (H1), $g$ is even, and we have $d_{i_0} \equiv 3 \bmod 4$.
The inequality $11+\frac{d_{i_0}-1}{2} \leq k-1\leq 12$ implies then
$k=13$, $d_{i_0}=3$ and $g=14$.
We have proved
 $$\psi = {\rm Sym}^2 \Delta_{11}[3] \oplus \psi'$$
with $\psi' = \oplus_{i \neq i_0} \pi_i[d_i]$,
and $\psi'_\infty$ has the eigenvalues $\pm 9, \pm 8, \dots, \pm 1$ and $0$ twice.
But the $\pi_i$ are in $\mathcal{L}_{24}$ with motivic weight $\leq 18$,
hence in $\{1,\Delta_{11},\Delta_{17}\}$,
and we have $d_i \leq 3$ for $i$ in $I_{\rm odd}$.
The only possibility is thus $\psi'=\Delta_{11}[8] \oplus [3] \oplus [1]$.
We have $I_{\rm even}=\{i\}$ with $\pi_i[d_i]=\Delta_{11}[8]$,
$\epsilon_\psi(s_i)=\epsilon(\Delta_{11} \times {\rm Sym}^2 \Delta_{11})=-1$ but
$\chi_{\rho_{13}(14)}(\iota(s_i))=1$ by Formula \eqref{chiukbasic2bis}: the
multiplicity formula is not satisfied.
\end{pf}

Note that $\pi_{i_0}=1$ implies $k \equiv 0 \bmod 2$ by (H1),
hence $k\leq 12$. Write again
\[ \psi = [2(g-k)+1] \oplus \psi' \]
where $\psi'=\oplus_{i\neq i_0} \pi_i[d_i]$ is a certain $2k$-dimensional
parameter with weights $\pm (k-1), \pm (k-2), \dots, \pm 1$ and $0$ twice, and $k-1 \leq 11$.
 Each $\pi_i$ with $i \neq i_0$ is then regular of motivic weight $\leq 22$.
 The list of all possible $\psi'$ with these properties is easy to determine:
 see Proposition 9.2.2 of \cite{CheLan}.
 For $k=2,4,6$ we find $\psi'= [2k-1] \oplus [1]$.
 For $k=8$ we have $[15] \oplus [1]$ and $\Delta_{11}[4]\oplus [7] \oplus [1]$.
 For $k=10$ we have
{\small $$[19]\oplus[1], \Delta_{11}[8]\oplus[3]\oplus[1], \Delta_{15}[4]\oplus [11] \oplus [1], \Delta_{17}[2]\oplus [15]\oplus [1], \Delta_{17}[2]\oplus\Delta_{11}[4]\oplus[7]\oplus[1].$$}
\noindent For $k=12$, we have $24$ possibilities for $\psi'$,
namely the ones in \cite[Thm. E]{CheLan}.

 \begin{lemm}\label{lemm:congmod4}
 We have $k \equiv 0 \bmod 4$.
 \end{lemm}

\begin{pf}
  Assume that $k \equiv 2 \bmod 4$, then by inspection $|I_\mathrm{odd}|=3$ and
  we denote $I_\mathrm{odd} = \{ i_0, i_1, i_2 \}$ so that
  $\pi_{i_1}[d_{i_1}]=[1]$ and $\pi_{i_2}[d_{i_2}]=[2a-1]$.
  For any $i \in I_\mathrm{even}$ we have $n_i d_i/2 \equiv 0 \bmod 2$ and so $a
  \equiv k \bmod 2$, i.e.\ $a$ is even.
  Thus Arthur's multiplicity formula \eqref{AMF} implies
  \[ \chi_{\rho_k(g)}(\iota(s_{i_1 i_2})) = (-1)^{a/2} = \epsilon(i_1)
  \epsilon(i_2), \]
  the first equality being \eqref{eq:chi_sij}.
  For $\psi'=[1]\oplus[2k-1]$ those epsilon factors are $1$ and we have $a=k
  \equiv 2 \bmod 4$ so this formula does not hold.
  This rules out $k=2$ and $k=6$.
  The four other parameters for $k=10$ are ruled out the same way using
  $\epsilon(\Delta_w)=(-1)^{(w+1)/2}$.
\end{pf}

We will now prove that, apart from the case $\psi = \Delta_{11}[12] \oplus [1]$,
none of the remaining $\psi$ come from a {\it cuspidal} modular form.
We will need first to recall some results on orthogonal automorphic forms and theta series.
For each integer $n \equiv 0 \bmod 8$ we fix arbitrarily an even unimodular lattice of rank $n$
and denote respectively by $\Omega_n$ and ${\rm S}\Omega_n$
its orthogonal and special orthogonal group schemes over $\Z$.
We refer to \cite[Sects. 4.4 \& 6.4.7]{CheLan} for
the basics of the theory of level $1$ automorphic forms for $\Omega_n$ and ${\rm S}\Omega_n$
(beware that these group schemes are rather denoted by ${\rm O}_n$ and ${\rm SO}_n$ {\it loc. cit.}).
By results of Arthur \cite{Arthur_book} and Ta\"ibi \cite{TaiMult},
any discrete automorphic representation of ${\rm S}\Omega_n$ or $\Omega_n$
has a standard parameter $\psi$ in $\Psi({\rm S}\Omega_n)$,
the latter being defined exactly as in the case of ${\rm Sp}_{2g}$
(see \S \ref{arthurreview}) but with the condition $\sum_{i\in I} n_id_i=n$ instead of $\sum_{i \in I} n_id_i=2g+1$. \ps

For $n \equiv 0 \bmod 8$, we denote by ${\rm X}_{n}$
the set of isomorphism classes of even unimodular lattices of rank $n$.
The vector-space $\C[{\rm X}_{n}]$ is in a natural way
the dual of a space of level $1$ automorphic forms for $\Omega_n$.
Any Hecke eigenform $G$ in $\C[{\rm X}_{n}]$ generates
a discrete automorphic representation $\pi_G$ of $\Omega_n$
(with trivial Archimedean component and $(\pi_F)_p^{\Omega_n(\Z_p)} \neq 0$ for each prime $p$),
which has a standard parameter $\psi_G$ in $\Psi({\rm S}\Omega_n)$.
Moreover, Siegel theta series provide a linear map
\begin{equation} \label{linearthetasiegel} \vartheta_g : \C[{\rm X}_{n}] \longrightarrow {\rm M}_{n/2}(\Gamma_g)\end{equation}
for all $g\geq 0$ (see e.g. \cite[\S 5.1]{CheLan},
in particular for the conventions for $g=0$), with $\Phi \circ \vartheta_g = \vartheta_{g-1}$
(here $\Phi$ denotes the Siegel operator).
For $G$ in $\C[{\rm X}_n]$, the {\it degree} of $G$ is the smallest integer $g_0 \geq 0$ with $\vartheta_{g_0}(G) \neq 0$;
the form $\vartheta_{g_0}(G)$ is then cuspidal
and we have $\vartheta_{g}(G)\neq 0$ for $g\geq g_0$.
If $G$ in $\C[X_n]$ is an eigenform with degree $g_0$, and for $g\geq g_0$,
then the Eichler commutation relations show that
$\vartheta_g(G)$ is an eigenform in ${\rm M}_{n/2}(\Gamma_g)$,
and there is a simple relation due to Rallis \cite[\S 6]{Rallis_eichler}
between the Satake parameters of $G$ and that of $F=\vartheta_g(G)$ (see \cite[Sect 7.1]{CheLan}).
Concretely, if $F$ is square integrable (e.g. cuspidal),
this relation is the equality $\psi_G = \psi_F \oplus [n-2g_0-1]$ for $n>2g_0+1$,
$\psi_F = \psi_G \oplus [2g_0+1-n]$ for $n<2g_0+1$.
Last but not least, we have the following result,
a consequence of \cite[Thm. I.1.1]{Rallis_Howe} and \cite[Lemme I.4.11]{MW_eisenstein}
that we learnt from \cite[\S 16.2]{MR_scalar}.\ps

\begin{lemm}
\label{thetasquareint}
Let $G$ be an eigenform in $\C[{\rm X}_{n}]$ of degree $g_0$.
If we have $g>g_0$ and $g>n-1-g_0$,
then $F'=\vartheta_{g}(G)$ is square integrable
and $\psi_{F'}=\psi_G \oplus [2g-n+1]$.
\end{lemm}

(Note that we have $2g+1>g+g_0+1>n$, hence the last assertion.)
We finally go back to our analysis of case (H1), setting $n=2k$.
The spaces $\C[{\rm X}_{8}], \C[{\rm X}_{16}]$ and $\C[{\rm X}_{24}]$
have respective dimension $1,2,24$,
and the standard parameters of their eigenforms turn out to be exactly
the $1$, $2$ and $24$ parameters $\psi'$ discussed above for $k=4,8$ and $12$,
by \cite[Cor. 7.2.7 \& Thm. E]{CheLan}.
This reference determines as well the degree of each eigenform
(see \cite[Thm. 9.2.6]{CheLan},
note that most of these degrees had already been found before by Nebe and Venkov):
this is the smallest integer $g_0$ such that $[2k-1-2g_0]$ is a summand of $\psi'$ (hence $g_0 < k$),
unless we have $\psi'=\Delta_{11}[12]$ and $g_0=12$.
For $\psi' \neq \Delta_{11}[12]$ we have thus $g>g_0$ as well as $g>2k-1-g_0$
by the necessary condition $d_{\rm max}=2(g-k)+1$ of (H1).
By Lemma \ref{thetasquareint}, the automorphic representation $\pi_{F'}$ is thus
the (necessarily unique) discrete automorphic representation of ${\rm Sp}_{2g}$ with parameter $\psi=[2(g-k)+1] \oplus \psi'$,
and it is not cuspidal since we have $g>g_0$.
In the remaining case we have
$\psi = [2(g-k)+1] \oplus \Delta_{11}[12]$, $k=12$ and $g_0=12$,
and again $\pi_{F'}$ is discrete but not cuspidal if we have $g>12$.
We conclude since for $g=12$ we recover the form found in case (I). \ps\ps

{\bf Case (H2)}. We are going to show that
there are exactly two Siegel eigenforms in this remaining case,
both for $k=13$, of respective genus $16$ and $24$, and parameters
$$\Delta_{17}[8] \oplus [9] \oplus [7] \oplus [1]\, \, \, \, \, {\rm and}\, \, \, \, \, [25 ]\oplus \Delta_{11}[12].$$

\begin{lemm}
We have $\pi_{i_0}=1$, $k$ odd and $g$ even.
\end{lemm}

\begin{pf}
  As we are in case (H2) we must have $w(\pi_{i_0}) + d_{i_0}-1 \leq 2(k-1) \leq
  24$ and $d_{i_0} = 2(g-k)+3 \geq 3$, and so $w(\pi_{i_0}) \leq 22$.
Assume $\pi_{i_0}$ is non trivial.
We have $\pi_{i_0} = \Sym^2 \Delta_{11}$ by the Chenevier-Lannes theorem, so
$w(\pi_{i_0})=22$, $d_{i_0}=3$ and $k=13$ is odd.
This contradicts the last condition of (H2).
So $\pi_{i_0}$ is trivial, $k$ is odd by the last condition in (H2),
hence $g$ is even as we are in full level $\Gamma_g$.
\end{pf}

Write $\psi = [2(g-k)+3] \oplus \psi'$, with $\psi' = \bigoplus_{i \neq i_0} \pi_i[d_i]$.
The eigenvalues of $\psi_\infty$ corresponding to $\psi'$
are the $2k-2$ integers $\pm i$ with $0 \leq j \leq k-1$,
with the even number $j=g-k+1$ omitted
(we shall call those $2k-2$ eigenvalues the "weights" of $\psi'$ for short).
Each $\pi_i$ is regular algebraic of motivic weight $\leq 24$,
hence in the list $\mathcal{L}_{24}$.
We are now led to do a simple enumeration exercise:
for every odd $k \in \{1, \dots, 13\}$,
enumerate all possible $\psi'$, with $\pi_i$ in $\mathcal{L}_{24}$ for each $i$,
and with weights $\pm 0, \dots, \pm (k-1)$
where the even integer $\pm (g-k+1)$ is excluded and satisfies $k-1 \geq g-k+1 > 1$. \ps

\begin{lemm} \label{infoio}
Assume $i \in I_0$ and $\pi_i \neq 1$, then we have $\pi_i={\rm Sym}^2
\Delta_{11}$, $d_i=1$, $k=13$ and $g=24$, as well as $I_0 = I_\mathrm{odd} = \{
i_0, i, j \}$ with $j \neq i_0,i$, $\pi_j=1$ and $d_j \geq 5$.
\end{lemm}

\begin{pf} Assume first $\pi_i={\rm O}^{\rm o}_{24,16,8}$.
The only $\pi$ in $\mathcal{L}_{24}$ with $w(\pi)  \notin \{24,23,17,15\}$,
and having a weight $5 \leq \lambda \leq 7$, are $\Delta_{11}$ and
$\Delta_{21,13}$.
It follows that among the three consecutive integers $5,6,7$,
either $7$ or $5$ is not a weight of $\psi'$,
hence must be $g-k+1$: a contradiction as $g-k+1$ is even. \ps

An inspection of $\mathcal{L}_{24}$ shows then $\pi_i={\rm Sym}^2 \Delta_{11}$,
$k-1 \geq 11$ and $d_i \leq 3$.
As $k$ is odd we have $k=13$.
Assuming $d_i=3$, $\pi_i[d_i]$ contributes to the weights $\pm 12,\pm11,\pm10$ and $\pm 1,0$ of $\psi'$.
There is thus a unique $i_1$ in $I$ with $\pi_{i_1}[d_{i_1}]=[1]$, and for $j
\neq i_0,i_1,i$ the representation $\pi_j$ is symplectic with motivic weight $
\leq 17$.
This shows $2 \leq g-k+1 <9$.
But there is an odd number of integers $2 \leq n \leq 9$ with $n \neq g-k+1$: a contradiction.
We have proved $d_i=1$. \ps

As $12$ is an eigenvalue of $\psi_\infty$,
we have either $g-k+1=12$
or there exists some $j \in I$ with $w(\pi_j)=24$ and not having the weight $22$.
The only remaining possibility in this latter case is
$\pi_j={\rm O}^{\rm e}_{24,18,10,4}$ for some $j$.
But there is no $\pi$ in $\mathcal{L}_{24}$ with $w(\pi)  \notin \{24,23,19,17,11\}$
and having a weight $3 \leq \lambda \leq 4$: a contradiction
with the presence of the weight $3$ or $4$ of $\psi'$.
We have proved $g-k+1=12$, i.e. $g=24$. \ps

The weights of $\psi'$ are thus $\pm 11, \pm 10, \pm 9, \dots, \pm 1$ and $0$ twice.
Those possible $\psi'$ are easily determined (see \cite[Prop. 9.2.2]{CheLan} or \cite[Thm. E]{CheLan}):
there are $10$ possibilities, all of them containing some $[d]$ with $d\geq 5$, except
$$\psi' = {\rm Sym}^2 \Delta_{11} \oplus \Delta_{11}[10] \oplus [1].$$
We exclude this case using Arthur's multiplicity formula.
We have $I_\mathrm{odd}=\{i_0,i,j\}$ with $\pi_j=\pi_{i_0}=1$.
We have $\epsilon_\psi(s_{i_0 j})=1$,
but by Formula \eqref{eq:chi_sioi2} we have $\chi_{\rho_k(g)}(\iota(s_{i_0 j}))=-1$: a contradiction.\end{pf}

We are now able to conclude the proof.
Assume first that $\Delta_{11}[12]$ is a summand of $\psi'$.
We must have
$$\psi = [25] \oplus \Delta_{11}[12],$$
which trivially satisfies Arthur's multiplicity formula by \eqref{chiukbasic2bis}.
In this case there is thus an eigenform $F \in {\rm L}_{13}^2(\Gamma_{24})$
with parameter $\psi$;
this $F$ is necessarily cuspidal as we have ${\rm M}_k(\Gamma_g)={\rm S}_k(\Gamma_g)$ for $k$ odd.
As we will explain in \S \ref{complconstruction},
this $F$ is actually the form constructed by Freitag in the last section of \cite{Freitag_harm_theta}.  \ps

So we may assume that $\Delta_{11}[12]$ is not a summand of $\psi'$.
Consider the double weight $0$ of $\psi'$.
An inspection of $\mathcal{L}_{24}$ shows then
that there are two elements $i_1,i_2$ in $I_{\rm odd}-\{i_0\}$,
say with $1=d_{i_1} \leq d_{i_2}$.
Better, Lemma \ref{infoio} implies that we have $\pi_{i_2}=1$, $d_{i_2} \geq 3$,
and either $\pi_{i_1}=1$ or $\pi_{i_1}= {\rm Sym}^2 \Delta_{11}$.
We apply Arthur's multiplicity formula at the element $s_{i_0 i_2}$.
We have $\chi_{\rho_k(g)}(\iota(s_{i_0 i_2})) = -1$
by Formula \eqref{eq:chi_sioi2}.
This implies $\epsilon_{\psi}(s_{i_0 i_2}) = -1$, which is equivalent to
\[ \prod_{\substack{l \in L}} \epsilon(\pi_l) = -1 \]
where $L$ is the set of elements $l$ in $I_{\mathrm{even}}$ such that $\pi_l$ is symplectic
and with $d_{i_2} < d_l < d_{i_0}$.
We have $d_l \geq 4$ for $l \in L$, as $d_l$ is even and $d_{i_2} \geq 3$,
which imposes $w(\pi_l) \leq 21$.
Among the $9$ symplectic representations with such motivic weight,
only $\Delta_{17}$ and $\Delta_{21}$ have a negative epsilon factor.
As a consequence, at least one summand of $\psi'$ is among
\begin{equation} \label{pikkpossible}
  \Delta_{17}[4],\Delta_{17}[8], \Delta_{21}[4].
\end{equation}
Observe that this implies $d_{i_2} < 8$ and that such a summand always
contributes the weights $9$ and $10$ to $\psi'$, so that we have $k \geq 11$.
\ps

Assume first $\pi_{i_1}={\rm Sym}^2 \Delta_{11}$, hence $g=24$, $k=13$ and
$d_{i_2}\geq 5$ by Lemma \ref{infoio}.
Then $\Delta_{17}[8]$ is a summand of $\psi'$, but the weight $11$ occurs in
both $\Delta_{17}[8]$ and $\mathrm{Sym}^2 \Delta_{11}$, a contradiction.
We have proved $\pi_{i_1}=\pi_{i_2}=1$.
The congruence \eqref{propioi1i2} implies then $d_{i_2} \equiv -1 \bmod 4$,
which leaves only the two cases $d_{i_2}=3$ or $d_{i_2}=7$
by the inequality $d_{i_2}<8$.
In the case $d_{i_2}=7$ the only possibility is thus
that $\psi'$ contains $\Delta_{17}[8] \oplus [7] \oplus [1]$,
hence is equal to the latter for weights reasons, and
$$\psi = \Delta_{17}[8] \oplus [9] \oplus [7] \oplus [1].$$
Arthur's multiplicity formula is satisfied for this $\psi$ by Formulas \eqref{chiukbasic2bis},
\eqref{eq:chi_sij} and \eqref{eq:chi_sioi2}.
There is thus an eigenform in ${\rm L}_{13}^2(\Gamma_{16})$
with parameter $\psi$, necessarily cuspidal as its weight is odd. \ps

We are left to study the case $d_{i_2}=3$.
In this case we focus on the weight $3$ of $\psi'$.
It cannot come from a summand in the list \eqref{pikkpossible}.
It must thus come from a summand $\pi_m[d_m]$ of $\psi'$ which does not
contribute to any weight in $\{0,1,9,10\}$, in particular $\pi_m$ does not have any
weight in $\{21/2,10,19/2,9,17/2,3/2,1,0\}$.
If $\pi_m$ has motivic weight $<23$, an inspection of $\mathcal{L}_{24}$ shows
$\pi_m[d_m]=\Delta_{11}[6]$, which leads to
$$\psi = \Delta_{21}[4] \oplus \Delta_{11}[6] \oplus [5] \oplus [3] \oplus [1].$$
If $\pi_m$ has motivic weight $23$,
then we have $d_m=2$, $\Delta_{17}[4]$ is a summand of $\psi'$,
so $15/2$ and $13/2$ are not weights of $\pi_m$,
and the only possibility is $\pi_m=\Delta_{23,7}$ and
$$\psi = \Delta_{23,7}[2] \oplus \Delta_{17}[4] \oplus \Delta_{11}[2] \oplus [5] \oplus [3] \oplus [1].$$
In both cases, $\psi$ does not satisfy
the multiplicity formula at the element $s_h$ with $\pi_h[d_h]=\Delta_{11}[d_h]$:
we have $\chi_{\rho_k(g)}(\iota(s_h))=(-1)^{d_h/2}=-1$ by \eqref{chiukbasic2bis} and $\epsilon_\psi(s_h)=1$.
This concludes the proof of Theorem \ref{thmintro2}.

\subsection{Complements: theta series constructions}\label{complconstruction} 
Recall that for any integer $n \equiv 0 \bmod 8$ we denote by 
$\mathcal{L}_n$ the set of even unimodular lattices in the standard 
Euclidean space $\R^n$, and by ${\rm X}_n={\rm O}(\R^n) \backslash \mathcal{L}_n$ the finite set of isometry
classes of such lattices.\ps \ps

Our first complement concerns the question of the surjectivity of the linear map
$\vartheta_{g} : \C[{\rm X}_{2k}] \rightarrow {\rm M}_{k}(\Gamma_g)$
of Formula \eqref{linearthetasiegel}, also called the {\it Eichler basis problem}, for $k \equiv 0 \mod 4$.
This surjectivity was proved in \cite[\S 1.3]{CheLan} in the case $g\leq k \leq 12$.

\begin{coro}\label{eichlerbasis}
Assume $k=4,8$ or $12$.
Then $\vartheta_g : \C[{\rm X}_{2k}] \rightarrow {\rm M}_{k}(\Gamma_g)$
is surjective for all $g$.
In particular, the Siegel operator
$\Phi_g : {\rm M}_k(\Gamma_g) \rightarrow {\rm M}_k(\Gamma_{g-1})$
is surjective as well for all $g\geq 1$.
\end{coro}

\begin{pf}
We have the relations
$\Phi_g \circ \vartheta_g = \vartheta_{g-1}$ and
${\rm Ker}\, \Phi_g\, =\, {\rm S}_k(\Gamma_g)$.
The corollary follows thus from the case $g\leq k$,
and from the vanishing ${\rm S}_k(\Gamma_g)=0$ for $g>k$ and $k\leq 12$,
implied by Theorem \ref{thmintro2}.
\end{pf}

In other words, the Eichler basis problem holds for all $g$ for those three values of $k$.
We stress that this is {\it not} a general phenomenon:
as was observed in \cite[Cor. 7.3.5]{CheLan} the map $\vartheta_{14}$ is not surjective for $k=16$.

\begin{rema} For $k=4,8,12$, the surjectivity of $\Phi_g$
and the determination of $\dim {\rm S}_k(\Gamma_g)$ for all $g$ by Table \ref{tab:tableleq13scal}
allow to determine $\dim {\rm M}_k(\Gamma_g)$ for all $g$.
In particular, we have $\dim {\rm M}_k(\Gamma_g)=|{\rm X}_{2k}|$
for $k=4$ and $g\geq 1$,
for $k=8$ and $g\geq 4$,
and for $k=12$ and $g\geq 12$.
\end{rema}

%

Our second complement concerns
the concrete construction via theta series of the
four weight $13$ Siegel modular eigenforms ${\rm F}_g$ of
respective genus $g=8,12,16$ and $24$
given by Table \ref{tab:tableleq13scal}.
Consider again the standard Euclidean space $\R^n$ with $n \equiv 0 \bmod 8$. 
For any finite-dimensional continuous representation $U$ of the compact
orthogonal group ${\rm O}(\R^n)$ over the complex number, we denote by ${\rm
M}_U(\Omega_n)$ the complex vector space of ${\rm O}(\R^n)$-equivariant
functions $\mathcal{L}_n \rightarrow U$; this is a space of automorphic forms
for the orthogonal group scheme $\Omega_n$ introduced after Lemma
\ref{lemm:congmod4} (see also \cite[Sect. 4.4.4]{CheLan}).
For any integers $g,\nu \geq 1$, we denote by ${\rm H}_{\nu,g,n}$ 
the representation of ${\rm O}(\R^n)$ on the space of harmonic polynomials of degree $\nu$ on ${\rm M}_{n,g}(\C)$ 
in the sense of \cite[\S XI]{bocherer_theta}. The construction of Siegel theta series with harmonic
coefficients gives rise to a linear map 
(see \cite[\S XI]{bocherer_theta} and \cite[Sect. 5.4.1]{CheLan})
\begin{equation}\label{thetanugn} \vartheta_{\nu,g,n}\,:\, {\rm M}_{{\rm H}_{\nu,g,n}}(\Omega_n) \longrightarrow {\rm S}_{\frac{n}{2}+\nu}(\Gamma_g),\end{equation}
mapping any $\Omega_n$-eigenform to a Siegel eigenform (Eichler's {\it commutation relations}) or to zero. The following proposition is suggested by Rallis's theory \cite{Rallis_eichler}
and the fact that the standard parameters 
of the four weight $13$ Siegel eigenforms ${\rm F}_g$ are respectively
$\Delta_{21,13}[4] \oplus [1], \Delta_{19,7}[6] \oplus [1], \Delta_{17}[8] \oplus [9] \oplus [7] \oplus [1]\, \, \, {\rm and}\, \, \, \Delta_{11}[12]\oplus[25]$.

\begin{prop} \label{propimtheta}
\begin{itemize} 
\item[(i)] For each $g$, the form ${\rm F}_g$ is in the image $\vartheta_{1,g,24}$. \ps
\item[(ii)] The form ${\rm F}_8$ is in the image of $\vartheta_{5,8,16}$.
\end{itemize}
\end{prop}

\begin{pf} B\"ocherer\footnote{Important contributions to this problem have been made by
Siegel, Weissauer, Kudla-Rallis ({\it Siegel-Weil formula}),
and also by Freitag, Garrett, Piatetski-Shapiro, Rallis and Waldspurger.}
gives in \cite[Thm. 5]{bocherer_theta} a
necessary and sufficient condition
for these properties to hold
in terms of the order of vanishing of
the standard ${\rm L}$-function
${\rm L}(s,{\rm F}_g,{\rm St})$ of ${\rm F}_g$ at $s=n/2-g$,
with $n=24$ in case (i) and $n=16$ in case (ii).
A case-by-case analysis reveals that this criterion holds true in all five cases.
We refer to \cite{homepage} for the details of this simple, but rather tedious,
verification. The only non-trivial necessary ingredient is the non-vanishing
at $1/2$ of the Godement-Jacquet ${\rm L}$-function of $\Delta_{19,7}$ and $\Delta_{21,13}$,
that was proved in \cite[Prop. 9.3.39]{CheLan}.
\end{pf}

In the companion paper \cite{CheTaiLeech}, 
we study the maps $\vartheta_{1,g,24}$ in a much more 
elementary way. Note that a harmonic polynomial of weight $1$ on 
${\rm M}_{n,g}(\C)$ is just the datum of a $g$-multilinear alternating form on $\C^n$.
For any element $f$ in ${\rm M}_{{\rm  H}_{1,g,24}}(\Omega_{24})$, and any Niemier lattice $\Lambda$ in $\mathcal{L}_{24}$, 
the $g$-multilinear alternating form $f(\Lambda)$ is invariant under the orthogonal group ${\rm O}(\Lambda)$
of $\Lambda$. This actually forces $f$ to vanish outside the ${\rm O}(\R^{24})$-orbit of the Leech lattice by \cite[Prop. 4.1]{CheTaiLeech}. 
A curious consequence of Proposition \ref{propimtheta} is thus that for $g=8,12,16,24$, 
there is a nonzero, ${\rm O}({\rm Leech})$-invariant, $g$-multilinear alternating form on the Leech lattice! 
A computation using the ${\rm O}({\rm Leech})$-page of the $\mathbb{ATLAS}$ fortunately confirms this property, and 
reveals furthermore that there is a unique such form of to scalar, and none for the other values of $g \geq 1$.
(The existence of such a form for $g=24$ is well-known, and follows from the fact that the Leech lattice is
 {\it orientable}, which means that any element in ${\rm O}({\rm Leech})$ has determinant $1$).
In other words, ${\rm M}_{{\rm  H}_{1,g,24}}(\Omega_{24})$ 
has dimension $1$ for $g=8,12,16,24$ (and $0$ otherwise). 
The main result of \cite{CheTaiLeech} is a direct proof of the non-vanishing 
of the map $\vartheta_{1,g,24}$ for these four values of $g$.
The non-vanishing of $\vartheta_{1,24,24}$, hence of ${\rm S}_{13}(\Gamma_{24})$, 
and had already been observed in the past by Freitag, in the last section of \cite{Freitag_harm_theta}.
The Mathieu group ${\rm M}_{24}$ 
and certain oriented rank $g$ sublattices of the Leech lattice
play an important role in our argument for $g<24$.  
We also prove differently {\it loc. cit.} that the standard parameter of the line of Siegel eigenforms in the image of $\vartheta_{1,g,24}$ is 
the one given in Table 13.
All of this fully confirms Corollary \ref{corweight13} and Proposition \ref{propimtheta} (i), 
and show the following.

\begin{coro} The linear map $\vartheta_{1,g,24}$ in \eqref{thetanugn} is an isomorphism for all $g\geq 1$. \end{coro}
\ps
Case (ii) of  Proposition \ref{propimtheta} also implies 
the nonvanishing of ${\rm  M}_{{\rm H}_{5,8,16}}(\Omega_{16})$.
Let us simply mention that we actually have 
$\dim {\rm  M}_{{\rm H}_{5,8,16}}(\Omega_{16}) = 2$ 
using a computation similar to that  of \cite[Cor. 9.5.13]{CheLan}. 
The space ${\rm  M}_{{\rm H}_{5,8,16}}(\Omega_{16})$ is actually 
generated by two $\Omega_{16}$-eigenforms, with respective standard parameters $\Delta_{21,13}[4]$ and 
  $\Delta_{17}[8]$.

\subsection{Remarks on the case $g \geq 2k$}

\label{rema:kleqgsur2}
 Let $k$ and $g$ be non-negative integers satisfying $g \geq 2k$.
  In this case we have $\mathrm{L}^2_k(\Gamma_g) = \mathrm{M}_k(\Gamma_g)$ by
  \cite[Satz 3]{Weissauer}. We may thus apply Arthur's endoscopic classification to study ${\rm M}_k(\Gamma_g)$. \ps
  
   \begin{prop} \label{prop:dim2k}We have $\dim \mathrm{M}_k(\Gamma_g) = \dim
  \mathrm{M}_k(\Gamma_{2k})$ whenever $g \geq 2k$, and this dimension vanishes
  unless $k$ is divisible by $4$.
  \end{prop}

 \begin{pf} 
  For any eigenform in $\mathrm{L}^2_k(\Gamma_g)$ with
  standard parameter $\psi$, we are in case (H1) as $g-k\geq k$, 
  and with $\pi_{i_0}=1$ as $n_{i_0}=1$. In particular $k$ is
  even, $i_0$ is in $I_{\rm odd}$, and we have $\psi = \psi' \oplus [2(g-k)+1]$ where $\psi'$ is such that
  $\psi'_\R$ is an Adams-Johnson parameter for the compact group $\SO(2k)$.
  By the Arthur multiplicity formula, the characters $\epsilon_\psi$ and $\chi_{\rho_{k}(g)}$
  coincide on ${\rm C}_\psi$. Consider the element $s \in {\rm C}_\psi$ defined as the product of $\prod_{i \in I_{\rm even}} s_i$ 
  and of $s_{i_1}s_{i_2}$ in the case $I_{\rm odd}=\{i_0,i_1,i_2\}$. Formulas \eqref{chiukbasic2bis} and 
 \eqref{eq:chi_sij} imply $\chi_{\rho_{k}(g)}(s)=(-1)^{k/2}$. Indeed, one way to argue is to use the interpretation of the signs {\it loc. cit.} given right after Formula
 \eqref{chiukbasic2bis}, which shows $\chi_{\rho_{k}(g)}(s)=\prod_{i=1}^{k} \mathfrak{s}_{k-i}=(-1)^{k/2}$. On the other hand, we have $\epsilon_\psi(s)=\prod_{i \neq i_0}\epsilon(\pi_i)^{{\rm Min}(d_i,2(g-k)+1)}$.
We claim that for  $i \in I$ and any integer $g' \geq 2k$ we have the equality
\begin{equation} \label{eq:eps=1sympl} \epsilon(\pi_i)^{{\rm Min}(d_i,2(g'-k)+1)} = 1.\end{equation}
Indeed, we may assume $\pi_i$ symplectic (otherwise $\epsilon(\pi_i)=1$), in which case we have $n_i \geq 2$ and $n_i d_i \leq 2k$ and thus $d_i \leq k
  < 2(g'-k)+1$, and we conclude as $d_i$ is even. This shows in particular $\epsilon_\psi(s)=1$, and together with$\chi_{\rho_{k}(g)}(s)=(-1)^{k/2}$, proves that $k$ is divisible by $4$ if $\mathrm{M}_k(\Gamma_g)$ is nonzero.

  To prove the asserted equality of dimensions, it is enough to show that the fact that the multiplicity formula holds for $\psi$ implies that for any
  genus $g' \geq 2k$ it also holds for the parameter $\psi_{g'} := \psi' \oplus
  [2(g'-k)+1)]$, still in weight $k$ and case (H1). We may index the summands of $\psi_{g'}$ with the same set $I$ as for $\psi$, with the same $\pi_i$ for $i \in I$, the same $i_0$, and the same $d_i$ for $i \neq i_0$. There is an obvious bijection between $\mathrm{C}_{\psi}$ and
  $\mathrm{C}_{\psi_{g'}}$ matching all $s_i$ and $s_{ij}$. Via this bijection the characters $\epsilon_{\psi}$
  and $\epsilon_{\psi_{g'}}$ coincide, as for all $i \neq i_0$ we have $\epsilon(\pi_i)^{{\rm Min}(d_i,2(g-k)+1)}=\epsilon(\pi_i)^{{\rm Min}(d_i,2(g'-k)+1)}=1$ 
  by \eqref{eq:eps=1sympl}. We conclude as the characters $\chi_{\rho_{k}(g)}$ and $\chi_{\rho_{k}(g')}$ trivially coincide as well.
  \end{pf}
   \par
  Of course this proposition is coherent with the known properties of Siegel modular forms
  for $g>2k$:  \begin{equation}\label{eq:knownresults}
  \mathrm{S}_k(\Gamma_g) = 0 \text{ and } \mathrm{M}_k(\Gamma_g) =
    \begin{cases}
      \operatorname{Im} \vartheta_g & \text{ if } k \geq 0 \text{ and } k
      \equiv 0 \bmod 4, \\
      0 & \text{ otherwise}.
    \end{cases}
  \end{equation}
  by \cite{Resnikoff_sing}, \cite{Freitag_holdiff} (first equality),
  \cite{Freitag_stab} (second equality, see also \cite{Howe_lowrk}).\ps

  \begin{coro} \label{cor:siegeliso} Assume $g\geq 2k$. \begin{enumerate}
  \item The Siegel operator $\Phi_{g+1}  : \mathrm{M}_k(\Gamma_{g+1})
  \rightarrow \mathrm{M}_k(\Gamma_g)$ is bijective.\ps
  \item If $k \equiv 0 \bmod 4$, the linear map $\vartheta_g: \C[{\rm X}_{2k}] \rightarrow {\rm M}_k(\Gamma_g)$ is an isomorphism .
  \end{enumerate}
  \end{coro}
  
  \begin{pf} The first equality in \eqref{eq:knownresults} means that $\Phi_{g+1}$ is injective.
  By the equality of dimensions of Proposition \ref{prop:dim2k}, this implies that $\Phi_{g+1}$ is bijective.
  (In the case $g>2k$, the surjectivity of $\Phi_{g+1}$ also follows from the second equality in \eqref{eq:knownresults}, as $\mathrm{M}_k(\Gamma_g)$ is
  generated by theta series). This proves the first assertion.

   For the second, it is obvious that $\vartheta_g$ is injective for $g=2k$, hence for all $g \geq 2k$ as well by the relation $\Phi_{g+1} \circ \vartheta_{g+1} = \vartheta_g$.
  The surjectivity of $\vartheta_g$ follows from the second equality of \eqref{eq:knownresults} for $g>2k$. The surjectivity of $\vartheta_{2k+1}$, and the surjectivity of $\Phi_{2k+1}$ 
  proved in (1), implies the remaining surjectivity of $\vartheta_{2k}$. 
  \end{pf}
  
  This corollary seems to be new for $g=2k$.

\clearpage 
\section{Tables}

\begin{table}[htp]

\renewcommand{\arraystretch}{1.5}

{\scriptsize

\hspace{-.6cm}
\begin{tabular}{c|c|c}
$\psi$ & $g$ & $\underline{k}$ \cr
\hline
$ {\rm Sym}^2 \Delta_{11} \oplus \Delta_{11}[2]$ & $3$ & $(12, 8, 8)$ \cr
${\rm Sym}^2 \Delta_{11} \oplus \Delta_{15}[2] $ & $3$ & $(12, 10, 10)$ \cr
${\rm O}^{\rm o}_{24,16,8}$ & $3$ & $(13, 10, 7)$ \cr
\hline
$\Delta_{19, 7}[2] \oplus [1]$ & $4$ & $(11, 11, 7, 7)$ \cr
$\Delta_{21, 5}[2] \oplus [1]$ & $4$ & $(12, 12, 6, 6)$ \cr
$\Delta_{21, 9}[2] \oplus [1]$ & $4$ & $(12, 12, 8, 8)$ \cr
$\Delta_{21, 13}[2] \oplus [1]$ & $4$ & $(12, 12, 10, 10)$ \cr
${\rm O}^{\rm e}_{24,18,10,4} \oplus [1]$ & $4$ & $(13, 11, 8, 6)$ \cr
${\rm O}^{\rm e}_{24,20,14,2} \oplus [1]$ & $4$ & $(13, 12, 10, 5)$ \cr
$\Delta_{23, 7}[2] \oplus [1]$ & $4$ & $(13, 13, 7, 7)$ \cr
\hline
${\rm Sym}^2 \Delta_{11} \oplus \Delta_{15}[2] \oplus \Delta_{11}[2]$ & $5$ & $(12, 10, 10, 10, 10)$ \cr
${\rm Sym}^2 \Delta_{11} \oplus \Delta_{19, 7}[2]$ & $5$ & $(12, 12, 12, 8, 8)$ \cr
${\rm Sym}^2 \Delta_{11} \oplus \Delta_{19}[2]\oplus \Delta_{11}[2]$ & $5$ & $(12, 12, 12, 10, 10)$ \cr
${\rm O}^{\rm o}_{24,16,8} \oplus \Delta_{19}[2]$ & $5$ & $(13, 12, 12, 12, 9)$ \cr
\hline
$\Delta_{21}[2] \oplus \Delta_{11}[4]\oplus [1]$ & $6$ & $(12, 12, 10, 10, 10, 10)$ \cr
$\Delta_{21, 5}[2] \oplus \Delta_{17}[2]\oplus [1]$ & $6$ & $(12, 12, 12, 12, 8, 8)$ \cr
$\Delta_{21, 9}[2] \oplus \Delta_{17}[2]\oplus [1]$ & $6$ & $(12, 12, 12, 12, 10, 10)$ \cr
${\rm O}^{\rm e}_{24,20,14,2} \oplus \Delta_{17}[2]\oplus [1]$ & $6$ & $(13, 12, 12, 12, 12, 7)$ \cr
$\Delta_{23, 7}[2] \oplus \Delta_{17}[2]\oplus [1]$ & $6$ & $(13, 13, 12, 12, 9, 9)$ \cr
\hline
${\rm Sym}^2 \Delta_{11} \oplus \Delta_{11}[6]$ & $7$ & $(12,10,10,10,10,10,10)$ \cr
${\rm Sym}^2 \Delta_{11} \oplus \Delta_{19, 7}[2]\oplus \Delta_{11}[2]$ & $7$ & $(12, 12, 12, 10, 10, 10, 10)$ \cr
${\rm Sym}^2 \Delta_{11} \oplus \Delta_{19, 7}[2]\oplus \Delta_{15}[2]$ & $7$ & $(12, 12, 12, 12, 12, 10, 10)$ \cr
${\rm O}^{\rm o}_{24,16,8} \oplus \Delta_{21, 5}[2]$ & $7$ & $(13, 13, 13, 12, 9, 9, 9)$ \cr
\hline
$\Delta_{21, 5}[2] \oplus \Delta_{11}[4] \oplus [1]$ & $8$ & $(12, 12, 10, 10, 10, 10, 10, 10)$ \cr
$\Delta_{19, 7}[4] \oplus [1]$ & $8$ & $(12, 12, 12, 12, 10, 10, 10, 10)$ \cr
$\Delta_{21, 5}[2] \oplus \Delta_{15}[4] \oplus [1]$ & $8$ & $(12, 12, 12, 12, 12, 12, 10, 10)$ \cr
$\Delta_{23, 7}[2] \oplus \Delta_{15}[4] \oplus [1]$ & $8$ & $(13, 13, 12, 12, 12, 12, 11, 11)$ \cr
$\Delta_{21, 5}[4] \oplus [1]$ & $8$ & $(13, 13, 13, 13, 9, 9, 9, 9)$ \cr
$\Delta_{21, 9}[4] \oplus [1]$ & $8$ & $(13, 13, 13, 13, 11, 11, 11, 11)$ \cr
\end{tabular} \ps
}\ps

{\small
\caption{Standard parameters $\psi$ of the non scalar-valued cuspidal Siegel modular eigenforms
of weight $\underline{k}=(k_1,\dots,k_g)$
and genus $g$ with $k_1 \leq 13$ and $k_1 \geq k_2 \geq \dots \geq k_g>g\geq 1$.}
\label{tab:tableleq13nonscal}
}

\end{table}

\begin{table}[H]
\renewcommand{\arraystretch}{1.5}
{\scriptsize

\hspace{-.6cm}
\begin{tabular}{c|c|c}
$\psi$ & $g$ & $k$ \cr
\hline
${\rm Sym}^2 \Delta_{11}$ & $1$ & $12$ \cr
\hline
$\Delta_{17}[2] \oplus [1]$ & $2$ & $10$ \cr
$\Delta_{21}[2] \oplus [1]$ & $2$ & $12$ \cr
\hline
${\rm Sym}^2 \Delta_{11} \oplus \Delta_{19}[2]$ & $3$ & $12$ \cr
\hline
$\Delta_{11}[4] \oplus [1]$ & $4$ & $8$ \cr
$\Delta_{15}[4] \oplus [1]$ & $4$ & $10$ \cr
$\Delta_{19}[4] \oplus [1]$ & $4$ & $12$ \cr
$\Delta_{21}[2] \oplus \Delta_{17}[2]\oplus [1]$ & $4$ & $12$ \cr
\hline
${\rm Sym}^2 \Delta_{11} \oplus \Delta_{17}[4]$ & $5$ & $12$ \cr
${\rm Sym}^2 \Delta_{11} \oplus \Delta_{19}[2] \oplus \Delta_{15}[2]$ & $5$ & $12$ \cr
\hline
$\Delta_{17}[2] \oplus \Delta_{11}[4] \oplus [1]$ & $6$ & $10$ \cr
$\Delta_{17}[6] \oplus [1]$ & $6$ & $12$ \cr
$\Delta_{21}[2] \oplus \Delta_{15}[4]\oplus [1]$ & $6$ & $12$ \cr
$\Delta_{21, 13}[2] \oplus \Delta_{17}[2]\oplus [1]$ & $6$ & $12$ \cr
\hline
${\rm Sym}^2 \Delta_{11} \oplus \Delta_{15}[6]$ & $7$ & $12$ \cr
${\rm Sym}^2 \Delta_{11} \oplus \Delta_{17}[4] \oplus \Delta_{11}[2]$ & $7$ & $12$ \cr
${\rm Sym}^2 \Delta_{11} \oplus \Delta_{19}[2] \oplus \Delta_{15}[2] \oplus \Delta_{11}[2]$ & $7$ & $12$ \cr
\hline
$\Delta_{11}[8] \oplus [1]$ & $8$ & $10$ \cr
$\Delta_{15}[8] \oplus [1]$ & $8$ & $12$ \cr
$\Delta_{19}[4] \oplus \Delta_{11}[4] \oplus [1]$ & $8$ & $12$ \cr
$\Delta_{21}[2] \oplus \Delta_{17}[2] \oplus \Delta_{11}[4] \oplus [1] $ & $8$ & $12$ \cr
$\Delta_{21, 9}[2] \oplus \Delta_{15}[4] \oplus [1]$ & $8$ & $12$ \cr
$\Delta_{21, 13}[4] \oplus [1]$ & $8$ & $13$ \cr
\hline
${\rm Sym}^2 \Delta_{11} \oplus \Delta_{19}[2] \oplus \Delta_{11}[6]$ & $9$ & $12$ \cr
${\rm Sym}^2 \Delta_{11} \oplus \Delta_{19, 7}[2] \oplus \Delta_{15}[2] \oplus \Delta_{11}[2]$ & $9$ & $12$ \cr
\hline
$\Delta_{21}[2] \oplus \Delta_{11}[8] \oplus [1]$ & $10$ & $12$ \cr
$\Delta_{21, 5}[2] \oplus \Delta_{17}[2] \oplus \Delta_{11}[4] \oplus [1]$ & $10$ & $12$ \cr
\hline
${\rm Sym}^2 \Delta_{11} \oplus \Delta_{11}[10]$ & $11$ & $12$ \cr
\hline
$\Delta_{11}[12] \oplus [1]$ & $12$ & $12$\cr
$\Delta_{19, 7}[6] \oplus [1]$ & $12$ & $13$ \cr
\hline
$\Delta_{17}[8] \oplus [9] \oplus [7] \oplus [1]$ & $16$ & $13$ \cr
\hline
$[25] \oplus \Delta_{11}[12]$ & $24$ & $13$\cr
\end{tabular} \ps
}
\ps\ps
{\small
\caption{Standard parameters $\psi$ of the scalar-valued cuspidal Siegel modular eigenforms
of weight $k\leq 13$ and arbitrary genus $g\geq 1$.}
\label{tab:tableleq13scal}
}

\end{table}

\bibliographystyle{amsalpha}
\bibliography{mot23}

\end{document}